\documentclass[12pt]{article} 
\usepackage{amsmath,amssymb,amsthm,eucal,MnSymbol,amscd} 
\usepackage{hyperref}

\font\tenrm=cmr10

\font\bss=cmssdc10 at 12 pt   
\font\bigss=cmssdc10 scaled 2300

\font\cmsslll=cmss10 at 14 pt  
\renewcommand{\a}{\alpha}  
\renewcommand{\b}{\beta}

\newcommand{\f}{\varphi}  
\newcommand{\g}{\gamma}

\renewcommand{\l}{\lambda}  
\renewcommand{\o}{\omega}  
  
\renewcommand{\r}{\rho}  
\newcommand{\s}{\sigma}

\newcommand{\G}{\Gamma}

\newcommand{\bR}{\mathbb{R}}

\newcommand{\ga}{\mathfrak{a}}  
\newcommand{\gb}{\mathfrak{b}}

\renewcommand{\gg}{\mathfrak{g}}  
\newcommand{\gh}{\mathfrak{h}}  
\newcommand{\gi}{\mathfrak{i}}  
\newcommand{\gj}{\mathfrak{j}}  
  
\newcommand{\gl}{\mathfrak{l}}  
  
\newcommand{\gn}{\mathfrak{n}}  
  
\newcommand{\gp}{\mathfrak{p}}  
\newcommand{\gq}{\mathfrak{q}}  
  
\newcommand{\gs}{\mathfrak{s}}  
\newcommand{\gt}{\mathfrak{t}}  
\newcommand{\gu}{\mathfrak{u}}

\newcommand{\gsp}{\mathfrak{sp}}

\newcommand\Sp{\mathrm{Sp}}  
\newcommand\GL{\mathrm{GL}}  
\newcommand\SL{\mathrm{SL}}

\newcommand\Z{\mathrm{Z}}

\newcommand{\cC}{\mathcal{C}}

\newcommand{\ra}{\rightarrow}  
\newcommand{\tensor}{\otimes}

\newcommand{\spanof}[1]{{\langle \, {#1} \, \rangle}}
\DeclareMathOperator\End{End}

\DeclareMathOperator\ad{ad}  
  
\DeclareMathOperator\Hom{Hom}  
\DeclareMathOperator\Der{Der}  
\DeclareMathOperator{\codim}{codim}
\DeclareMathOperator{\corank}{corank}
\DeclareMathOperator\Inn{Inn}  

\DeclareMathOperator\out{out}

\newtheorem{Th}{Theorem}[section]  
\newtheorem{Prop}[Th]{Proposition}  
\newtheorem{Cor}[Th]{Corollary}  
\newtheorem{Lem}[Th]{Lemma}  

\theoremstyle{definition} % in particular, no italics 
\newtheorem{Def}[Th]{Definition}  
\newtheorem{Ex}[Th]{Example}

\theoremstyle{remark}
\newtheorem{claim}{Claim}
\newtheorem*{remark}{Remark}

\newcommand{\bt}{\begin{Th}\ \ }  
\newcommand{\et}{\end{Th}}  
\newcommand{\bp}{\begin{Prop}\ \ }  
\newcommand{\ep}{\end{Prop}}  
\newcommand{\bc}{\begin{Cor}\ \ }  
\newcommand{\ec}{\end{Cor}}  
\newcommand{\bl}{\begin{Lem}\ \ }  
\newcommand{\el}{\end{Lem}}  
\newcommand{\bd}{\begin{Def}\ \ }  
\newcommand{\ed}{\end{Def}}  
\newcommand{\bex}{\begin{Ex}\ \ }  
\newcommand{\eex}{\end{Ex}}  

\newcommand{\pf}{\begin{proof}}
\newcommand{\epf}{\end{proof}}
\newcommand{\n}{\nabla}  
  
\newcommand{\ot}{\otimes}

\newcommand{\be}{\begin{equation}}  
\newcommand{\ee}{\end{equation}}  
  
\newcommand\re[1]{(\ref{#1})}  
\newcommand{\arr}{\begin{array}{rlll}}  
\newcommand{\ea}{\end{array}}  
\newcommand{\bea}{\begin{eqnarray}}  
\newcommand{\eea}{\end{eqnarray}}  
\newcommand{\bean}{\begin{eqnarray*}}  
\newcommand{\eean}{\end{eqnarray*}}  
  
\catcode`@=11  
\@addtoreset{equation}{section}  
\catcode`@=12

\DeclareMathOperator{\im}{im}

\begin{document}  
\begin{titlepage}
%\rightline{} 
%\rightline{hep-th/yymmnnn}  

\vskip 1.5 true cm  
\begin{center}  
{\bigss  Symplectic Lie groups I-III \\ \vspace{1ex}
\bss Symplectic Reduction, Lagrangian extensions, and 
existence of Lagrangian normal subgroups
%\\[.5em] 
}

\vskip 1.0 true cm   
{\cmsslll  O.\ Baues and  V.\ Cort\'es} \\[3pt] 
{\tenrm   Institut f\"ur Algebra und Geometrie\\ 
Karlsruher Institut f\"ur Technologie, 
D-76128 Karlsruhe, Germany\\
oliver.baues@kit.edu}\\[1em]  
{\tenrm   Department Mathematik 
und Zentrum f\"ur Mathematische Physik\\ 
Universit\"at Hamburg, 
Bundesstra{\ss}e 55, 
D-20146 Hamburg, Germany\\  
cortes@math.uni-hamburg.de}\\[1em]   
%October 28, 2008 

\vspace{2ex}
% \rightline
% {Draft: v30, \today}  
{October 9, 2013}

\end{center}

\vskip 1.0 true cm  
%%%%%%%%%%%%%%%%%%%%%%%%%%%%%%%%%%%%%%%%%%%%%%%%%%%%%%%%  
\baselineskip=18pt  
\begin{abstract}  
\noindent  
We develop the structure theory of symplectic Lie groups based on the study of their isotropic normal subgroups. The article consists of three main parts.
In the first part we show that every symplectic Lie group admits a sequence of subsequent symplectic reductions to a unique irreducible symplectic Lie group. 
% which is independent of the particular choice of reduction sequence. 
 The second part concerns the symplectic geometry of cotangent symplectic Lie groups and the theory of Lagrangian extensions of flat Lie groups.
% 
%In the next part we describe the symplectic geometry of cotangent symplectic  Lie groups in conjuction with the theory of Lagrangian extensions of flat Lie groups. We show that Lagrangian extensions of flat Lie groups are classified on the Lie algebra level by a suitable cohomology group with coefficients in the dual module associated to the flat connection.  
% 
In the third part of the article we analyze the existence problem for Lagrangian normal subgroups in nilpotent symplectic Lie groups. 
%We construct examples of nilpotent symplectic Lie groups which do not admit Lagrangian subgroups, and thereby show that starting in dimension eight such subgroups do not always exist. 
%To the contrary we show that Lagrangian normal
%subgroups always exists for filiform or two-step nilpotent  symplectic Lie groups. 
\end{abstract}

\end{titlepage}

\setcounter{tocdepth}{2} % only sections and subsections 
\tableofcontents 

\newpage 

\section{Introduction}
A symplectic homogeneous manifold $(M=G/H,\o )$ is a symplectic manifold $(M,\o)$ which admits a transitive Lie group $G$ of symplectic automorphisms. Symplectic homogeneous manifolds have been studied by many authors, see for instance 
\cite{ZB} for a classification of compact symplectic homogeneous 
manifolds. Every symplectic homogeneous manifold of a given 
connected Lie group $G$ can be described as 
follows: %  \cite{Chu}:  
Let $\o \in Z^2(\gg )$ be a two-cocycle on the Lie 
algebra $\gg = Lie\, G$. Then $\gh= \ker \o \subset \gg$
is a Lie subalgebra, which is contained in the stabilizer 
$\gg_{\o}$ of $\omega$ under the coadjoint action of $\gg$ on forms. Let $G_\o$ be the stabilizer of $\omega$ under the coadjoint action of $G$ and assume that there exists a closed subgroup $H\subset G_\o$  with Lie algebra 
$\gh$. Then  $M=G/H$ is a smooth
manifold and $\o$ (considered as a two-form on $\gg/\gh \cong T_oM$, where $o=eH$ is the canonical base point) 
uniquely extends by the $G$-action to an invariant 
symplectic form $\o$ on $M$. 
Notice that if $\gg$ is perfect, that is,  if 
$[\gg ,\gg ] = \gg$, then the Lie subgroup $H\subset G$ generated 
by $\gh$ is closed. In fact, in that case $H=(G_\o)_0$. This 
happens, in particular, for semisimple Lie groups $G$,
where all symplectic homogeneous  manifolds $(M=G/H,\o )$ are 
coadjoints orbits in $\gg^*$, which are endowed with their canonical symplectic form (which is unique up to scale). On the other hand, comparatively  little is known about 
symplectic homogeneous manifolds of non-reductive groups $G$. 

In this article we will concentrate on % the case of
 \emph{symplectic Lie groups},  that is, on symplectic homogeneous manifolds 
with trivial stabilizer $H$. Thus a symplectic Lie group $(G,\o )$ is a Lie group $G$ endowed with a left-invariant symplectic form 
$\o$. 
The most important case is that of solvable Lie groups. 
In fact, it is known  that unimodular symplectic Lie groups 
are \emph{solvable}, see  \cite[ Thm.\ 11]{Chu}. For example, 
if  $G$ 
admits a \emph{cocompact} discrete subgroup $\Gamma$  then 
$G$ belongs to this class. In this case, the coset-manifold 
$\Gamma \, \backslash \, G$, with symplectic structure induced 
by $\omega$ is a compact symplectic solvmanifold. Symplectic manifolds of this type have been of intense interest because they provide a rich source of non-K\"ahler compact symplectic manifolds \cite{Thurston, BensonGordon,BC2}.
Another important motivation to study 
symplectic homogeneous spaces of solvable Lie groups
comes from their role in the theory of geometric quantization 
and the unitary representation theory of $G$ as is apparent  in the 
celebrated work of Kostant, Kirillov and Souriau (see, 
\cite[Lecture 9]{Weinstein2} for a discussion).

 A recurring theme in the study of symplectic Lie groups is their interaction with \emph{flat} Lie groups. A {flat Lie group} is a Lie group endowed with a left-invariant torsion-free flat connection. For example, it is well known that 
every symplectic Lie group $(G,\o)$ 
carries  a torsion-free flat connection $\nabla^\o$ on $G$, which
comes associated with the symplectic structure \cite[Theorem 6]{Chu}.
Furthermore, flat Lie groups arise as quotients
% in the process of symplectic reduction 
with respect to Lagrangian normal subgroups of symplectic Lie groups \cite{Bo},  and, more generally, in the context of reduction with respect to isotropic normal subgroups satisfying certain extra assumptions  \cite{DM1}. Another noteworthy appearance of flat Lie groups is as Lagrangian subgroups of symplectic 
Lie groups with the induced Weinstein connection.
%  \cite{Weinstein1}. 

These facts, which arise in the
relatively rigid and restricted category of symplectic Lie groups,
% stem from
mirror more general constructions for arbitrary
symplectic manifolds, in particular from 
% , which stem from  
the theory of Lagrangian and isotropic foliations 
\cite{Dazord, Vaisman, Weinstein1, Weinstein2}.  
%where
%the associated flat connections usually 
%arise from the canonical flat connection in the normal bundle associated to a foliation. 
In fact, every \emph{Lagrangian subgroup}  $L$ of a symplectic Lie group
$(G, \o)$ defines a left-invariant Lagrangian foliation on $(G, \o)$. 
Therefore, a Lagrangian subgroup gives a left-invariant
\emph{polarization} in the sense of geometric quantization \cite[Def.\  4.5.1]{Woodhouse}. 
 An important role is played by  the 
method of  \emph{symplectic reduction} \cite{Weinstein2} adapted to the setting of symplectic Lie groups. Every normal isotropic subgroup of $(G, \o)$  determines a coisotropic subgroup  whose symplectic reduction is a symplectic Lie group. 
As in the general theory, foundational questions in the context of symplectic Lie groups are the \emph{existence problem} for Lagrangian and isotropic subgroups 
(in particular, for \emph{normal} subgroups of this type), %   in symplectic Lie groups, 
as well as associated \emph{classification} (that is, ``normal form'') problems for symplectic Lie groups with Lagrangian subgroups.
%in relation to the geometry of associated flat quotients and the properties of the associated reduced symplectic Lie groups. 

There are several important contributions to the structure theory \cite{DM1, DM2, MR1, O2, E1} and considerable classification work for symplectic Lie groups in low dimensions \cite{Ovando,KGM,GJK,Macri,Salamon,E2,CV,O3}. However, the general picture concerning the above questions seems far from complete, 
with some of the main conjectures or research hypotheses remaining unverified in the literature. 
This concerns, in particular, the existence question for Lagrangian \emph{normal} subgroups in completely solvable or nilpotent groups (first raised in \cite{Bo}), % a general existence criterion for 
the existence of Lagrangian subgroups in arbitrary symplectic Lie groups (with partial results given in \cite{DM1,DM2,EO}), and the % general 
theory of reduction %  in symplectic Lie groups 
with respect to general isotropic normal subgroups (where it remains to extend the approach pursued  in \cite{DM1,MR1}), as well as various related or more specific questions (for example, in the context of nilpotent symplectic Lie groups  as in  \cite{Guan} or \cite{GJK,Million,DottiTirao}). 

A natural class of symplectic Lie groups arises from  
Lie groups for which the 
coadjoint representation has an open orbit.  The corresponding 
Lie algebras are called \emph{Frobenius Lie algebras} 
\cite{O1}. There are many constructions and classification results for Frobenius Lie algebras \cite{O2,E1,E2,EO,CV,O3}. 
In the Frobenius case the canonical symplectic structure on the open coadjoint orbit induces 
a left-invariant symplectic structure on the Lie group, which is the differential of
a left-invariant one-form. Therefore, such a Lie group is never unimodular and, in particular, never nilpotent. 

The principal aim of our article is to further develop the structure theory of symplectic Lie groups and to shed some more light on the aforementioned circle of ideas. We describe now the contents of the article and its main results. Let us first remark though, that, for simply connected $G$, the study of symplectic Lie  groups reduces to the study of  \emph{symplectic Lie algebras} $(\gg, \omega)$, that is, Lie algebras $\gg$ endowed with a non-degenerate 
two-cocycle $\o \in Z^2(\gg)$. Therefore, most of our results (with the noteworthy exception of Section \ref{sect:cotangent_groups}) will be derived and formulated in the Lie algebra setting. 
The article is divided into three main parts which we will discuss subsequently now.  

\subsection*{Complete reduction of symplectic Lie groups}
In the first part of the article we study the concept of \emph{symplectic reduction} of symplectic Lie groups with respect  to their isotropic normal subgroups. %  from a general point of view. 
In the basic construction we can perform symplectic reduction of any coisotropic subgroup of $(G,\o)$  to obtain a reduced symplectic manifold of \emph{lower} dimension than the dimension of $G$.  If the coisotropic subgroup arises as a symplectic orthogonal of a normal isotropic subgroup $J$, the reduced symplectic manifold carries the induced structure of a symplectic Lie group. Strictly speaking, the reduced manifold is possibly defined only locally, but in the context of symplectic Lie groups the construction always determines a unique simply connected symplectic Lie group. 
This symplectic Lie group will be called the \emph{symplectic reduction of $(G, \o)$} with respect to the isotropic normal subgroup $J$.  
On the Lie algebra level symplectic reduction is well defined and corresponds to studying the reduction of  symplectic Lie algebras with respect to isotropic ideals. 

We will derive fundamental properties of the procedure of symplectic reduction in Section \ref{sect:basic_concepts}.  In particular, we describe the role of flat Lie groups in symplectic reduction, a theme which will be important, in particular, in the second part of the article. Another important idea which will be studied in Section \ref{sect:basic_concepts} is the concept of \emph{symplectic oxidation}, that is, in principle, the construction of new symplectic groups, which are reducing to a given
symplectic Lie group $(G, \o)$, by attaching suitable 
extra data to $(G, \o)$.    

\paragraph{Reduction sequences and irreducible symplectic Lie groups}
If a symplectic Lie group does not admit any proper isotropic normal subgroup, it will be called an \emph{irreducible} symplectic Lie group. Since the dimension decreases in every possible reduction step, every symplectic Lie group admits a finite sequence of reductions to an irreducible symplectic Lie group. 
Every irreducible symplectic Lie group which terminates a reduction sequence arising from the symplectic Lie group $(G, \o)$ will be called a \emph{symplectic base} of $(G, \o)$. 
If it admits a \emph{reduction sequence} to the trivial group, that is, if the trivial group is a symplectic base, $(G, \o)$ will be called \emph{completely reducible}. We will show in Section \ref{sect:cr_algebras} that every completely solvable symplectic Lie group belongs to this class.\footnote{The nilpotent case is also a consequence of \cite[Th\'eor\`eme 2.5]{MR1}.} 
However, the property of complete reducibility is not confined to solvable symplectic Lie groups;  indeed, the affine Lie group with its canonical symplectic structure \cite{Agaoka} is completely reducible, see Example \ref{ex:affine}. Moreover, \emph{Lagrangian extensions} of flat Lie groups, which will be studied in the second part of this article, are always completely reducible. Therefore, \emph{non-solvable} completely reducible symplectic Lie groups can also be constructed by Lagrangian extension of flat Lie groups which are reductive\footnote{For examples of such, see \cite{Ba_LS}.}. 

Completely reducible symplectic Lie groups have particular tractable properties because they are susceptible to inductive arguments based on symplectic reduction. For example, 
as we deduce in Corollary \ref{cor:Lagr_subalgs2}, every completely reducible symplectic Lie group has a Lagrangian subgroup. 
%As a consequence, every completely reducible symplectic Lie group admits a left-invariant
%polarization (in the sense of geometric quantization \cite[Def.\  4.5.1]{Woodhouse}). 
%

Naturally, the \emph{classification}  problem for \emph{irreducible}  symplectic Lie groups arises as an important question, which we discuss now. Building on a previous related investigation in \cite{DM2}, we will  show in Section \ref{sect:irreducible} that the structure of irreducible symplectic Lie groups is surprisingly simple. Indeed every irreducible symplectic Lie group is metabelian and has the additional structure of a flat K\"ahler Lie group\footnote{See \cite{LM,DM2} for this notion.}. In particular, all irreducible symplectic Lie groups are solvable of very restricted type. The complete local classification of irreducible symplectic Lie groups is described in Corollary \ref{cor:irreducible}. As the classification shows, the first non-trivial example of an irreducible symplectic Lie group arises in dimension six. 

\paragraph{Uniqueness of the symplectic base}
Our main result on symplectic reductions is a Jordan-H\"older type  \emph{uniquess theorem} for the symplectic base. Namely, as we will show in Theorem 
\ref{thm:base_unique},  the isomorphism class of the simply connected symplectic base of $(G, \o)$ is independent of the chosen reduction sequence. 
Therefore, the symplectic base is uniquely defined and it is an invariant of $(G, \o)$. A derived numerical invariant is the \emph{symplectic length} of $(G,\o)$, which denotes the minimal length of a reduction sequence to the irreducible base. 

\paragraph{Existence of Lagrangian subgroups}
Certain properties of $(G, \o)$ may be deduced from the properties of its symplectic base using induction over a reduction sequence. Using this approach, we prove in Proposition \ref{prop:Lagr_subalgs1} that a symplectic Lie group $(G, \o)$ admits a Lagrangian subgroup if and only if its symplectic base admits a Lagrangian subgroup. The latter result, in a sense, gives a complete solution to the existence problem for Lagrangian subgroups by effectively constricting it to the irreducible case. 

Remarkably, this immediately implies that completely reducible symplectic Lie groups do always admit a Lagrangian subgroup.  In particular, completely solvable symplectic Lie groups have Lagrangian subgroups. 
Moreover,  Lemma \ref{lem:Lag_irr} implies that in dimension less than or equal to six every symplectic Lie group \emph{has} a Lagrangian subgroup.\footnote{Contradicting \cite[Th\'eor\`eme 3.8]{DM2}, see Section \ref{sect:irr_Lag} for further remarks.} We next employ the structure theory for irreducible symplectic Lie groups as developed in Section \ref{sect:irreducible} to construct an eight-dimensional symplectic Lie group which does \emph{not}  have a Lagrangian subgroup, see Proposition \ref{prop:irr_noLag}.  
To our knowledge, 
this is the first example 
which shows that there do exist symplectic Lie groups without Lagrangian subgroups at all.

\subsection*{Cotangent Lie groups and Lagrangian extensions}
 An important class of examples of completely reducible symplectic Lie groups are those which arise as 
cotangent bundles $(T^*H,\Omega)$ of \emph{flat Lie groups} $(H,\nabla)$, that is,  of Lie groups
$H$ endowed with a left-invariant flat torsion-free connection $\n$. 
Here $\Omega$ denotes  the standard symplectic structure on the cotangent bundle and the 
Lie group structure on $T^*H$ 
is given by identification of $T^*H$ with a semi-direct product  group $G=H\ltimes_\varrho \gh^*$, where $\gh = \mathrm{Lie}\, H$ is the Lie algebra of $H$ and $\varrho$ is a representation of $H$ on the dual vector space $\gh^*$. If $\Omega$ is left-invariant 
with respect to this Lie group structure, $(G, \Omega)$ is a symplectic Lie group, which will be called a \emph{cotangent symplectic Lie group}. 

Generalizing partial results obtained in \cite{Chu} and \cite{Lic}, we prove in Section \ref{sect:cotangentLGs} % Proposition \ref{prop:omegali} 
that $(G, \Omega)$ is a symplectic Lie group  if and only if 
the dual of the group representation $\varrho: H \ra \mathrm{GL}(\gh^*)$  induces a Lie algebra 
representation $\n:  \gh \ra \mathfrak{gl}(\gh )$, which is 
defined by a left-invariant flat torsion-free connection $\nabla$ on  $H$. This fact exemplifies  the appearance of flat Lie groups in the construction of symplectic Lie groups.
%  in a particular transparent  manner. 
For example, as observed  in \cite[Thm.\ 4.1 and Prop.\ 4.2]{O2} (see also \cite[Thm.\ 3.5.1]{Bo}), there is a one-to-one correspondence between
Frobenius  Lie algebras of dimension $2n$ splitting over an abelian ideal of dimension $n$ 
and $n$-dimensional linear Lie algebra representations which are \'etale. The corresponding symplectic Lie groups
are examples of cotangent symplectic Lie groups and the flat torsion-free connection mentioned 
above is induced by the \'etale representation. 

\paragraph{Lagrangian subgroups and flat Lie groups}
In a cotangent symplectic Lie group $(G, \Omega)$, $H$ becomes a Lagrangian subgroup and  $\gh^*$ is a normal Lagrangian subgroup.  As will be derived in Section \ref{sect:Lag_subgps}, % We prove in Proposition \ref{prop:} that 
 the flat connection $\nabla$ associated to the representation $\varrho$ coincides with the Weinstein connection induced on the Lagrangian subgroup $H$ of $(G, \Omega)$.
%  whereas the abelian Lagrangian normal subgroup $ \gh^*$ carries the canonical Riemannian left-invariant flat connection. 
The construction thus allows to transport a theorem of Weinstein \cite{Weinstein1} on leafs of Lagrangian foliations to the category of symplectic Lie groups, showing that every flat Lie group whose connection has trivial \emph{linear holonomy} group (for example, every simply connected flat Lie group) can be realized as a Lagrangian subgroup of a symplectic Lie group. This will be proved in Theorem \ref{thm:WeinsteinLG}. 

\paragraph{Lagrangian extensions of flat Lie groups}
More generally, given a flat Lie group $(H, \nabla)$, we can deform the associated group structure on $T^{*} H$ (and possibly also the symplectic structure) to construct symplectic Lie groups
 $(T^* H = G, \omega)$, which have the property  that $\gh^*$ is a normal Lagrangian subgroup of $G$.  In particular, $\gh^*$ is a totally geodesic subgroup of $(T^*H, \nabla^\omega)$.
The corresponding symplectic groups are thus characterized by the requirement that  the natural maps 
 $$  \gh^* \ra (T^*H, \nabla^\omega)  \ra (H, \nabla)$$
are connection preserving. We call a symplectic Lie group $(T^* H = G, \omega)$ of this type 
a \emph{Lagrangian extension} of the flat Lie group $(H, \nabla)$. 

In Section \ref{sect:Lagrange_ext} we develop the theory of Lagrangian extensions on the Lie algebra level. 
Our principal result, Theorem \ref{thm:Lagrangian_corresp}, shows that the isomorphism classes of Lagrangian extensions of a flat Lie algebra $(\gh,\nabla)$  are classified (in a way analogous 
to the classical extension theory of Lie algebras) by a suitable two-dimensional cohomology group $H^{2}_{L,\nabla}(\gh, \gh^{*})$,
which we call the \emph{Lagrangian extension group} for 
$(\gh,\nabla)$. 
  
In an intermediate step we show that the extension cocycles representing elements of the group $H^{2}_{L,\nabla}(\gh, \gh^{*})$ classify \emph{strongly polarized} symplectic Lie algebras. A strongly polarized symplectic Lie algebra % $(\gg, \o, \ga, V)$  
is a symplectic Lie algebra with a fixed Lagrangian ideal and a complementary Lagrangian subspace. The reader should consult Section \ref{sect:Lfunctoriality} for the precise definition of the relevant concepts and the corresponding results.  

The classification of Lagrangian extensions sets up a one-to-one correspondence of symplectic Lie groups with Lagrangian normal subgroups on the one hand,  and, on the other hand, flat Lie groups together with extension classes in their  Lagrangian extension cohomology group, see Corollary \ref{cor:Lagrangian_corresp}. 
Concerning the computation of the extension group $H^{2}_{L,\nabla}(\gh, \gh^{*})$, we point out that this group has a natural homomorphism to the ordinary extension group $H^{2}_{\nabla}(\gh, \gh^{*})$. However, this homomorphism is (as we discuss in Section \ref{sect:Lcohomo_comp}) not necessarily an injective\footnote{This shows that the statement of \cite[Theorem 3.2.1]{Bo} is not correct.} nor a surjective map.

Finally, it is an important observation that every simply connected symplectic Lie group which has a Lagrangian normal subgroup   
is indeed a Lagrangian extension of a flat Lie group \cite{Bo}. 
It also has been claimed\footnote{See 
\cite[top of p.\ 1164, preceding Theorem 3.5.2]{Bo}.} that \emph{all} simply connected completely solvable 
symplectic Lie groups arise as Lagrangian extensions of flat Lie groups. If this assertion were true,  it would essentially 
reduce the study of a huge class of symplectic Lie groups to 
that of flat Lie groups. The claim relies on the 
equivalent statement 
%This statement is not true and neither  is Lemma 3.4.2.\ on p.\ 1164 of \cite{Bo}, on which its proof is based. 
that every completely solvable symplectic Lie algebra
admits a Lagrangian ideal. It follows from the work in
\cite{Ovando} that every four-dimensional completely solvable 
symplectic Lie algebra has a Lagrangian ideal; on the other hand,  there exists a four-dimensional solvable Lie algebra of imaginary type which does not have a Lagrangian ideal. The existence question for Lagrangian ideals is our next main topic. 
%
% we may 
%consider symplectic Lie groups $(T^*H = G, \omega)$ such 
%that $ \gh^*$ is a Lagangian normal subgroup of $G$ such that
%the bundle projection $T^*H \ra H$ natural projection homomorphisms cps such that 
% $T^*H$
%by a two-cocycle on $\gh$ with values in $\gh^*$, with respect to 
%the $\gh$-module structure on the vector space $\gh^*$ associated with $\rho$.  

\subsection*{Existence of Lagrangian normal subgroups}
%%%%%%%%%%%%%%%%%%%%%%%%%%%%%%%%%%%%%%%%%%
%%%%%%%%%%%%%%%%%%%%%%%%%%%%%%%%%%%%%%%%%%
In the third part of the article we study the existence of Lagrangian ideals in completely solvable and nilpotent Lie algebras. We are 
particularly interested in the nilpotent case, since all nilpotent symplectic Lie algebras admit reduction with respect to central isotropic ideals (see \cite[Th\'eor\`eme 2.5]{MR1} or Proposition \ref{prop:nilpotent_reduction}). It follows that nilpotent symplectic Lie algebras may be constructed inductively by subsequent \emph{symplectic oxidation}
% \footnote{named ``double extension''  in \cite{MR1}}
starting from the trivial symplectic Lie algebra. (We explain the oxidation 
procedure in Section \ref{sect:sympl_ox} in detail.)
Therefore,  nilpotent symplectic Lie algebras can be 
expected to have particular accessible properties. 

Before describing some of our results concerning Lagrangian ideals, we would
like to mention the closely related notion of  commutative polarization, which  in 
the context Frobenius Lie algebras is the same as an abelian Lagrangian subalgebra with
respect to some exact symplectic form.  As is proved  in [5, Thm. 4.1],  for solvable Frobenius 
Lie algebras over algebraically closed fields of characteristic 0, the existence of a commutative polarization implies 
the existence of a commutative polarization, which is also an ideal.  In particular, under these
assumptions an exact symplectic Lie algebra which has an abelian Lagrangian subalgebra
also admits a Lagrangian ideal. Unfortunately, for general symplectic Lie algebras, the existence of 
an abelian Lagrangian subalgebra bears no direct implication for the existence of Lagrangian
ideals, see Remark on page \pageref{RemarkCP}.

\paragraph{Counterexamples} 
As our starting point we construct two examples of symplectic
Lie algebras in dimension six, and dimension eight,  respectively,
which show that Lagrangian ideals in completely solvable, even 
in nilpotent symplectic Lie algebras do not necessarily exist. 
Our first example, Example \ref{ex:noLag_cs6},  is a completely solvable non-nilpotent
symplectic Lie algebra of dimension six whose maximal isotropic ideal 
has dimension two. Our next example, Example \ref{ex:noLag_n8}, is a four-step nilpotent 
symplectic Lie algebra of dimension eight, whose maximal isotropic ideal 
has dimension three. These examples show (as we believe, for
the first time) that, in general, completely solvable and nilpotent 
symplectic Lie groups do not admit Lagrangian normal subgroups.
The theory of Lagrangian extensions, as developed in the second part of this article, can only describe the proper subclass of completely solvable symplectic Lie groups which do admit 
Lagrangian normal subgroups. We will restrict our further considerations to the nilpotent case and analyze this
situation more closely. 
Our findings suggest that nilpotent symplectic Lie groups are very far from being understood, even if flat Lie groups were well understood. 

\paragraph{Symplectic rank and symplectic length}
The fact that Lagrangian ideals do not always exist leads us to define the following invariant of a symplectic Lie algebra $(\gg ,\o )$. The \emph{symplectic rank} 
$\s (\gg ,\o )$ is the maximal dimension of an isotropic ideal. 
In particular, $\s (\gg ,\o )$ is maximal if $\s(\gg ,\o ) = { 1 \over 2} \dim \gg$ and $(\gg, \o)$ has a Lagrangian ideal.  Observe also that   
the existence of a Lagrangian ideal in $(\gg ,\o )$ 
is equivalent to the fact that 
$(\gg, \o)$ can be reduced to the trivial algebra in one step, that is,
this holds if and only if the \emph{symplectic length} satisfies $sl(\gg, \o) = 1$. 
From this point of view, the above examples show that symplectic rank and length are important invariants in the context of nilpotent symplectic Lie groups. 

\paragraph{Existence results and a further counterexample}  
We would like to understand the class of 
nilpotent symplectic Lie algebras which admit a Lagrangian ideal 
as good as possible and give criteria for the existence 
of Lagrangian ideals. We  describe now our main results in this direction.

As a basic observation we show that every two-step nilpotent symplectic Lie algebra admits a Lagrangian ideal. 
This is proved in Theorem \ref{thm:2step}. Since Example \ref{ex:noLag_n8} is four-step nilpotent, one could 
hope that existence of Lagrangian ideals is satisfied also for three-step nilpotent Lie algebras. However, in Section \ref{3-stepcounterEx} we construct a three-step nilpotent Lie algebra of dimension ten which has symplectic rank four. Thus this symplectic Lie algebra has no Lagrangian ideal.

To round up the theory of low-dimensional nilpotent symplectic Lie algebras, we prove next that the 
dimensions in which the previous counterexamples occur are sharp. By Theorem \ref{thm:dimlt8}, every nilpotent symplectic Lie algebra  of dimension less than eight admits a Lagrangian ideal. Moreover, as is proved in Theorem \ref{thm:red38}, every three-step nilpotent Lie algebra  of dimension less than ten admits a Lagrangian ideal.

Recall that a nilpotent Lie algebra is called \emph{filiform} if its nilpotency class is maximal with respect to its dimension, and that such Lie algebras are generic in the variety of nilpotent Lie algebras. 
We show in Theorem \ref{thm:ff_Lagrangian} that every  filiform nilpotent symplectic Lie algebra has a Lagrangian ideal. Since this
ideal is also unique, the Lagrangian extension theory for such algebras takes a particular simple form and it leads directly 
to a classification theory of filiform nilpotent symplectic Lie algebras,
which can be based on a classification of filiform flat symplectic Lie algebras, see Corollary \ref{cor:filiform_corresp}.  Somewhat strikingly, the latter results also show that \emph{the existence of Lagrangian ideals is a generic property for nilpotent symplectic Lie algebras}. 

\paragraph{Reduction theory of nilpotent symplectic Lie algebras}
Symplectic reduction is one of the main
methods in our study of nilpotent symplectic Lie algebras. The basic properties of reduction,
which are relevant to this  part of our article are developed in Sections \ref{sect:reduction} - \ref{NormalSubsect}:  
The symplectic reduction of a symplectic Lie algebra $(\gg , \o )$ with respect to an isotropic ideal $\gj$ 
is a symplectic Lie algebra  
$(\bar\gg,\bar\o)$ of dimension $\dim \gg -2\dim \gj$. 
Depending on the properties of $\gj$ it comes with additional
algebraic structures. In the important case where 
$\gj^\perp\subset \gg$ is an ideal, the additional 
structure is essentially encoded in 
a Lie subalgebra $\gq \subset \Der (\bar\gg)$, the elements
of which are subject to a system of quadratic equations which
depend on $\bar \o$. The condition on $\gj^\perp\subset \gg$ 
is satisfied, in particular, 
when reducing with respect to a central isotropic ideal. 
This situation thus plays a mayor role in the nilpotent case. 

We initiate the study of  nilpotent endomorphism algebras 
$\gq$ of the above type in Section \ref{sect:qalgebras}, which
is at the heart of our results in this part of the article.    
As a particular application of this theory, we prove, for instance,
that any nilpotent symplectic  Lie algebra $(\gg ,\o )$, which 
has an abelian reduction $\bar\gg$ with respect to a one-dimensional central ideal, has a Lagrangian ideal, see Theorem \ref{thm:onedimabelian_reduction}. The analogous fact for reduction with respect to higher-dimensional central ideals fails. This is a basic observation for the construction of our ten-dimensional three-step nilpotent Example in Section \ref{3-stepcounterEx} without Lagrangian ideal.  
Furthermore,  the proofs
of Theorem \ref{thm:dimlt8} and Theorem \ref{thm:red38} 
depend crucially on the results in Section \ref{sect:qalgebras}.   

\paragraph{Further considerations and examples}
We would like to conclude the introduction to the third part of
our article with some additional remarks. 

The first remark concerns the notion of symplectic rank for nilpotent symplectic Lie algebras. The importance of the invariant $\s (\gg ,\o )$  is underlined by the fact that every isotropic ideal is abelian, 
see Lemma \ref{lem:normal_red}.   
Hence,  $\s (\gg ,\o )$ is bounded by the maximal dimension $\mu (\gg )$ of an 
abelian ideal in $\gg$:
\[ \s (\gg ,\o ) \le \mu (\gg ). \]  
Obviously, $\mu (\gg )$ is in turn bounded by the maximal dimension
$\nu (\gg)$ of an abelian subalgebra:
\[ \mu \le \nu .\] 
An invariant related to  
$\nu$  has been studied in detail  
by Milenteva \cite{Mil,Mil2,Mil3}. 

Since every two-step nilpotent symplectic Lie algebra $(\gg ,\o )$ 
admits a Lagrangian ideal, these Lie algebras satisify % and, hence,  
\be \mu (\gg )\ge \frac{1}{2}\dim \gg \label{muEqu} .\ee
Therefore, two-step nilpotent Lie algebras violating \re{muEqu} 
do not admit any symplectic structure. This is already a strong restriction. For instance, the existence of
infinitely many two-step nilpotent Lie algebras 
violating 
\be \nu (\gg )\ge \frac{1}{2}\dim \gg \label{nuEqu}\ee
and, hence, \re{muEqu} 
follows from  \cite[Lemma 11]{Mil3}. 

Note, however, that Lagrangian subalgebras of $(\gg, \o)$ 
do not  lead to further information on the invariant $\nu (\gg)$.
We explained before  that every  
nilpotent symplectic Lie algebra 
$(\gg ,\o )$ has a Lagrangian 
subalgebra, see Corollary \ref{cor:Lagr_subalgs2}. 
However, contrary to Lagrangian
ideals, Lagrangian subalgebras are not necessarily abelian. 
In fact, it follows from the results in Section \ref{sect:Lag_subgps}, 
that every nilpotent flat Lie algebra whose flat connection is \emph{complete} arises as a Lagrangian subalgebra of a nilpotent symplectic Lie algebra. 

Finally, we would like to point out that based on our examples and the theory developed in this work we already disproved several claims and conjectures, which play a role in the literature. Another application of
this kind concerns the following question: It was asked in \cite{Guan}, whether, given any compact symplectic solvmanifold which may be presented as a coset-space $\Gamma \backslash G$, where $G$ is a solvable Lie Group and $\Gamma \leq G$ is a lattice, $G$ can be at most three-step solvable. This assertion has  been verified in the literature mostly in low dimensions. For example, it holds in dimension less than or equal to six by \cite{Macri, Ovando}. 
We answer this question in the negative by constructing in Example \ref{ex:uppertriangquot} a family of symplectic nilmanifolds, which % , quite to the contrary, 
has unbounded solvability degree. However, the lowest dimensional example of a symplectic solvmanifold with solvability degree greater than three, which we can construct, satisfies $\dim \Gamma \backslash G = 72$. In general,
this series of examples implies that, quite to the contrary of the original conjecture,  \emph{the solvability degree of symplectic solvmanifolds is unbounded} with increasing dimension.

%%%%%%%%%%%%%%%%%%%%%%%%%%%%%%%%%%%%%%%%%%%

% \tableofcontents

%%%%%%%%%%%%%%%%%%%%%%%%%%%%%%%%%%%%%%%%%%%
%%%%%%%%%%%%%%%%%%%%%%%%%%%%%%%%%%%%%%%%%%%
\subsection*{Notation and conventions} 

\paragraph{Ground fields.}
For our results on symplectic Lie algebras, if not mentioned 
otherwise, we work over a fixed unspecified ground field $k$ 
of characteristic $0$. Global geometric interpretations
for simply connected Lie groups may, of course, be given over the
field $k = \bR$ of real numbers, which therefore is of principal interest 
for  our investigations.  
For a vector space $V$ and subset $A \subseteq V$, we let $\spanof{A}$ 
denote the span of $A$ (over $k$).   
\paragraph{Lie algebras.}
Let $\gg = (V, [ \, , \,  ])$ be a Lie algebra.
Here $V$ is the underlying vector space of $\gg$ 
and $ [ \, , \,  ]$ is a Lie bracket on $V$. 
For subsets $A, B \subseteq \gg$, the \emph{commutator} $[ A, B ]$ is the vector subspace which is generated by all brackets $[ a, b ]$, where 
$a \in A$, $b \in B$.

We denote with $\Der(\gg) \subset \End(V)$ the Lie algebra of
derivations of $\gg$ and with $ \Inn(\gg)$ the image of $\gg$ under the \emph{adjoint representation} $\ad: \gg \ra  \Der(\gg)$,
where $\ad(u) v = [ u, v ]$, for all $u,v \in \gg$.  
The Lie algebra $\out(\gg) = \Der(\gg)/ \Inn(\gg)$ is
called the Lie algebra of \emph{outer derivations}.  

The {\em derived series}  $D^{i} \gg$ of ideals in $\gg$ 
satisfies $D^0 \gg = \gg$, and $D^{i+1} \gg = [ D^{i} \gg,  D^{i} \gg ]$. We say that $\gg$ has \emph{derived length}, or solvability degree $s(\gg)$ if $D^{s(\gg)} = \{ 0 \}$ and $D^{s(\gg)-1} \neq \{ 0 \}$. Define a Lie algebra to be \emph{completely solvable} if it admits a complete flag of ideals. 

% \noindent 
Define the \emph{descending central series} $C^{i} \gg$ of ideals in $\gg$ inductively 
by $$  C^0 \gg = \gg , \, C^{i+1} \gg = [ \gg, C^{i} \gg ] \; . $$ Dually, the \emph{ascending central series} $C_{i}\,  \gg$, $ i \geq 0$, of ideals is defined by   $$ C_0\,  \gg = \{ 0 \} , \, \;  C_{i+1}\,  \gg = \{ v \mid [\gg, v ] \subseteq \, C_{i} \, \gg \, \} \; . $$
In particular, $C_{1}\,  \gg = \Z(\gg) = \{  v \mid [v, \gg] = 0\}$ is the \emph{center} of $\gg$. 

\noindent
We observe the relations % \marginpar{\bf !!!!}
\begin{equation} \label{eq:cs1}
[ C^{i} \gg, C^{j} \gg ]  \subseteq  C^{i+j+1} \gg \; , \\
\end{equation}  
\begin{equation} 
\label{eq:cs2}
 [ C^{i} \gg, C_{\ell}\,  \gg ]  \subseteq C_{\ell-i-1}\,  \gg \; \, . 
 \end{equation}  

The Lie algebra $\gg$ is called \emph{$k$-step nilpotent} if $C^k \gg = \{ 0 \}$, for some $k \geq 0$. 
More precisely, we say that  $\gg$ is of \emph{nilpotency class} $k$ if $C^k \gg = \{ 0 \}$ and $C^{k-1} \gg \neq \{ 0 \}$. 
If $\gg$ is of nilpotency class $k$, we have % the inclusion 
\begin{equation} 
\label{eq:cs3}
C^{k-i} \gg \, \subseteq C_{i} \, \gg \; .
 \end{equation}  
%\begin{equation} C^{k-i} \gg \subseteq C_{i} \gg \; .
%\end{equation} 

\paragraph{Low dimensional cohomology of Lie algebras.}
Let $\varrho: \gg \ra \End(W)$ be a representation of $\gg$ on a vector space $W$. The elements of ${\mathcal C}^{i}(\gg,W) = \Hom(\bigwedge^{i} \gg, W) $ % = (\bigwedge^{i} \gg^*)  \tensor W$ 
are called $i$-cochains for $\varrho$. We have the coboundary operators $$ \partial =  \partial^{i}_{\varrho}: \; {\mathcal C}^{i}(\gg,W) \ra {\mathcal C}^{i+1}(\gg,W)$$  of Lie algebra cohomology for $\varrho$ 
(see \cite{HS}),  satisfying $\partial \, \partial = 0$.  The elements of 
$$ Z^{i}_{\rho}(\gg,W) = \{ \alpha \in {\mathcal C}^{i}(\gg,W) \mid 
 \partial^{i}_{\varrho} \,  \alpha = 0 \}$$ are called $i$-cocycles and
the elements of  $$ B^{i}_{\rho}(\gg,W) = \{ \partial^{i-1}_{\varrho} \, \alpha  \mid  \alpha \in  {\mathcal C}^{i-1}(\gg,W) \}$$  are called $i$-coboundaries. We write $Z^{i}(\gg)$, respectively $B^{i}(\gg)$, for these modules if $\varrho$ is the trivial representation 
on the ground field $k$. 
We need to consider $\partial$ only in low-degrees. 
Here we have, for all $u,v, w \in \gg$, $t \in W$: 
\begin{equation}  \label{eq:cocycles}
\begin{split}
 ( \partial^{0} t)  (u)  & = \rho(u) t \; , \\
 ( \partial^{1} \lambda) (u,v) & =  \rho(u) \lambda(v) - \rho(v) \lambda(u) - \lambda ([u,v]) \; ,  \\
  ( \partial^{2} \alpha) (u,v,w) &=   \rho(u) \alpha(v,w) + \rho(v) \alpha(w,u) + \rho(w) \alpha(u,v)  \\ 
  &   + \alpha(u, [v,w]) + \alpha(w, [u,v]) + \alpha(v, [w,u]) \; .
\end{split}
\end{equation}

\paragraph{Connections on Lie algebras.}
A bilinear map $\nabla: \gg \times \gg \ra \gg$, written as
$(u,v) \mapsto \nabla_{u} v$, is 
called a \emph{connection} on $\gg$. For all $u,v,w \in \gg$, 
the \emph{torsion} 
$T = T^\nabla$ is defined by $T(u,v) = \nabla_{u} v- \nabla_{v} u -
[u, v]$ and the \emph{curvature} $R= R^\nabla$ by 
$R(u,v) w  = \nabla_{u} \nabla_{v} w - \nabla_{v} \nabla_{u} w - \nabla_{[u,v]} w$. If $T= 0$ then the connection $\nabla$ is called \emph{torsion-free}. The curvature $R$ vanishes if and only if the map $\sigma^\nabla: u \mapsto \nabla_{u} \; $, $\gg \ra \End(\gg)$ is a representation of $\gg$ on itself. In this case,  the connection $\nabla$ is called \emph{flat}. A flat connection is torsion-free if and only if the identity of $\gg$ is a one-cocyle in 
$Z^1_{\sigma^\nabla}(\gg, \gg)$.  A
\emph{flat Lie algebra} is a pair $(\gg, \nabla)$, where $\nabla$ is a torsion-free and flat connection on $\gg$. A subalgebra $\gh$ of $\gg$ is called a \emph{totally geodesic subalgebra} with respect to a connection $\nabla$ if, for all $u,v \in \gh$, $\nabla_{u} v \in \gh$. In particular,
if $(\gg, \nabla)$ is a flat Lie algebra and $\gh$ is totally geodesic, 
then $\nabla$ restricts to a connection on $\gh$, which is  torsion-free and flat.

\paragraph{Symplectic vector spaces.}
Let $\omega \in \bigwedge^2 V^*$ be a non-degenerate alternating form. The pair $(V, \omega)$ is called a \emph{symplectic
vector space}. 

Let $W \subseteq V$ be a subspace. The \emph{orthogonal} of $W$ in $(V, \omega)$ is $$W^{\perp_{\omega}} = \{ v ÿ\in V \mid \omega(v,w) = 0 \text{, for all } w\in W \} \; . $$

The subspace $W$ is called \emph{non-degenerate} if 
$W \cap W^{\perp_{\omega}} = \{ 0 \}$. 
% The subspace $W$ 
It is called \emph{isotropic} 
%if $\omega(v,w) = 0$, for all
%$v,w \in W$. That is, $W$ is isotropic 
if $W \subseteq W^{\perp_{\omega}}$. A maximal isotropic subspace is called 
\emph{Lagrangian}. A Lagrangian subspace satisfies 
$W = W^{\perp_{\omega}}$. 

A direct sum decomposition $V = V_{1} \oplus W \oplus V_{2}$ is called an \emph{isotropic decomposition} 
if  $V_{1}$, $V_{2}$ are isotropic 
subspaces and $W = V_{1}^{\perp_{\omega}}  \cap   V_{2}^{\perp_{\omega}}$.

The subspace $U$ is called \emph{coisotropic} 
%if $\omega(v,w) = 0$, for all
%$v,w \in W$. That is, $W$ is isotropic 
if $U^{\perp_{\omega}}  \subseteq U$. If $W$ is
isotropic then $W^{\perp_{\omega}}$ is coisotropic.   
For any coisotropic subspace $U$ in $(V, \omega)$, 
$\omega$ induces a symplectic form $\bar \o$ on 
$U/ U^{\perp_{\omega}}$. The symplectic vector
space $(U/ U^{\perp_{\omega}}, \bar \o)$ is called 
the \emph{symplectic reduction} with respect to the isotropic
subspace $ U^{\perp_{\omega}}$. 

For any isotropic subspace 
$W$, we define $\corank_{\o} W = {1 \over 2} \dim \, W^{\perp_{\omega}}/ \, W$. In particular, $W$ is Lagrangian if and only if 
$\corank_{\o} W = 0$.

\bl  \label{lem:corank} 
Let $U \subset (V,\o)$ be a coisotropic subspace and 
$W \subset (V,\o)$ an isotropic subspace. Then the image $\bar W$ of\/ $W \cap U$  in $(U/ U^{\perp_{\omega}}, \bar \o)$ is an isotropic subspace which satisfies 
$\corank_{\bar \o} \bar W \leq \corank_{\o} W$. 
\el

\paragraph{Symplectic Lie algebras.} A \emph{symplectic Lie algebra} is a pair $(\gg, \omega)$, where $\omega \in Z^2 (\gg)$ 
is a closed and non-degenerate two-form. The condition that $\omega$ is closed is equivalent to %the identity
\begin{equation} \label{eq:oclosed}
\omega([u,v], w) + \omega([w,u], v) + \omega([v,w], u) = 0 \; ,
\end{equation}
for all $u,v,w  \in \gg$.

\paragraph{Two-cocycles and derivations.}
Let $\alpha \in Z^2(\gg)$ be a closed two-form for the Lie algebra $\gg$
and $\varphi \in \Der(\gg)$ a derivation. Then the associated two-form 
$\alpha_{\varphi}$, which is defined by declaring 
 \begin{equation}
 \alpha_{\varphi}(v,w) = \alpha(\varphi v, w) + \alpha(v, \varphi w) \; \text{,  for all $v,w \in \ \gg$,} \label{eq:alphaphi}
\end{equation}
is a two-cocycle as well (that is, $\alpha_{\varphi} \in Z^2(\gg)$).

%%%%%%%%%%%%%%%%%%%%%%%%%%%%%%%%%%%%%%%%%%%
%  Part I , Theory of reduction
%%%%%%%%%%%%%%%%%%%%%%%%%%%%%%%%%%%%%%%%%%%

\part{Theory of reductionÚ} %  with respect to isotropic normal subgroups}

\section{Basic concepts} \label{sect:basic_concepts}
We introduce important notions in the
context of symplectic Lie algebras and study their basic 
properties. These will play a fundamental role in our approach to the subject. 
% and introduce several new concepts. 

\subsection{Isotropic ideals and symplectic rank} \label{sect:srank}
Let $(\gg, \o)$ be  a symplectic Lie algebra. An ideal
$\gj$ of $\gg$ is called an \emph{isotropic ideal} of $(\gg, \o)$ 
if $\gj$ is an isotropic subspace for $\omega$. If  the ortohogonal $\gj^{\perp_{\omega}}$ is an ideal in $\gg$ we 
call $\gj$ a \emph{normal isotropic ideal}. If $\gj$
is a maximal isotropic subspace $\gj$ is called 
a \emph{Lagrangian ideal}. \\

Using \eqref{eq:oclosed}, the following is easily verified: 

\bl \label{lem:normal_red}
Let $\gj$ be an ideal of $(\gg, \o)$.
Then \begin{enumerate}
\item If\/  $\gj$ is isotropic then $\gj$ is abelian. 
\item $\gj^{\perp_{\omega}}$ is a subalgebra of $\gg$.
\item $\gj^{\perp_{\omega}}$ is an ideal in $\gg$ if and
only if $[ \gj^{\perp_{\omega}}, \gj ] = \{ 0 \}$. 
\end{enumerate}
\el

\bd \label{def:symp_rank}
Let  $(\gg ,\o )$ be a symplectic Lie algebra. 
The \emph{symplectic rank} of $(\gg ,\o )$ is
the maximal dimension of any  isotropic ideal in $(\gg ,\o )$. 
\ed

The following dual definition is particularly useful in the context of symplectic reduction (compare Section \ref{sect:Lifting}): The \emph{symplectic corank} of $(\gg ,\o )$ is the corank of 
any isotropic ideal of maximal dimension. For example, 
 $\corank (\gg ,\o ) = 0$ if and only if $(\gg ,\o )$ has a Lagrangian ideal. 

\subsection{Symplectic reduction} \label{sect:reduction}
Let $(\gg, \o)$ be  a symplectic Lie algebra and  $\gj \subseteq \gg$ an isotropic ideal. The orthogonal 
$\gj^{\perp_{\omega}}$ 
is a subalgebra of  $\gg$ which contains $\gj$, and
therefore $\omega$ descends to a symplectic form 
$\bar \o$ on  % $\bar \gg  = \gj^{\perp_{\omega}} / \;  \gj 
the quotient Lie algebra 
$$\bar \gg = \gj^{\perp_{\omega}} / \;  \gj \; . $$

\bd
The symplectic Lie algebra $(\bar \gg, \bar \o)$ is called the \emph{symplectic
reduction} of $(\gg, \o)$ with respect to the isotropic ideal\/
$\gj$.
\ed

%If $(\gg, \o)$ reduces to $(\bar \gg, \bar \o)$, we can choose a 
%vector space decomposition $\gg = \gj^* \oplus  \gj^{\perp_{\omega}} 
%= \bar \gg \oplus$. 
%
%If $\gj$ is contained in the center of $\gj^{\perp_{\omega}}$ then the 
%subalgebra $\gj^{\perp_{\omega}}$ is an ideal in $\gg$. 

If $\gj^{\perp_{\omega}}$ is an ideal in $\gg$ we call the symplectic reduction \emph{normal}, and $\gj$ is called a \emph{normal isotropic ideal}. 
Normal reduction is of 
particular interest and we single out several important special cases:\\

The reduction of $(\gg, \o)$ with respect to $\gj$ is called  
\begin{itemize}
\item \emph{ Lagrangian reduction} if $\gj$ is a Lagrangian ideal.  
In this case, 
$\gj^{\perp_{\omega}} = \gj$ and $\bar \gg = \{0 \}$.
\item \emph{ Central reduction} if $\gj$ is central. In this case, 
$\gj^{\perp_{\omega}}$ is an ideal in $\gg$, which contains $C^1 \gg = [\gg, \gg ]$.
\item \emph{Codimension one normal reduction} if $\gj$ is one-dimensional and $[\gj ,  \gj^{\perp_{\omega}}] = \{0 \}$. 
\end{itemize}

\paragraph{Reducibility of  symplectic Lie algebras}
Not every symplectic Lie algebra may be reduced to one of
lower dimensions (compare Section \ref{sect:irreducible}). 
If $(\gg, \o)$ cannot be symplectically 
reduced, that is, if $\gg$ does not have a non-trivial,
isotropic ideal $\gj$, then $(\gg, \o)$ will be called \emph{symplectically irreducible}.
(Similarly if $\gj$ does not have a non-trivial normal or central ideal,
we speak of \emph{normal}, respectively \emph{central} irreducibility
of the symplectic Lie algebra $(\gg,\o)$.) We shall construct an irreducible
six-dimensional solvable symplectic Lie algebra in Example \ref{ex:irreducible}. 
On the other hand we have: 
% \vspace{1ex}
\begin{Ex}[Completely solvable implies reducible] \label{Ex:cs}
Let $(\gg, \o)$ be a symplectic Lie algebra with $\gg$ completely solvable. 
The definition of complete solvability implies that there exists a
one-dimensional ideal $\gj$ in $\gg$. 
Then $\gj$ is also a one-dimen\-sio\-nal isotropic ideal in  $(\gg, \o)$. 
\end{Ex}

\subsection{Normal symplectic reduction}
\label{NormalSubsect}
In this subsection we analyze in detail the process 
of \emph{normal} reduction. That is, we consider 
reduction of $(\gg, \o)$ with respect to  isotropic 
ideals $\gj$,  which have the property that 
$\gj^{\perp_{\omega}}$ is an ideal as well.\footnote{A similar analysis has been carried out in \cite{DM1}. The ``symplectic double extension'' covers in particular Lagrangian reduction and one-dimensional central reduction. However, the assumption that the sequence \eqref{eq:n3} splits (which is required for the constructions in \cite{DM1}) is too restrictive for our purposes.  
Also we use a slightly different terminology.} 
Given such $\gj$,  %  the following 
the Lie algebra extensions
\begin{eqnarray}
& 0 \ra \gj^{\perp_{\omega}} \ra  \gg   \ra \gh \ra 0 \label{eq:n1} \; \,  ,  \\
& 0 \ra \, \gj \, \ra   \gj^{\perp_{\omega}}  \ra \bar \gg \ra 0 \; \;     \label{eq:n2}\end{eqnarray} 
naturally arise in the reduction process. Here $(\bar \gg, \bar \omega)$ is the symplectic reduction and $\gh$ arises, since 
$\gj^{\perp_{\omega}}$ is an ideal.
As will be explained in the following two subsections, the quotients $(\bar \gg, \bar \omega)$ and $\gh$ carry additional structure, which is induced from the sympletic Lie algebra $(\gg, \o)$:
\begin{enumerate}
\item  The Lie algebra $\gh = \gg / \,  \gj^{\perp_{\omega}}$ admits
an induced torsion-free flat connection 
$\bar \nabla = \bar \nabla^{\omega}$ such that  the projection map
$(\gg, \nabla^{\omega}) \ra (\gh,  \bar \nabla^{\omega})$ is an
affine map of flat Lie algebras. 
% which 
%is defined by the equation
%\begin{equation} \label{eq:nablao} \bar \o (\bar \n_{\bar u} \bar v, a) \, = \, \,  -\o ( v, [u,a])
%\quad\text{, for all } \bar u, \bar v\in \gh, \; 
%a\in \gj.
\item The symplectic Lie algebra $(\bar \gg, \bar \omega)$ comes equipped with a distinguished  algebra $\hat \gq$   
of derivations of $\bar \gg$, which satisfies a
compatibility condition with respect to $\bar \o$.
%  induced by  \eqref{eq:n1}.
\end{enumerate}

%\emph{We view $\hat \gq   \subset 
%\Der(\bar \gg)$ as additional structure on $\bar \gg$ which is induced
%by the reduction process.} 

%\begin{remark}
Not every Lie algebra carries a torsion-free flat connection.
(See, for example,  \cite[\S 2]{Burde} for discussion.)
Note therefore that the existence of the torsion-free flat connection 
$\bar \nabla$  puts a strong restriction on
the possible quotients $\gh$, which may occur for 
normal reductions in  \eqref{eq:n1}. 
% \end{remark}

\subsubsection{The induced flat connection on $\gh$}
We show that normal reduction induces a torsion-free flat 
connection on its associated quotient algebra $\gh$.   
This stems from a  more general construction 
which we introduce now. 

\paragraph{Induced flat connection on normal quotients}
Let $\gj$ be an ideal of $\gg$, which satisfies  
$[\gj^{\perp_{\omega}}, \gj] = \{0 \}$. According to
Lemma \ref{lem:normal_red}, the subalgebra $\gj^{\perp_{\omega}}$ is then  an ideal of $\gg$.
We call such an ideal $\gj$ a \emph{normal ideal} of $(\gg, \o)$
and let $$\gh = \gg/ \gj^{\perp_{\omega}} \;  $$
denote the associated quotient Lie algebra. 
% Let $\gj$ be a normal isotropic ideal in $(\gg, \o)$.
From $\omega$ we obtain a \emph{non-degenerate} bilinear pairing $\o_{\gh}$ 
between $\gh$ and $\gj$, by declaring
\begin{equation}
\o_{\gh} (\bar u, a) = \o(u,a) \quad\text{, for all } \bar u \in \gh,\; 
a\in \gj  \label{eq:baro} \,  , 
\end{equation} 
where, for $u \in \gg$, $\bar u$ denotes its class in $\gh$. 
Since $\gj$ is normal, there is an induced representation $\ad_{| \, _{\gj}}^*$ of $\gh$ on  $\gj^*$ and viewing the form $\o_{\gh}$ as a linear map $\gh \ra \gj^*$ it defines a non-singular one-cocycle with respect to this
 representation: 
 
\bp \label{prop:induced_fc}
Let $\gj$ be a normal ideal in $(\gg,\o)$. Then the 
following hold:
\begin{enumerate}
\item The adjoint representation of 
$\gg$ gives rise to an induced representation $\ad_{| \, {\gj}}$ 
of $\gh$ on $\gj$. 
\item The homomorphism $\o_{\gh} \in \Hom(\gh, \gj^*)$, $\bar u \mapsto \o_{\gh}(\bar u, \cdot)$,   is an
isomorphism $\gh \ra   \gj^*$. Moreover, it defines a one-cocycle  
in $Z^1_{\ad|_{\gj}^*}(\gh, \gj^*)$, where $\ad_{| \, _{\gj}}^*$ denotes
the representation of $\gh$ on $\gj^*$ which is dual to  $\ad_{| \, {\gj}}$. 
\item The Lie algebra $\gh = \gg / \,  \gj^{\perp_{\omega}}$ carries 
a torsion-free flat connection $\bar \nabla = \bar \nabla^{\omega}$ which 
is defined by the equation
\begin{equation} \label{eq:nablao} \o_{\gh} (\bar \n_{\bar u} \bar v, a) \, = \, \,  -\o ( v, [u,a])
\quad\text{, for all } \bar u, \bar v\in \gh, \; 
a\in \gj.
\end{equation}
\end{enumerate}
\ep  
\pf % By Lemma \ref{lem:normal_red}, 
Since $\gj$ is centralized by $\gj^{\perp_{\omega}}$, the restricted representation $\ad_{| \, {\gj}}$ of $\gg$ is well defined on $\gh = \gg / \,  \gj^{\perp_{\omega}}$. Thus 1.\ holds.

By \eqref{eq:baro}, the element $\bar u \in \gh$ is contained in $\ker \o_{\gh} $ if and only if $u \perp_{\omega} \gj$, that is, if $u \in \gj^{\perp_{\omega}}$. Thus  $\bar u \in \ker \o_{\gh}$ implies
$\bar u = 0$. Since $\codim \gj^{\perp_{\omega}} = \dim \gj$, 
$\dim \gh = \dim \gj^*$. Therefore, $\o_{\gh}$ is an isomorphism.

The one-cocycle condition (cf.\ \eqref{eq:cocycles}) 
for $\o_{\gh}$ with
respect to the coadjoint representation $\rho = \ad_{| \, {\gj}}^*$ on $\gj^*$ reads: 
$$  0 = \rho(\bar u) \o_{\gh}(\bar v) -  \rho(\bar v) \o_{\gh} (\bar u) - \o_{\gh}  ([\bar u, \bar v]) ,  \quad\mbox{for all}\quad \bar u, \bar v\in \gh, $$
This is equivalent to the requirement that, for all $a \in \gj$,
\begin{equation}  -  \o(v, [u, a]) +  \o( u, [v, a]) - \o([u,v], a) = 0 \; . 
\label{eq:onecocycle}
\end{equation}
The latter holds by \eqref{eq:oclosed}, since $\omega$ is closed. This shows 2.

Finally, 3.\ is a standard consequence of 1.\ and 2.
Indeed, the coadjoint representation together with the
one-cocycle $\o_{\gh}$ define an affine representation of the Lie algebra $\gh$ on the vector space $\gj^*$. This representation is \'etale, since $\o_{\gh}$ is an isomorphism. 
To every \'etale affine representation of $\gh$ there is an associated 
torsion-free flat
connection on $\gh$  (compare, for example, \cite{Ba,BC}), which in our case
satisfies \eqref{eq:nablao}.
\epf  

\emph{Remark.} 
The proposition generalizes the well-known fact (compare \cite[Theorem 6]{Chu} and also \cite[\S 1] {MR1}) 
that $\gg$ admits a torsion-free flat connection which arises from $\o$ in a natural way. Indeed, taking $\gj = \gg$ we have $\gj^{\perp_{\omega}} = \{ 0 \} $ and therefore $\gg$ is a normal 
ideal. Now  \eqref{eq:nablao} 
states that the expression 
\begin{equation}  \label{eq:nablaog}
 \o( \n^\o_{ u}  v, w) \, = \, \,  -\o ( v, [u, w])
\quad\text{, for all } u, v, w \in \gg \; ,
\end{equation}
defines a torsion-free flat connection $\nabla^\omega$ on $\gg$. 
Note also that the natural map of flat Lie algebras $(\gg, \nabla^\omega) \ra (\gh,\bar \nabla^{\omega})$ 
is affine, that is, it is a connection preserving map. \\

As a particular consequence, we note:
\bp   \label{prop:ab_h}
If $\gj$ is central then the quotient Lie algebra $\gh= \gg / \,  \gj^{\perp_{\omega}}$ is abelian and the induced connection 
$\bar \nabla$ on $\gh$ is the trivial connection $\bar \nabla \equiv 0$.
\ep 
\pf
Indeed, if $\gj$ is central then $\bar \nabla \equiv  \bar \nabla^{\omega} \equiv 0$, by 
\eqref{eq:nablao}. Since $\bar \nabla$ is torsion-free,  
$[\bar u,\bar v]  = \bar \nabla_{\bar u} \bar v -  \bar \nabla_{\bar v} \bar u = 0$, for all $\bar u, \bar v \in \gh$. This shows the proposition. (Alternatively, the first claim can be read off directly from \eqref{eq:onecocycle}.)
\epf

In the context of \emph{normal reduction},  we are 
considering normal \emph{isotropic} ideals $\gj$. 
By the above, for every normal isotropic ideal $\gj$, 
the associated quotient Lie algebra $\gh = \gg/ \gj^{\perp_{\omega}}$ carries a torsion-free flat
connection $\bar \nabla = \bar \nabla^\o$.  

\bd \label{def:flatquotient}
The flat Lie algebra $(\gh,\bar \nabla)$ is called the
\emph{quotient flat Lie algebra} 
associated to reduction with respect to the normal
ideal $\gj$. 
\ed

\paragraph{Totally geodesic subalgebras}
The properties of the symplectic Lie algebra $(\gg, \omega)$  closely interact with the geometry of the flat Lie algebra $(\gg, \nabla^\o)$ and also with its flat quotients $(\gh, \bar \nabla^\o)$.
For instance, we have:
\bp \label{prop:tg}
Let $(\gg, \omega)$ be a symplectic Lie algebra and $\nabla^\o$ its associated torsion-free flat connection.Then:
\begin{enumerate}
\item  A subalgebra $\gl$ of $\gg$ is a totally geodesic subalgebra with respect to 
$\nabla = \nabla^\o$ if and only if $[ \gl, \gl^{\perp_{\o}} ] \subseteq 
 \gl^{\perp_{\o}} $. 
 \item For every ideal $\gi$ of $\gg$, the orthogonal subalgebra 
 $\gi^{\perp_{\o}}$ is totally geodesic.
 \item Every isotropic ideal $\gj$ of $(\gg, \omega)$ is a totally geodesic subalgebra and the induced connection $\nabla^\omega$ on $\gj$ is trivial.
 \item Lagrangian subalgebras of $(\gg, \omega)$ are totally geodesic.
\end{enumerate}
\ep 
\pf Let $u,v \in \gl$. Then $\nabla_{u} v \in \gl$ if and only if, for all $w \in  \gl^{\perp_{\o}}\! $, $\o(  \nabla_{u} v , w) = 0$. 
That is $\o(  \nabla_{u} v , w) = - \o(v, [u, w]) = 0$. Hence, 1.\ is proved. The further assertions are immediate consequences. 
For 3.\ one uses \eqref{eq:nablaog} to show that the induced connection $\nabla^\omega$ on $\gj$ is trivial. 
\epf 
In particular, Lagrangian subalgebras $\gl$ of $(\gg, \omega)$ carry an induced flat structure $(\gl,  \nabla^\o)$ from $(\gg, \nabla^\o)$. Contrasting with 3., the results in Section \ref{sect:Lag_subgps} below show that every flat Lie algebra $(\gh, \nabla)$ appears in this way as a Lagrangian subalgebra 
for some $(\gg, \nabla^\o)$.   

\subsubsection{Induced derivation algebra on $(\bar \gg, \bar \o)$}
% on $(\bar \gg,\bar \o)$.}
Let $\gj$ be a normal isotropic ideal in $(\gg, \o)$. Dividing out the ideal $\gj$ from the exact 
sequence \eqref{eq:n1}  gives an extension of Lie algebras
\begin{equation} 
0 \ra \bar \gg \ra   \gg/ \gj   \ra  \gh \ra 0 \; \, ,  \label{eq:n3}
\end{equation}
which directly relates $\gh$ and $\bar \gg$. Let
%  conjugation class 
$$ \hat \varphi: \gh \ra \out(\bar \gg) = \Der(\bar \gg)/\,  \Inn(\bar \gg) \;  $$ 
%and a cohomology
%class $\hat \mu \in H^2_{\hat \varphi}(\gh,Z(\bar \gg))$. 
be the conjugation class which belongs to this extension.
The preimage of $\im \hat \varphi = \hat \varphi(\gh) \subseteq \out(\bar \gg)$ under the projection $\Der(\bar \gg) \ra \out(\bar \gg)$ is
a Lie subalgebra \begin{equation}
 \hat \gq  \, \subseteq \, \Der(\bar \gg) \label{eq:hatq} 
\end{equation}
of derivations of $\bar \gg$, which is determined by $\hat \varphi$.
The algebra $\hat \gq $ consists of the inner derivations of $\bar \gg$ together with the derivations induced from $\gh$. This also shows that $\hat \gq$ is induced by 
the image of the adjoint representation of $\gg$ restricted to the ideal $\gj^{\perp_{\omega}}$.
As will be derived in Proposition \ref{prop:sympcond_q} below, 
the algebra $\hat \gq $ satisfies a strong compatibility condition with
respect to the symplectic form $\bar \o$ on $\bar \gg$.\\
%In order to control the properties of the algebra 
%$\gq$, 
 
\noindent 
Since $\gj$ is isotropic in $(\gg, \o)$ we can 
choose an isotropic decomposition  
\begin{equation}
 \gg = N \oplus W \oplus \gj   \label{eq:iso_dec} \; .
\end{equation}   
Then, in particular, 
$$ \gj^{\perp_{\omega}} = W \oplus \gj \text{ and }
N^{\perp_{\omega}}= W \oplus N \; . $$
Using \eqref{eq:iso_dec},  we identify the  underlying vector spaces of the quotient Lie algebras $\gh$, $\gg/\, \gj$, $\bar \gg$ with $N$, $N \oplus W$, $W$ via the respective quotient maps. 
%For simplicity of the exposition, \emph{we will consider in the following  exclusively the situation that $\gj$ is central in $\gg$}. 

Choosing  a  representing map 
 \begin{equation}
 \varphi: N \longrightarrow \Der(\bar \gg) \; \, , \;  n \, \mapsto \,  \varphi_{n} \;  
 \label{eq:rep_map}
 \end{equation} for the conjugation class $\hat \varphi$, 
the Lie products for  $\gg/\, \gj$ and $\gg$ can be described 
% with respect to \eqref{eq:iso_dec} 
as follows: 
 % (belonging to \eqref{eq:n3})
\begin{enumerate}
\item 
There exists 
$$ \mu \in {\bigwedge}^{\! 2} \, \gh^* \tensor \bar \gg$$   
such that, given $n,n' \in N$, $w,w' \in W$,
the expression 
$$ [ n+w , n'+ w' ]_{\gg/\gj} =  [n,n']_{\gh} + \mu(n,n') + 
\varphi_{n} (w') -  \varphi_{n'} (w) + [w,w']_{\bar \gg}   $$  
describes the Lie product of $\gg/\, \gj$ on the 
vector space $N \oplus W$.  \\
\item % Assuming that 
Since $\gj$ is central in $\gj^{\perp_{\o}}$, 
the extension \eqref{eq:n2} is determined by a two-cocycle 
$$\alpha \in Z^2(\bar \gg, \gj) \,  $$
such that,  for all $v,w \in W$,
\begin{equation}
\label{eq:lie1f}  [v,w]   =  {[v , w]}_{\bar \gg}  + \alpha(v,w)  \;  
\end{equation}
describes the Lie product 
$ [ , ] : \gg \times \gg \longrightarrow \gg$ restricted to $W$.
\item Similarly, there is a one-cochain 
$$\lambda \in {\mathcal C}^1(\bar \gg,\Hom(\gh,\gj)) \, \; , \; v \mapsto \lambda(v)
\, , $$ 
such that, for all $n \in N$, $v \in W$,
\begin{equation}
\label{eq:lie2f} \lbrack n , v \rbrack  \, =  \, \varphi_{n}(v) + \lambda_{n}(v)  \; . 
% \label{eq:lie3} \lbrack n , j \rbrack  =  [v,j]  & = & 0 \; . 
\end{equation} 
\end{enumerate}

Observe that 1.\ to 3.\ determine the Lie product of $\gg$ in case that
$\gj$ is central in $\gj$. 
 We remark that  
% In a symplectic Lie algebra $(\gg, \o)$ 
the extension cocycles 
$\alpha$ and $\mu$ 
are determined by the symplectic reduction $(\bar \gg,\bar \o)$ together with
the maps  $\varphi$ and $\lambda$.   % This is shown by the following:

\bl  \label{lem:cocycle_relations}
Let $(\bar \gg, \bar \o)$ be the reduction with respect to the
normal isotropic ideal $\gj$ of $(\gg, \o)$, and $\varphi: N \ra \Der(\bar \gg)$
the representing map  \eqref{eq:rep_map}. 
Then, for all $u,v \in W$, $n,n' \in N$, the following hold:
\begin{enumerate}
\item $\o_{\gh}(\alpha(u,v),n) = \; - \left(\bar \o(\varphi_{n}(u),v) + \bar \o(u,\varphi_{n}(v)) \right)$. 
\item $\bar \o(\mu(n, n'),u)  = \; \o_{\gh}(\lambda_{n}(u), n') + \o_{\gh}(n, \lambda_{n'}(u))$.
\end{enumerate}
\el  
\pf Both equations are enforced by the cocycle condition
 \eqref{eq:oclosed} for $\o$. 
For example, we have (using \eqref{eq:lie1f} and  \eqref{eq:lie2f}): 
\begin{equation*} \begin{split}
0  & = \o([n,u],v) +  \o([v,n],u) +  \o([u,v],n) \\
& =  \bar \o( \varphi_n u,v) +  \bar \o(- \varphi_{n}v ,u) +  \o_{\gh}( \alpha(u,v),n) \; . 
\end{split}
\end{equation*}
This implies 1.
\epf 

We derive now 
our principal result on the derivation algebra $\hat \gq $ (defined in \eqref{eq:hatq})
in the case of central reduction: 

\bp   \label{prop:sympcond_q}
Let $(\bar \gg, \bar \o)$ be the reduction with respect to the
central isotropic ideal $\gj$ of $(\gg, \o)$. Then, for all $u,v \in W$, $n,n' \in N$: 
\begin{equation} 
\begin{split}  \label{eq:sympcond_q}
 \o_{\gh}(\lambda_{n'}([u,v]_{\bar \gg}),n)   = &  \\ 
- \, ( \,  \bar \o(\varphi_{n} \varphi_{n'} u, v) + &  \bar \o(\varphi_{n'} u, \varphi_{n} v) + \bar \o(\varphi_{n} u, \varphi_{n'} v) +  \bar \o(u, \varphi_{n} \varphi_{n'} v) \, ) \; .
\end{split}
\end{equation}
\ep 
\begin{proof}  Note that the expression
 \begin{equation}
 \alpha_{\varphi}(v,w) = \; \alpha(\varphi v, w) + \alpha(v, \varphi w) \; \text{ ,  for all $v,w \in \bar \gg$}, 
\end{equation}
defines a two-cocycle $\alpha_{\varphi} \in Z^2(\bar \gg, \Hom(\gh, \gj))$, where
$\Hom(\gh, \gj)$ is considered to have the trivial $\bar \gg$-module structure.
Using  that $\gj$ is central we may read off from the Jacobi-identity in $\gg$ 
that the relation  
\begin{equation} \label{eq:cobound_alphaphi}
\alpha_{\varphi} (v,w) =   \lambda([v,w]_{\bar \gg})   
\end{equation}
is satisfied, for all $v,w \in \bar \gg$.  (This shows that 
$  \alpha_{\varphi} = - \partial \lambda$ 
and, in  particular, $\alpha_{\varphi}$ is a coboundary.)
Combining Lemma  \ref{lem:cocycle_relations} with  
\eqref{eq:cobound_alphaphi} proves \eqref{eq:sympcond_q}.
\end{proof}

For central reductions to abelian symplectic Lie algebras 
$(\bar \gg, \bar \o)$ we deduce:
\bp \label{prop:abelian_red} Let $(\bar \gg, \bar \o)$ be the reduction with respect to the
central isotropic ideal $\gj$ of $(\gg, \o)$ and $\gq$ the image of 
the lift $\varphi: N \ra \Der(\bar \gg)$.
If $\bar \gg$ is abelian then the following hold:
\begin{enumerate}
\item The image 
$ \gq = \varphi_{N}  \, \subseteq \End(\bar \gg)$ 
is an abelian Lie subalgebra of endomorphisms of 
$\bar \gg$. 
\item For all $\varphi, \psi \in \gq$,  we have 
% the cocycle vanishing  condition 
\begin{equation} \label{eq:cobound_zero}
% \begin{split}
0  \; =  \; \bar \o (\varphi \psi u ,v) +
\bar \o (\varphi u,\psi v)  + \bar \o (\psi u,\varphi v)+ \bar \o (u,\varphi \psi v )
\; .  
% \end{split}
\end{equation}
\end{enumerate}
\ep 
\pf Since $\bar \gg$ is abelian, the map $\hat \varphi$ is 
indeed a homomorphism $\gh \ra \End(\bar \gg)$, and
the representing map $\varphi = \hat \varphi$ is a homomorphism. 
As observed in Proposition \ref{prop:ab_h}, the Lie algebra $\gh$ is abelian, 
since $\gj$ is central. It follows that $\gq$ is a homomorphic 
image of an abelian Lie algebra.  Moreover, \eqref{eq:cobound_zero} is a direct consequence of 
\eqref{eq:cobound_alphaphi}.
\epf 

\subsection{Symplectic oxidation} \label{sect:sympl_ox}
Given an appropiate set of additional data on a symplectic 
Lie algebra $(\bar \gg, \bar \o)$ the process of reduction can
be reversed to construct a symplectic Lie algebra $(\gg, \o)$ 
which reduces to $(\bar \gg, \bar \o)$. In this context we speak of
\emph{symplectic oxidation}. Here we deal with the important 
special case of one-dimensional central 
oxidation.\footnote{This special case of symplectic 
oxidation essentially coincides with the ``double extension'' as developed 
in \cite[Th\'eor\`eme 2.3]{MR1}.} The case of Lagrangian oxidation will be covered separately in Section \ref{sect:Lagrange_ext}. 
%(Similar constructions
%have been considered in the literature and with different naming 
%conventions, cf.\ \cite{Bo,DM, ??}.)

% \subsubsection{Oxidation of Lie algebras} 
\paragraph{Oxidation of Lie algebras.}
Let $\bar \gg$ be a Lie algebra and $\overline{[ , ]}$ its Lie bracket.
Assume the following set of additional data is  given:
\begin{enumerate}
\item a derivation $\varphi \in {\Der(\bar \gg)}$, 
\item   
a two-cocycle $\alpha \in Z^2(\bar \gg)$,
\item  a linear form $\lambda \in \bar \gg^*$.
%\item {a linear map $\varphi \in {\rm End(\bar \gg)}$.} 
%\item {a linear form $\lambda \in \bar \gg^*$.} 
%\item {a two-form $\alpha \in \bigwedge^2 \bar \gg^*$.}
\end{enumerate}
Let  \begin{equation} \label{eq:lie0}
\gg = \; \,  \langle \xi  \rangle \oplus \, \bar \gg \, \oplus \langle H \rangle 
\end{equation}
be the vector space direct  sum of  $\bar \gg$ with one-dimensional
vector spaces $\langle \xi  \rangle , \langle H \rangle$.  
Define an alternating bilinear product 
$$ [ , ] : \gg \times \gg \longrightarrow \gg  \; 
$$ % $ \in \bigwedge^2 \gg^* \otimes \gg$ 
by requiring that the non-zero 
brackets are given by
\vspace*{-1ex} \begin{eqnarray}
\label{eq:lie1}  [v,w]  & = &   \overline{[v , w]} + \alpha(v,w) H \; , \\
\label{eq:lie2} \lbrack \xi, v \rbrack & = & \varphi(v) + \lambda(v) H \; , 
% \label{eq:lie3} \lbrack \xi, H \rbrack  =  [v,H]  & = & 0 \; . 
\end{eqnarray} 
where $v,w \in \bar \gg \,  \subseteq \gg$,  and  $\overline{[\,  , \, ]}$
denotes the bracket in $\bar \gg$. In particular, $\gj = \spanof{H}$ is a central
ideal of $\gg$.

\bp  \label{prop:Lieox}
The alternating product $[ , ]$ as declared in \eqref{eq:lie1} - \eqref{eq:lie2} above defines
a Lie algebra $\gg = (\gg, [ , ])$ % =  \gg_{\bar \gg, \varphi, \lambda, \alpha}$ 
if and only if 
\begin{equation}
\alpha_{\varphi} =  - \partial  \lambda  \;  \in B^2( \bar \gg) .   \label{eq:cobound}  
\end{equation}  
(That is, $\alpha_{\varphi}$  is a coboundary and 
satisfies  $ \alpha_{\varphi}(v,w) =   \lambda( \overline{[v , w]})$.) 
\ep 
\pf Recall from \eqref{eq:alphaphi} that $\alpha_{\varphi}(v,w) = \alpha(\varphi v, w) + \alpha(v, \varphi w)$,  $v,w \in \bar \gg$. 
Then it is easily verified that the coboundary condition \eqref{eq:cobound}  is equivalent to the Jacobi-identity 
for the alternating product $[ , ]$.
\epf

\vspace{1ex} \noindent 
We call the Lie algebra $\gg = \bar \gg_{\varphi, \alpha, \lambda}$ the \emph{central oxidation of $\bar \gg$} (with respect to the data $\varphi, \alpha, \lambda$). \\

We remark next that central oxidation with nilpotent derivations constructs nilpotent Lie algebras. 

\bp  Let  $\bar \gg$ be a nilpotent Lie algebra and assume that the derivation
$\varphi$ above is nilpotent. Then the Lie algebra 
$\gg = \bar \gg_{\varphi, \alpha, \lambda}$ 
is nilpotent, and $H$ is contained in the center of $\gg$.
\ep 
\pf  Recall that Engel's theorem \cite[III, Theorem 2.4]
{Helgason} asserts 
that a Lie algebra is nilpotent if and only if all linear 
operators of the adjoint representation are nilpotent.
With our assumptions, the 
linear operators of the adjoint representation of $\gg$, 
which are defined by \eqref{eq:lie1} and \eqref{eq:lie2},   %\eqref{eq:lie3} 
are nilpotent. Moreover,  $H$ is always central in $\gg$. 
%Therefore, the proposition follows
%from Engel's theorem.
% and
%that, for all $j \geq 1$,
%$ C^{n- c_{j}} \gg \subseteq 
% C^{j} \bar \gg   \oplus  \langle H \rangle$, where $c_{j} = 
% \dim C^j \gg$, $n = \dim \bar \gg$. 
\epf
\paragraph{Symplectic oxidation.}
Now let $(\bar \gg, \bar \o)$ be a symplectic Lie algebra.
% (recall that,  for $u,v \in \bar \gg$,
%$\bar \o_{\varphi} (u,v)  = \bar \omega(\varphi u, v) + \bar \omega(u, \varphi v)$). 
%$\omega$ be the non-degenerate two-form  on the vector space $\gg$ 
%such that the decomposition \eqref{eq:li0} is isotropic and $\omega(\xi,H) = 1$.
%
% By Proposition \ref{prop:Lieox}, 
%According to Proposition \ref{prop:Lieox}, equations \eqref{eq:lie1} -  \eqref{eq:lie1}
%define a Lie algebra $\gg = \bar \gg_{\varphi, \lambda, \alpha_{\varphi}}$ if 
%and only if 
%% with respect to the data $\varphi, \alpha =  \alpha_{ \varphi}$ equations \eqref{eq:lie1} - \eqref{eq:lie3} 
%$\beta = \alpha_{\varphi}$ satisfies \begin{equation}  \label{eq:partiallambda}
%\beta = \beta_{\alpha_{\varphi}, \varphi} = \;  \, - \partial \lambda   \;  .  
%\end{equation} 
We define a non-degenerate two-form $\omega$ on 
the vector space $\gg$ defined in \eqref{eq:lie0} by requiring that \begin{enumerate}
\item $\omega$ restricts to $\bar \omega$
on $\bigwedge^2 \bar \gg$,
\item the decomposition \eqref{eq:lie0} of $\gg$ satisfies $\bar \gg^{\perp_{\omega}}  =  \langle \xi, H \rangle$, 
\item $\omega(\xi, H)  = 1$.
%$\xi \perp_{\omega} \bar \gg$,
%$H  \perp_{\omega} \bar \gg$.
\end{enumerate}
In particular, the decomposition \eqref{eq:lie0} of $\gg$ is isotropic. \\
% It turns out that $(\gg, \omega)$ is  a symplectic Lie algebra: 
%
%\bl  \label{lem:symplecticlie}
%Let $\alpha \in Z^2(\bar \gg)$ be a two-cocycle
%such that the coboundary condition \eqref{eq:partiallambda} is satisfied. Let $\omega \in \bigwedge^2 \gg$ be declared as above,
% if and only if $\alpha = \alpha_{\varphi}$.
%\el 
% 
%In particular, we note:  

Let $\varphi \in {\rm Der}(\bar \gg)$ and 
$  % \alpha \, = \;
 \bar \o_{\varphi} $ the  two-cocycle associated to
$\bar \o$ and  $\varphi$ (using notation as  in \eqref{eq:alphaphi}).

\bp \label{prop:symplecticox} 
Let $\alpha \in Z^2(\gg)$ be a two-cocycle such that 
$ \alpha_{\varphi} = - \partial \lambda$, for some $\lambda \in \bar \gg^*$.
Then $$ (\gg, \omega)  =  
( \bar \gg_{\varphi, \alpha, \lambda}, \omega)$$ is a symplectic Lie algebra
if and only if
\begin{equation}  \label{eq:alphabo}
 \alpha =  \bar \o_{\varphi}  \; . 
\end{equation}
\ep 
\pf The coboundary condition for $\alpha$ is required by
 Proposition \ref{prop:Lieox} and ensures that $ [ , ]$ as 
 defined in \eqref{eq:lie1} - \eqref{eq:lie2} is a Lie algebra
 product. Now, for $\omega$ to be closed it is required that, for 
 all $u,v \in \bar \gg$, 
 \begin{equation*}  \begin{split} 
  0  & = \omega([\xi, u], v) +  \omega([v, \xi], u) +  \omega([u, v], \xi ) \\ 
  & =   \bar  \omega(\varphi u, v) - \bar \omega( \varphi v, u) -  \alpha(u,v) ,
 \end{split}
 \end{equation*}
 which is equivalent to $\alpha =  \bar \omega_{\varphi}$. 
 It is straightforward to verify that 
 $ (\gg, \omega) $ is symplectic if and only if 
 the latter condition holds. 
 \epf

\noindent 
The symplectic Lie algebra $(\gg,\o) = ( \bar \gg_{\varphi, \bar \o_{\varphi}, \lambda}, \omega)$ is called \emph{symplectic oxidation 
of $(\bar \gg, \bar \o)$} with respect to the data $\varphi$ and $\lambda$. 
Observe that the  symplectic oxidation
reduces 
to $(\bar \gg, \bar \o)$ with respect to the one-\-dimensional
central ideal $\gj = \langle H \rangle$.

Given $\varphi  \in \Der(\bar \gg)$, the derived two-cocycle is $\alpha = \bar \o_{\varphi}$, and its second derivation is the two-cocycle $\bar \o_{\varphi,\varphi} =  \alpha_{\varphi}$. 
%which, for all $u,v \in \bar \gg$, thus satisfies,  
%\be \o_{\varphi,\varphi} (u,v) := \, \o (\varphi^2 u,v) +
%2\o (\varphi u,\varphi u) + \o (u,\varphi^2 v)   .
%\ee
According to the previous proposition the cohomology class
of $\alpha_{\varphi}$ is the obstruction for the existence of a symplectic oxidation to the data $(\bar \gg, \bar \o, \varphi)$: 

\bc Let $(\bar \gg, \bar \o)$ be a symplectic Lie algebra and
$\varphi \in \Der(\bar \gg)$ a derivation. Then there exists 
a symplectic oxidation $( \bar \gg_{\varphi, \bar \o_{\varphi}, \lambda}, \omega)$ if and only if the cohomology class
$$   [ \bar \o_{\varphi, \varphi} ] \, \in H^{2}(\bar \gg) $$
vanishes. 
\ec 

\paragraph{Oxidation reverses reduction}
Let  $(\gg, \omega)$ be a symplectic Lie algebra,  
$\gj =  \langle H \rangle$ a one dimensional central ideal
and $(\bar \gg, \bar \omega)$ the symplectic reduction with respect to $\gj$.
We may choose
$\xi$ in the complement of $\gj^{\perp_{\omega}}$ with 
$\omega(\xi, H) =1$ and an isotropic     
direct sum decomposition $$ 
\gg = \langle \xi \rangle \oplus \gj^{\perp_{\omega}}=  \langle \xi \rangle \oplus \bar \gg \oplus \gj
%   \; ,  \text{ where } \bar \gg =  \xi^{\perp_{\omega}} \cap H^{\perp_{\omega}}
   \; . $$
The adjoint operator  $\ad\!({\xi}) \in \Der(\gg)$ induces a derivation 
$$ \varphi =  \ad\!({\xi}){|_{\bar \gg}} \in \Der(\bar \gg) $$
of the reduced Lie algebra $\bar \gg = \gj^{\perp_{\omega}}/ \gj$.
With respect to the above decomposition we may write, as in \eqref{eq:lie1}, \eqref{eq:lie2}, 
$$  \lbrack \xi, v \rbrack  =  \varphi(v) + \lambda(v) H \; ,  \; [v,w]   =    \overline{[v , w]} + \alpha(v,w) H $$ 
where  $$ \lambda   \in  \bar \gg^* \; , \; \alpha \in Z^2(\bar \gg)$$
are determined by the bracket $[ \, , \, ]$ of $\gg$. \\ 

Symplectic oxidation as introduced in Proposition \ref{prop:symplecticox} above reverses one-\-dimension\-al 
central symplectic reduction:

\bp \label{prop:red_ox}
Let $(\gg, \omega)$ as above be a symplectic Lie algebra which reduces 
to $(\bar \gg, \bar \omega)$ with respect to a one dimensional central ideal 
$\gj =  \langle H \rangle$. Put $\lambda_\omega (v) = 
\omega(\xi, [\xi, v])$, for all $v \in \bar \gg$.
Then
\begin{enumerate}
\item $\lambda = \lambda_{\omega}$.
\item The data $\bar \omega$, $\varphi$, $\lambda_{\omega}$ 
for $\bar \gg$ satisfy the
coboundary relation \eqref{eq:cobound}.
\end{enumerate}
In particular,  
$( \gg, \omega)$ %  =  (\bar \gg_{\varphi, \lambda}, \omega) $$
is the symplectic oxidation of $(\bar \gg, \bar \omega)$ 
with respect to $\varphi$, $\lambda_{\omega}$. 
\ep
\pf  Since  $\lbrack \xi, v \rbrack  =  \varphi(v) + \lambda(v) H$, evidently 
$- \lambda(v) = \omega([\xi, v], \xi)$. Hence 1.\  holds. The remaining assertions
are a consequence of Proposition \ref{prop:Lieox} and of Proposition \ref{prop:symplecticox}.
\epf 

\subsection{Lifting and projection of isotropic subalgebras and ideals, corank of reductions} \label{sect:Lifting}
Let $\gj$ be an isotropic ideal in $(\gg, \o)$ 
and let $(\bar \gg, \bar \o)$ be the symplectic reduction 
with respect to $\gj$. Recall that $\bar \gg = \gj^{\perp_{\o}}/ \gj$.

\paragraph{Isotropic lifting} Let $\bar \ga$ be an isotropic subalgebra in $\bar \gg$. Then the 
preimage $\ga \subset  \gj^{\perp_{\o}}$ of $\bar \ga$ under the quotient homomorphism 
$\gj^{\perp_{\o}} \ra \bar \gg$ is an ideal of $\gj^{\perp_{\o}}$ and it is called the \emph{lifting of $\bar \ga$}. 
The lifting $\ga$ is an isotropic subspace 
of $(\gg, \omega)$; in fact it is a subalgebra of $\gg$. Furthermore if $\bar \ga$ is Lagrangian then $\ga$ is Lagrangian in $(\gg, \o)$.
If $\bar \ga$ is an ideal of $\bar \gg$, then, in general, the lifting
$\ga$ is not an ideal of $\gg$. However, in the context of normal reductions,  we have the following criterion:

\bp \label{prop:isotropic_lifting}
Let  $(\bar \gg, \bar \o)$ be a normal symplectic reduction with
representing cochain $\varphi: N \ra \Der(\bar \gg)$
defined with respect to an isotropic decomposition as in \eqref{eq:rep_map}.  Let $\bar \ga$ be an isotropic ideal in $(\bar \gg, \bar \o)$ and $\ga \subseteq \gg$ the lifting of $\ga$. Then:
\begin{enumerate}
\item The lifting $\ga$ is an isotropic ideal in $(\gg, \o)$ if and 
only if $\bar \ga$ is a $\hat \gq$-invariant ideal of $\bar \gg$.
\item The lifting $\ga$  is an isotropic ideal in $(\gg, \o)$ if and 
only if
$$ \varphi_{n} \, \bar \ga \subseteq \bar \ga \; \;  \text{,  for all $n \in N$}. $$ 
\item The lifting $\ga$ is a Lagrangian ideal in $(\gg, \o)$ if and 
only if $\bar \ga$ is a $\hat \gq$-invariant Lagrangian ideal of $(\bar \gg, \bar \o)$. 
\item If the ideal $\bar \ga$ is characteristic in $\bar \gg$ (that is, $\bar \ga$ is invariant under every derivation of $\bar \gg$) then $\ga$ is an isotropic ideal in $(\gg, \o)$. 
\end{enumerate}
\ep
\pf  Since the reduction is normal, $\gj^{\perp_{\o}}$ is an ideal in $\gg$.  
Recall that $\hat \gq \subset \Der(\bar \gg)$ is the homomorphic image 
of the adjoint representation of $\gg$ restricted to $\gj^{\perp_{\o}}$. 
Therefore, $\ga$ is an ideal of $\gg$ if and only if $\hat \gq \, \bar \ga \subseteq \bar \ga$. This shows 1. Now 2.\ follows also, since $\hat \gq = \varphi_{N} + \Inn(\bar \gg)$ consists of  $\varphi_{N}$ and inner derivations of $\bar \gg$. Finally,  3.\ holds since the lift of every Lagrangian subspace in $(\bar \gg, \bar \o)$ is a Lagrangian subspace in $(\gg, \o)$.
\epf 

\paragraph{Isotropic projection} 
Conversely, if $\ga$ is an isotropic subalgebra of $(\gg, \o)$ then $\bar \ga = \ga \cap \gj^{\perp_{\o}} / \, \gj$ is an isotropic subalgebra of $(\bar \gg, \bar \o)$, which is called the \emph{projection of 
$\ga$}. Note that if $\ga$ is an isotropic ideal in $\gj^{\perp_{\o}}$
then $\bar \ga$ is an isotropic ideal in $(\bar \gg, \bar \o)$. 

\bp \label{prop:iso_projection}
Let $\ga$ be  an isotropic subalgebra of $(\gg, \o)$. Then the following hold:
\begin{enumerate}
\item 
The subalgebra $\bar \ga$ is isotropic in $(\bar \gg, \bar \o)$ 
and $\corank_{\bar \omega} \bar \ga  \leq  \corank_{\omega} \ga$.
\item If $\ga$ is Lagrangian then $\bar \ga$ is Lagrangian. 
\item  Every isotropic subalgebra of $(\bar \gg, \bar \o)$ is the 
projection of an isotropic subalgebra of $(\gg, \o)$ which is contained in $\gj^{\perp_{\o}}$.
\item  Every Lagrangian subalgebra of $(\bar \gg, \bar \o)$ is the 
projection of a Lagrangian subalgebra of $(\gg, \o)$ which is contained in $\gj^{\perp_{\o}}$.
\end{enumerate}
\ep 
\pf
The first statement is a consequence of Lemma \ref{lem:corank}.
The second follows since $\ga$ is Lagrangian if and only if 
$\corank_{\omega} \ga = 0$. Since 
projection reverses the process of lifting, the next two statements are implied. 
\epf 

As an application, we observe that the corank of a symplectic Lie algebra as defined in Section \ref{sect:srank} behaves nicely with respect to symplectic reduction:

\bc \label{cor:corank}
Let $(\gg,\omega)$ be a symplectic Lie algebra and $(\bar \gg,\bar \omega)$ a symplectic reduction (with respect to some isotropic ideal of $(\gg, \o)$). Then 
$$ \corank (\bar \gg,\bar \omega)   \,  \leq \, \corank (\gg,\omega) \; . 
$$
\ec 

In particular, if $(\gg, \o)$ has a Lagrangian ideal then $(\bar \gg,\bar \omega)$ has a Lagrangian ideal. 
% \section{Nilpotent symplectic Lie groups}

%%%%%%%%%%%%%%%%%%%%%%%%%%%%%%%%%%%%%%%%%%%%%
\section{Complete  reduction of symplectic Lie algebras}
\label{sect:cr}
%\paragraph{Completely reducible symplectic  Lie algebras}
% Let us introduce here some new terminology. 
We say that a symplectic Lie algebra $(\gg, \o)$ can be 
\emph{symplectically reduced in $\ell$ steps} 
to a symplectic Lie algebra $(\bar \gg, \bar \o)$ if there exists a sequence  of subsequent symplectic reductions 
% \be \label{eq:rs}
$$
(\gg, \o), (\gg_{1}, \o_{1}), \, \cdots \,, 
(\gg_{\ell}, \o_{\ell}) = (\bar \gg, \bar \o) \; .
$$ 
% \ee 
The number $\ell$ is called the \emph{length} of the \emph{reduction
sequence} %  \eqref{eq:rs}  
and the symplectic Lie algebra  $(\bar \gg, \bar \o)$ is
called its \emph{base}.  The  reduction sequence is called \emph{complete} if its base is an irreducible symplectic Lie algebra, that is, if $(\bar \gg, \bar \o)$ does not have a non-trivial isotropic ideal. Every symplectic Lie algebra $(\bar \gg, \bar \o)$,
which occurs as the base of a \emph{complete} reduction sequence for $(\gg, \o)$, 
will be called an \emph{irreducible symplectic base} for $(\gg, \o)$. 

Our main result on symplectic reduction is Theorem \ref{thm:base_unique} below,  which asserts that all irreducible symplectic bases for $(\gg, \o)$ are isomorphic.  That is, we show that \emph{the irreducible base does not depend on the choice of complete reduction sequence}, but instead  $(\bar \gg, \bar \o)$ is an invariant associated with  $(\gg, \o)$. In particular, the theorem provides a rough classification of symplectic Lie algebras according to the type of their irreducible base. 

%Symplectic Lie algebras may be distinguished by the properties of their irreducible symplectic base. 
As an illustration of this principle, we will show that a symplectic Lie algebra admits a Lagrangian subalgebra if and only if its base has a Lagrangian subalgebra. 
This approach motivates a detailed discussion of the properties of irreducible symplectic Lie algebras. Somewhat surprisingly, as we show in Section \ref{sect:irreducible},  it turns out that irreducible symplectic Lie algebras are solvable of rather restricted type and a complete classification up to symplectomorphism is possible. 
% In Section \ref{sect:irr_Lag} 
This also leads us to the construction of an eight-dimensional irreducible symplectic Lie algebra, which does not have a Lagrangian subalgebra.\footnote{To our knowledge this is the first published example of a symplectic Lie algebra without a Lagrangian subalgebra. Compare also footnote 
\ref{fn:Lag_alg2}.} 
% \vspace{1ex}

\subsection{Complete reduction sequences}
For every reduction sequence of symplectic Lie algebras of the form 
\be \label{eq:rs2}
(\gg, \o), (\gg_{1}, \o_{1}), \, \cdots \,, 
(\gg_{\ell}, \o_{\ell}) = (\bar \gg, \bar \o) \;  
\ee 
let $\bar \gj_{i+1}$ denote the isotropic ideal in $(\gg_{i}, \o_{i})$, which determines the 
reduction to  $(\gg_{i+1}, \o_{i+1})$. 
To such a sequence \eqref{eq:rs2} there belongs an associated nested sequence of isotropic subalgebras $\gj_{i}$ 
of $(\gg, \o)$ such that: 
\begin{equation} \label{eq:rs3}
\gj = \gj_{1} \, \subseteq \, \gj_{2} \, \subseteq \, \ldots \, \subseteq \, \gj_{\ell}  \; 
\; \subseteq \; \gj_{\ell}^{\perp_{\o}} \, \subseteq \,  \ldots \, \subseteq   \, \gj_{2}^{\perp_{\o}} \, \subseteq \, \gj^{\perp_{\o}}  \; , 
\end{equation}
where $\gj_{i}^{\perp_{\o}}$, $i= 1, \ldots, \ell$,  is a subalgebra of $\gg$ and there are isomorphisms  
$$     (\gg_{i}, \omega_{i})  \, =  \, ( \gj_{i}^{\perp_{\o}} / \, \gj_{i}, \bar \omega_{i}) \, .  $$
(Here $\bar \omega_{i}$ is the symplectic form induced on the quotient by the restriction of $\omega$ to $\gj_{i}^{\perp_{\o}}$.) The  sequence 
\eqref{eq:rs3} is constructed inductively
in such a way that $\gj_{i+1}$ is the preimage in $\gj_{i}^{\perp_{\o}}$ of the isotropic ideal $\bar \gj_{i+1}$, where
we put $\gj_{0} = \{ 0 \}$.
Note in particular, that 
\begin{itemize}
\item[\rm{(*)}]
$\gj_{i+1}$ is  an ideal in $\gj_{i}^{\perp_{\o}}$, for all $i= 0, \ldots, \ell-1$.
\end{itemize}
Sequences \eqref{eq:rs3} with this property will be called reduction sequences of isotropic subalgebras.
%$\gj_{i+1}$ is an ideal in $\gj_{i}^{\perp_{\o}}$. In general, we have: 

\bp There is a one-to-one correspondence between reduction
sequences of symplectic  Lie algebras as in \eqref{eq:rs2} and reduction sequences of isotropic subalgebras  (that is, sequences \eqref{eq:rs3} which satisfy (*) ).
\ep  
\pf
Consider a nested sequence of isotropic subalgebras $\gj_{i}$ as in  \eqref{eq:rs3}, which satisfies (*).  Since $\gj_{i+1}$ is 
an ideal in $\gj_{i}^{\perp_{\o}}$, $\gj_{i+1}^{\perp_{\o}}$ is a subalgebra of $\gg$ and  $\gj_{i+1}$ projects to 
an isotropic ideal $\bar \gj_{i+1}$ in $$ (\gg_{i}, \omega_{i}) := ( \gj_{i}^{\perp_{\o}}/ \, \gj_{i}, \bar \omega) \; , $$ so that the reduction of $(\gg_{i}, \omega_{i})$ with respect to $\bar \gj_{i+1}$ is isomorphic to the quotient 
$( \gj_{i+1}^{\perp_{\o}}/ \, \gj_{i+1}, \bar \omega)$. This shows
that nested sequences \eqref{eq:rs3} which satisfy (*) give rise to reduction sequences of the form \eqref{eq:rs2}.
\epf 

Let $(\bar \gg, \bar \o)$ be a reduction of $(\gg, \o)$ 
% with  respect to the isotropic ideal $\gi$ 
and let $\gj_{i}$  be a reduction sequence of isotropic subalgebras for $(\gg, \o)$.  As the following result shows, there is an induced reduction sequence 
on  $(\bar \gg, \bar \o)$. 
\bp \label{prop:inducedreduction}
Given a reduction sequence $\gj_{i}$, $i= 1, \ldots, \ell$,  of isotropic subalgebras for $(\gg, \o)$ and $\gi$ an isotropic ideal of $(\gg, \o)$, the sequence
\begin{equation} \label{eq:rs4}
 \gi \; \subseteq \;  \gi +\gj_{1} \cap  \gi^{\perp_{\o}} \; \subseteq \;  \gi  +  \gj_{2} \cap \gi^{\perp_{\o}}  \; \subseteq \, \ldots \subseteq  \, \; \gi +  \gj_{\ell} \cap \gi^{\perp_{\o}}
\end{equation}
is a reduction sequence of isotropic subalgebras for $(\gg, \o)$. 
\ep 
\pf Clearly, the sums $ \gi +  \gj_{i} \cap \gi^{\perp_{\o}}$ are isotropic subalgebras of $(\gg, \o)$ and they form a nested sequence, contained in $ \gi^{\perp_{\o}}$. It remains to verify that the subalgebra $ \gi +  \gj_{i+1} \cap \gi^{\perp_{\o}}$ is an ideal of  
\begin{equation} \label{eq:jk1}
\left(  \gi +  \gj_{i} \cap \gi^{\perp_{\o}} \right)^{\perp_{\o}} =  \gi^{\perp_{\o}} \cap   ( \gj_{i}^{\perp_{\o}} + \gi) = \gi +  \gj_{i}^{\perp_{\o}} \cap   \gi^{\perp_{\o}}  \; .
\end{equation}
Taking into account that $\gj_{i+1}$ is  an ideal in $\gj_{i}^{\perp_{\o}}$, this follows, since:  
$$
[  \gi +  \gj_{i}^{\perp_{\o}} \cap   \gi^{\perp_{\o}}\, , \,  \gi +  \gj_{i+1} \cap \gi^{\perp_{\o}} ] \; \subseteq  \; 
\gi +  [ \gj_{i}^{\perp_{\o}},  \gj_{i+1}] \cap \gi^{\perp_{\o}} \; \subseteq  \; \gi +  \gj_{i+1} \cap \gi^{\perp_{\o}}  \; . 
$$
\epf 

Let $(\bar \gg, \bar \o)$ be 
the reduction of $(\gg, \o)$ with 
respect to $\gi$. Then the sequence $\bar \gj_{i} =  \gi +\gj_{i} \cap  \gi^{\perp_{\o}}$ defines the induced reduction sequence for $(\bar \gg, \bar \o)$. \\

Now the following observation on inheritance of irreducible bases is fundamental:

\bp \label{prop:baseinheritance}
Let $(\gb, \omega_{\gb})$ be an irreducible base for the
symplectic Lie algebra $(\gg, \o)$. Then $(\gb, \omega_{\gb})$
is an irreducible base for every reduction  $(\bar \gg, \bar \o)$
of $(\gg, \o)$.
\ep
\pf
Let $\gj_{i}$ be a reduction sequence of isotropic subalgebras of $(\gg, \o)$ such that  
 its associated reduction sequence $(\gg_{i}, \omega_{i})$ has 
base $$ (\gg_{\ell}, \omega_{\ell}) = (\gj_{\ell}^{\perp_{\o}}/ \, \gj_{\ell}, \bar \o_{\ell}) =  (\gb, \omega_{\gb}) \; . $$
According to Proposition \ref{prop:inducedreduction}, for any isotropic ideal  $\gi$ of $(\gg, \omega)$ the base of the induced sequence for the reduction $(\bar \gg, \bar \o)$ with respect to $\gi$ is 
\begin{eqnarray} 
(\bar \gg_{\ell}, \omega_{\bar \gg_{\ell}})  &=  &   
\left( \,  (\gi +  \gj_{\ell}^{\perp_{\o}} \cap   \gi^{\perp_{\o}})  \big/ \, (\gi +  \gj_{\ell} \cap \gi^{\perp_{\o}}) \, , \, \bar \omega \,  \right) \\   
& =  & \left(  \, (\gj_{\ell}^{\perp_{\o}} \cap   \gi^{\perp_{\o}})  \big/ \, (\gi \cap  \gj_{\ell}^{\perp_{\o}} +  \gj_{\ell} \cap \gi^{\perp_{\o}}) \,  , \, \bar \o \, \right) \; . 
\end{eqnarray}
The two bases are related by the diagram of symplectomorphic  natural maps 
\begin{equation} \label{eq:lrarrow}
  (\gj_{\ell}^{\perp_{\o}} / \, \gj_{\ell}, \bar \o) \;  \longleftarrow \; (\gj_{\ell}^{\perp_{\o}} \cap   \gi^{\perp_{\o}} /  \,   \gj_{\ell} \cap  \gi^{\perp_{\o}}, \bar \o)  \; \longrightarrow \; (\bar \gg_{\ell}, \omega_{\bar \gg_{\ell}}) \; . 
\end{equation}
Obviously,  the left arrow is injective, while the right arrow is surjective. 
% Now suppose that $(\gg_{\ell}, \omega_{\ell})$ is an irreducible symplectic Lie algebra. 
Then, since 
$\gi \cap \gj_{\ell}^{\perp_{\o}}$ projects to an isotropic ideal in 
$(\gg_{\ell}, \omega_{\ell})$ (which must be trivial, by the irreducibility of $(\gg_{\ell}, \omega_{\ell})$), we have 
\begin{equation} \label{eq:inclusions}
 \gi \cap \gj_{\ell}^{\perp_{\o}} \, \subseteq  \, \gj_{\ell} \; \text{ and } \; 
(\gi \cap \gj_{\ell}^{\perp_{\o}})^{\perp_{\o}}  \supseteq \,    \gj_{\ell}^{\perp_{\o}} \; .  
\end{equation}
The first inclusion implies 
that $ \gi \cap \gj_{\ell}^{\perp_{\o}} = \gi \cap \gj_{\ell} \subseteq \ \gj_{\ell} \cap  \gi^{\perp_{\o}}$ and, therefore,  
$$ (\bar \gg_{\ell}, \omega_{\bar \gg_{\ell}})   \, = \,  ( \gj_{\ell}^{\perp_{\o}} \cap   \gi^{\perp_{\o}}/ \,  \gj_{\ell} \cap \gi^{\perp_{\o}}, \bar \o) \; .  $$
In particular,  the right arrow in \eqref{eq:lrarrow} is an isomorphism. 
Furthermore, 
\begin{eqnarray} 
 \gj_{\ell}^{\perp_{\o}} \cap   \gi^{\perp_{\o}} + \gj_{\ell} & =  &\left(( \gj_{\ell} + \gi ) \cap \gj_{\ell}^{\perp_{\o}} \right)^{\perp_{\o}} \\
& = & \left(  \gj_{\ell} + \gi  \cap \gj_{\ell}^{\perp_{\o}} \right)^{\perp_{\o}}  \\
& =  &  \gj_{\ell}^{\perp_{\o}}  \cap  \left( \gi  \cap \gj_{\ell}^{\perp_{\o}}\right)^{\perp_{\o}} \; .
\end{eqnarray} 
From the second inclusion in \eqref{eq:inclusions}, we infer that
 $\gj_{\ell}^{\perp_{\o}} \cap   \gi^{\perp_{\o}} + \gj_{\ell} =    \gj_{\ell}^{\perp_{\o}}$.  
Therefore, 
the left arrow in \eqref{eq:lrarrow} is surjective, hence it is an  isomorphism of
symplectic Lie algebras. This shows that the base $(\bar \gg_{\ell}, \omega_{\bar \gg_{\ell}})$ 
for the induced reduction sequence on 
$(\bar \gg, \bar \o)$ is isomorphic to $(\gb, \omega_{\gb})$.
\epf 

We can now deduce the following Jordan-H\"older type uniqueness statement for symplectic bases:
\begin{Th}[Uniqueness of the base] \label{thm:base_unique}
Let $(\gg, \o)$ be a symplectic Lie algebra. Then all its irreducible symplectic bases are isomorphic. 
\end{Th}
\pf We call a reduction sequence $(\gg_{i}, \omega_{i})$ proper if in each step
the dimension of $\gg_{i}$ is reduced. For any symplectic Lie algebra $(\gg, \o)$ define its reduction length $\ell(\gg, \o)$ to be the maximal length of a \emph{proper} reduction sequence of $(\gg, \o)$. We prove the theorem by
induction over the reduction length. Observe that $\ell(\gg, \o) = 0$ if and only if $(\gg, \o)$ is irreducible. Therefore, the theorem holds  for $\ell(\gg, \o) = 0$.
Now let $(\gg, \o)$ have reduction length $\ell(\gg, \o) \geq 1$. We further assume that the theorem is satisfied for all symplectic Lie algebras of reduction length less than $\ell(\gg, \o)$. Let $(\gb, \omega_{\gb})$ be an irreducible base for the
symplectic Lie algebra $(\gg, \o)$.  Let $(\bar \gg, \bar \o)$
be any further irreducible base for $(\gg, \o)$ and 
$
(\gg, \o), (\gg_{1}, \o_{1}), \, \cdots \,, 
(\gg_{\ell}, \o_{\ell}) = (\bar \gg, \bar \o)  
$ a proper reduction sequence. Observe that $\ell(\gg_{1}, \o_{1}) < \ell(\gg, \o)$ 
and that $(\bar \gg, \bar \o)$ is an irreducible base for $(\gg_{1}, \o_{1})$. Moreover, by Proposition \ref{prop:baseinheritance}, 
$(\gb, \omega_{\gb})$ is another base for $(\gg_{1}, \o_{1})$. By our induction assumption it follows that $ (\bar \gg, \bar \o)  $ and $(\gb, \omega_{\gb})$ are isomorphic symplectic Lie algebras.
This shows that $(\gg, \o)$ has only one irreducible base. 
\epf

\subsection{Completely reducible symplectic Lie algebras}
\label{sect:cr_algebras}
We consider the important class of 
symplectic Lie algebras which admit the
trivial symplectic Lie algebra as a base. 
\bd
A symplectic Lie algebra $( \gg, \o)$ is called \emph{completely  
reducible} if it can be symplectically reduced (in several steps) to the trivial symplectic algebra
$ \bar \gg = \{ 0 \}$. 
\ed

Completely solvable symplectic Lie algebras are 
completely reducible:
\bt \label{thm:cs_cr}
Let $(\gg, \o)$ be a symplectic Lie algebra, where $\gg$ is completely solvable. Then $(\gg, \o)$ is completely reducible.
\et
\pf
By Example \ref{Ex:cs}, every completely solvable symplectic 
Lie algebra is completely  
reducible by a series of (not necessarily normal) codimension 
one reductions. 
\epf 

In particular, if $\gg$ is nilpotent then $(\gg, \o)$ is completely reducible.
As we shall remark in Proposition \ref{prop:nilpotent_reduction} if $\gg$ is nilpotent, the reduction steps are central. Note further, that the solvable four-dimensional symplectic Lie algebra in Example \ref{ex:fdim_metab} is completely reducible but not completely solvable. \\

However, in general, a completely reducible symplectic Lie algebra need not be solvable.
A trivial remark is the following: \emph{If $(\gg, \omega)$ has a Lagrangian ideal then 
it is completely reducible}. Therefore, starting with a non-solvable flat Lie algebra 
$(\gh, \nabla)$,  we can construct a non-solvable symplectic Lie algebra $(\gg, \omega)$
by Lagrangian extension. This  symplectic Lie algebra $(\gg, \omega)$ is then 
completely reducible. The following example is of different kind: 

\begin{Ex}[Affine Lie algebras are completely reducible]
\label{ex:affine} 

Let $$ \mathfrak{aff}(n) =  \mathfrak{gl}(n,\mathbb{R}) \oplus \mathbb{R}^n$$  be the Lie algebra of the real affine group.
The Lie product for $\mathfrak{aff}(n)$ is thus given by the formula $$[(A,u), (B, v)] = ( AB -BA, Av - Bu ) \; . $$
The affine Lie algebra $\mathfrak{aff}(n)$ admits symplectic structures and 
all of them are symplectically isomorphic \cite{Agaoka}. According to 
\cite[Th\'eor\`eme 3.7]{MR1}, a symplectic form $\omega$ on $\mathfrak{aff}(n)$ is obtained by declaring 
$$ \omega\left(\,  (A,u), (B, v) \, \right) = \lambda ( Av - Bu ) + \kappa(M,  AB -BA) \; , $$
where $\lambda: \mathbb{R}^n \ra  \mathbb{R}$ is a linear map, $\kappa$ denotes the Killing form on  $\mathfrak{gl}(n,\mathbb{R})$, and $M$ is a diagonal matrix with pairwise distinct eigenvalues, such that the standard basis vectors are not in the kernel of $\lambda$. Observe that the translation ideal $$ \gt = \{ (0, u ) \mid u \in \mathbb{R}^n \}$$ is  isotropic  in the symplectic Lie algebra $(\mathfrak{aff}(n), \omega)$.
Let  $$ \gg_{\lambda} =  \{ A \in  \mathfrak{gl}(n,\mathbb{R}) \mid
\lambda (A ( \mathbb{R}^n)) = \{ 0 \}\} \; . $$Then $\gg_{\lambda}$ is 
a subalgebra of $\mathfrak{gl}(n,\mathbb{R})$, which is easily seen to be isomorphic to $\mathfrak{aff}(n-1)$. 
Since $\gt^{\perp_{\o}} = \gg_{\lambda} \oplus \gt$, 
the Lie algebra $\gt^{\perp_{\o}} / \gt$ of the symplectic reduction of $\mathfrak{aff}(n)$ with respect to $\gt$ is isomorphic to 
$\mathfrak{aff}(n-1)$. Repeating the reduction inductively, we thus obtain a complete reduction sequence of length $n$, which reduces $\mathfrak{aff}(n)$ to the trivial symplectic Lie algebra. 
\end{Ex}

Chu \cite{Chu} observed that four-dimensional symplectic Lie algebras are 
always solvable. (A list of all real four-dimensional symplectic Lie algebras
was presented subsequently in \cite{MR1}, see also \cite{Ovando}.) Here we show (independently of the just mentioned classification results\footnote{Alternatively, a short proof can be given using Proposition \ref{prop:irreducible}.}): 

 \bp \label{prop:dim4reducible}
 Every symplectic Lie algebra of dimension four 
 over the reals is completely reducible. \ep
 
 \pf Let $(\gg ,\o )$ be a  four-dimensional symplectic Lie algebra of symplectic rank zero. 
 Then $\gg$ has, in particular,  
 no one-dimensional ideals. This implies that $\gg$ is solvable, since 
a four-dimensional Lie algebra with non-trivial Levi decomposition is
necessarily reductive with one-dimensional center. Next we conclude that  
that $\ga := [\gg , \gg ]$ is at least two-dimensional, since otherwise $\gg$  
would clearly contain a one-dimensional ideal. The case $\dim \ga =3$ is also
excluded, because $\ga$ would then be either a Heisenberg algebra and its center
is a one-dimensional ideal of $\gg$, so this is not possible. Otherwise  $\ga$ can be abelian, and (since we are working over the reals) contains a weight vector of $\gh=\gg/\ga$, which spans a  one-dimensional ideal. 

Thus the ideal $\ga$ is a two-dimensional nilpotent Lie algebra 
 and, therefore, abelian. Furthermore, the abelian Lie algebra $\gh := \gg/\ga$ acts on $\ga$ by 
 a faithful representation 
 \[ \rho : \gh \ra \mathfrak{gl}(\ga ) \cong \mathfrak{gl}_2,\]
 without eigenvectors. In fact, an eigenvector for $\rho$ would span a one-\-dimen\-sio\-nal (isotropic) ideal. 
 Similarly, a one-\-dimen\-sio\-nal kernel of $\rho$ gives rise to an ideal in $\gg$.  
 Considering the possible Jordan normal forms, this implies that we 
 can choose a basis $\{e_1,e_2 \}$ of $\ga$ and $\{e_3,e_4\}$ of $\gh$ such that
 \[ \r (e_3) =  \left(\begin{matrix}
1 & 0\\
0&1 
\end{matrix}\right),\quad \r (e_4) =  \left(\begin{array}{rr}
0 & -1\\
1&0 
\end{array}\right).\]
 We observe that the cohomology class $[\a] \in H_\r^2(\gh , \ga )$ of the cocycle $\a$, which is associated with the extension
 $0\ra \ga \ra \gg \ra \gh\ra 0$ is trivial. This follows from the fact that $Z^2_\r (\gh ,\ga )= \bigwedge^2\gh^*\ot \ga=
 B^2_\r (\gh ,\ga )$. Therefore, $\gg$ is a semi-direct sum.
 % It is also true that $H^1_\r (\gh , \ga )=0$. 
Since every ideal of $\gg$ is $\gh$-invariant, one can see that the only ideal 
of dimension less than or equal to two is $[\gg ,\gg ]= \ga$. Next we determine all 
closed 2-forms on $\gg$:
% and check that $\ga$ is isotropic for all of them. 
Using the notation introduced in Theorem \ref{thm:noLg84} we note the relations
\[ \partial e^1 =e^{13}-e^{24},\quad  \partial e^2=e^{23}+e^{14},\quad \partial e^3= \partial e^4=0,\]
where $e^{ij}:= e^i\wedge e^j$. Thus 
we obtain 
\[ Z^2(\gg ) = B^2(\gg ) \oplus \spanof{ e^{34}},\quad B^2(\gg) = 
\spanof{ e^{13}-e^{24},e^{23}+e^{14}}. \]
This shows that $\ga$ is isotropic for all closed forms on $\gg$. Thereby proving that there does not
exist a four-dimensional symplectic Lie algebra of symplectic rank zero (over the real numbers).
 \epf

\paragraph{Nilpotent symplectic Lie algebras}
We call a symplectic Lie algebra $(\gg, \o)$ nilpotent if the underlying Lie algebra $\gg$ is nilpotent. 
In the context of nilpotent symplectic Lie algebras, central reduction
provides an important tool, because of the following 
observation (cf.\ \cite[Th\'eor\`eme 2.5]{MR1}): 

\bp \label{prop:nilpotent_reduction}
Let $(\gg, \o)$ be a nilpotent symplectic Lie algebra. Then 
there exists a finite series of central reductions to the trivial 
symplectic Lie algebra. (In particular, a nilpotent symplectic Lie algebra
$(\gg, \o)$ is symplectically irreducible if and only if it is trivial.)%
\ep
\pf Indeed, the center of every nilpotent Lie algebra is a non-trivial ideal. Every one-dimensional subspace of the center is an isotropic ideal. Therefore, every nilpotent symplectic Lie algebra admits an (at least one-\-dimensional) central isotropic ideal. 
\epf

\subsection{Symplectic length}
Let $( \gg, \o)$ be a symplectic Lie algebra. The symplectic length of 
$( \gg, \o)$ measures how close $(\gg, \omega)$ is to its irreducible 
symplectic base. 

\bd
The \emph{symplectic length} $sl(\gg, \o)$ 
of $(\gg, \o)$ is defined to be the minimum number of
steps which is required for symplectic reduction of  
$(\gg, \o)$ to its irreducible base. (That is, $sl(\gg, \o)$ 
is the minimum length of a complete reduction sequence
for $(\gg, \o)$.)
\ed 

Note, if $\dim \gg = 2k$ we have 
$sl(\gg, \o) \leq k  $. Furthermore, $sl(\gg, \o)= 0$
if and only if $(\gg, \o)$ is irreducible. 

\paragraph{Isotropic ideals and corank}

The symplectic length and the base of $(\gg, \o)$ 
relate to the rank, resp.\  corank, of $(\gg, \o)$ of as follows: 

\bt Let $(\gg, \o)$ be a symplectic Lie 
algebra and $(\gb, \o)$ its irreducible base.
Then $\corank (\gg, \o) \geq {1 \over 2} \dim \gb$ with equality
if and only if $sl(\gg, \o) \leq  1$. 
\et 
\pf  Indeed, we have $\corank (\gb, \o) = {1 \over 2} \dim \gb$ 
for every irreducible symplectic Lie algebra $(\gb, \o)$. 
Moreover, by Corollary \ref{cor:corank},  $\corank (\gg, \o) \geq
\corank (\gb, \o)$ if $(\gg, \o)$ reduces to $(\gb, \o)$. 

Now assume that equality holds. Then $(\gg, \o)$ has an 
isotropic ideal $\gj$ with $\corank \gj = {1 \over 2} \dim \gb$. 
Let $(\bar \gg, \bar \o)$ be the reduction of $(\gg, \o)$ with respect 
to $\gj$. By the definition of corank , we have $\dim \bar \gg = \dim \gb$.
By Theorem \ref{thm:base_unique} we know that $(\gb, \o)$ is 
a base for $(\bar \gg, \bar \o)$. Thus, $(\bar \gg, \bar \o) = (\gb, \o)$.
We deduce that $sl(\gg, \o) \leq  1$. Conversely, if  $sl(\gg, \o) \leq  1$ then 
either $(\gg, \o)$ is irreducible or $(\gg, \o)$ has a complete reduction 
sequence of length one. In the latter case, $(\gg, \o)$ has a maximal isotropic 
ideal $\gj$ of corank $ {1 \over 2} \dim \gb$. 
\epf

In particular, if $(\gg,\o)$ is not completely reducible then 
$\corank (\gg, \o) \geq 3$. This follows by Proposition \ref{prop:dim4reducible} (compare also Corollary \ref{cor:irreducible}), which implies that the dimension of any irreducible symplectic Lie algebra is at least six. 
 Furthermore, if the non-trivial symplectic Lie algebra $(\gg,\o)$  is completely reducible then $sl(\gg, \o) = 1$ if and only if $(\gg, \o)$ has a Lagrangian ideal. 
\begin{remark}

%In the definition of symplectic length we do not require 
%that the reduction steps are central.
%%\footnote{\bf Oliver: Indeed, this could
%%define another invariant: central symplectic length.}
%Hence,  $ sl(\gg) =1 $ if and only if $(\gg, \o)$ 
%admits a Lagrangian ideal. 
In Section \ref{sect:c_examples}, we  construct various 
examples of completely reducible  symplectic Lie algebras, which satisfy $sl(\gg,\o) > 1$.
\end{remark}

%Compare, for example, Proposition \ref{prop:Lagr_subalgs} below, which ensures the existence of Lagragian subalgebras in completely reducible symplectic Lie algebras.

\paragraph{Existence of Lagrangian subalgebras}
Inductive arguments with respect to the symplectic length 
can be a useful tool. %  in the context of completely reducible symplectic Lie algebras:
For example, 
if  $\bar \ga$ is 
a Lagrangian subalgebra of the reduction $(\bar \gg, \bar \o)$ then
also the lift $\ga$ of $\bar \ga$ defined in Section \ref{sect:Lifting} is Lagrangian.
A simple consequence is:

\bp \label{prop:Lagr_subalgs1}
Let $(\gg, \o)$ be a symplectic Lie 
algebra and $(\gb, \bar \o)$ its irreducible base. 
Then $(\gg, \o)$ has a Lagrangian subalgebra if and only 
if its base $(\gb, \bar \o)$ has a Lagrangian subalgebra. 
\ep 
\begin{proof}  
As follows from Proposition \ref{prop:iso_projection},  we may project any 
Lagrangian subalgebra $\ga$ of $(\gg, \o)$  to a Lagrangian subalgebra of $(\gb, \bar \o)$ along a reduction sequence with base $(\gb, \bar \o)$. 

We now show that $(\gg, \o)$ has a Lagrangian subalgebra if 
$(\gb, \bar \o)$ has a Lagrangian subalgebra.
The proof is by induction over the sympletic length. The statement is obvious 
for every irreducible symplectic algebra. 
 Assume now that $sl(\gg,\o) = k \geq 1$ and that
the statement holds for all symplectic Lie algebras with symplectic length  $<k$.  
Considering a complete reduction sequence of minimal length, 
we see that $(\gg, \o)$ reduces to a  symplectic Lie algebra  $(\bar \gg, \bar \o)$ 
with $sl(\bar \gg, \bar \o ) < k$.  By the induction assumption, there exists a Lagrangian 
subalgebra $\bar \ga$ in $\bar \gg$. Then the lift 
 $\ga$ of $\bar \ga$ 
 is a Lagrangian subalgebra of $(\gg, \o)$.
\end{proof}

\bc   \label{cor:Lagr_subalgs2}
Let $(\gg, \o)$ be a completely reducible symplectic  Lie 
algebra. Then $(\gg, \o)$ has a Lagrangian subalgebra. 
In particular, Lagrangian subalgebras exist if $\gg$ is completely solvable (in particular if $\gg$ is nilpotent). 
\ec 
\pf It follows from Example \ref{Ex:cs} that every completely solvable symplectic Lie algebra is completely reducible.
\epf 

We further remark that the irreducible symplectic Lie algebras introduced in Example \ref{ex:irreducible} do admit Lagrangian subalgebras, see Lemma \ref{lem:Lag_irr}. This shows that  \emph{every symplectic Lie algebra of dimension less than or equal to six has a Lagrangian subalgebra}. In Proposition \ref{prop:irr_noLag}, we shall construct an eight-dimensional irreducible symplectic Lie algebra, which does not admit a Lagrangian subalgebra.     

\subsection{Lagrangian subalgebras in irreducible symplectic Lie algebras} \label{sect:irreducible}

%Recall % from Section \ref{sect:reduction} 
%that a symplectic Lie algebra is called {\em irreducible} if it
%does not contain a non-trivial isotropic ideal. 
In this subsection we describe the structure of irreducible symplectic Lie algebras and their Lagrangian subalgebras. 
The underlying Lie algebras which belong to irreducible symplectic Lie algebras are metabelian and solvable of imaginary type. 
We use this characterization to show that there exist eight-dimensional irreducible symplectic Lie algebras which do not have a Lagrangian subalgebra, whereas all irreducible symplectic Lie algebras of lower dimension admit a Lagrangian subalgebra. 
%Therefore, this is the lowest dimension in which  such a phenomenon occurs.

\subsubsection{Structure of irreducible symplectic Lie algebras}
We start by presenting a family of six-dimensional symplectic Lie algebras,  which are all irreducible over the real numbers. Indeed, remarkably, \emph{every} symplectic form on the Lie algebra, which is exhibited in the following example, gives rise to an irreducible symplectic Lie algebra.

 \begin{Ex}[six-dimensional, irreducible]
 \label{ex:irreducible}
 Let $\gg$ be the semi-direct sum of a two-dimensional abelian Lie 
 algebra $\gh = \spanof{ e_5,e_6}$ and a four-dimensional abelian ideal 
 $\ga=\spanof{ e_1,\ldots , e_4}$, 
where $\gh$ acts on 
$\ga$ by the following representation $\rho : \gh \ra \mathfrak{gl}(\ga ) = \mathfrak{gl}_4$ :
\begin{equation*} % \label{rhorep} 
\quad \r (e_5) =  \left(\begin{matrix}
J & 0\\
0&0 
\end{matrix}\right),\quad \r (e_6) =  \left(\begin{array}{rr}
0 & 0\\
0&J 
\end{array}\right),\quad  J:=\left(\begin{array}{rr}
0& 1\\
-1&0 
\end{array}\right).\end{equation*}
Using the fact that every ideal of $\gg$ is $\gh$-invariant, one can easily check that the only ideals  
of dimension less than or equal to three are $\ga_1=\spanof{ e_1,e_2}$ and
$\ga_2 = \spanof{ e_3,e_4}$. 
Using the relations
\[ \partial e^1 = e^{25},\quad \partial e^2 = -e^{15},\quad \partial e^3=e^{46},\quad \partial e^4 = -e^{36},\quad \partial e^5= \partial e^6=0,\]
we compute
\begin{eqnarray*} B^2(\gg )  &=& \spanof{  e^{15},e^{25},e^{36},e^{46}} = e^5\wedge \ga_1^* \oplus e^6\wedge \ga_2^*,\\ 
Z^2(\gg ) &=& B^2(\gg ) \oplus 
\spanof{ e^{12},e^{34},e^{56}} = B^2(\gg ) \oplus {\bigwedge}^{\! 2} \ga_1^* \oplus  {\bigwedge}^{\! 2} \ga_2^* \oplus
{ \bigwedge}^{\! 2} \gh^* .
\end{eqnarray*}
Since $e^{12}+e^{34}+e^{56}\in Z^2(\gg )$ is non-degenerate, we see that the 
the symplectic forms on $\gg$ form a \emph{non-empty} open subset of
the seven-dimensional vector space $Z^2(\gg )$. 
Furthermore, \emph{every}
symplectic form $\o$ has non-zero coefficients over the 
two elements $e^{12}$ and $e^{34}$ of the basis
$\{ e^{12},e^{34},e^{56},e^{15},e^{25},e^{36},e^{46} \}$ of $Z^2(\gg)$. Therefore, the ideals $\ga_1$, $\ga_2$ 
are non-degenerate with respect to $\o$. This proves that the symplectic rank 
of  $(\gg ,\o )$  is zero for all symplectic forms $\o$ on $\gg$. \end{Ex} 

It is not a coincidence that the above example splits has a semi-direct product 
with respect to a non-degenerate abelian ideal. Indeed, we note the following preliminary structure result for irreducible symplectic Lie algebras, which is a direct consequence of \cite[Th\'eor\`eme 1.3] {DM2}:

\bp \label{prop:irreducible} 
Let $(\gg, \omega)$ be a symplectic Lie algebra, which is irreducible over the reals.  Then  $\ga = [\gg, \gg]$ is a maximal abelian ideal of $\gg$, and it is non-degenerate with respect to the form $\omega$. The abelian subalgebra $\gh = \ga^\perp$ acts faithfully on $\ga$ and $\ga$ decomposes as an orthogonal direct sum of two-dimensional irreducible subspaces. 
\ep 
%\pf Since $(\gg, \omega)$ is symplectic, $\gg \neq [\gg, \gg]$ (see \cite[p.\ 250]{MR1}).
%Therefore, $\gg$ has a non-trivial abelian ideal $\ga$. As for any ideal, $\gh = \ga^{\perp_{\omega}}$ is a subalgebra. Moreover, since $\ga$ is abelian, 
%$\ga \cap \ga^{\perp_{\omega}}$ is an (isotropic) ideal in $\gg$. 
%Hence, by irreducibility of $(\gg, \omega)$, we have $\ga \cap \ga^{\perp_{\omega}} = \{0 \}$
%\epf 

Let $\gh$ be an abelian algebra and $V$ a
module for $\gh$ over the reals. 
We call $V$ a \emph{purely imaginary 
module}  if all eigenvalues for the action of $\gh$ are purely imaginary. The following gives a strengthening of Proposition \ref{prop:irreducible} 
and also a converse.  
The result essentially classifies all irreducible symplectic Lie algebras
by associating to an abelian algebra $\gh$, $\dim \gh = 2k$, a set of real characters $\lambda_{i}$, $i = 1, \ldots, m$, $m \geq 2k$,
which satisfy certain conditions (see Corollary \ref{cor:irreducible} below).

\bt \label{thm:irreducible} 
Let $(\gg, \omega)$ be a real symplectic Lie algebra and let
$\ga = [ \gg, \gg]$ denote the commutator ideal of $\gg$.
Then $(\gg, \omega)$ is irreducible over the reals if and only 
if the following conditions hold:
\begin{enumerate}
\item 
The commutator ideal $\ga$ is
a maximal abelian ideal of $\gg$, which is non-degenerate with respect to $\omega$.
\item   The symplectic Lie algebra $(\gg, \o)$ is an orthogonal semi-direct sum of an 
abelian symplectic subalgebra $(\gh, \omega) $ and the ideal 
$([\gg, \gg], \omega)$. %  where $\gh$ acts faithfully on $\ga$.
\item The abelian ideal $(\ga,\o)$ decomposes into an orthogonal sum of two-dimensional  purely imaginary irreducible 
submodules for $\gh$,  which are
pairwise non-iso\-mor\-phic. 
\end{enumerate}
\et 
\pf For the first part of the proof let us assume that the symplectic Lie 
algebra  $(\gg, \omega)$ satisfies 1.\ and 3. 
(Note that condition 2.\ is implied by 1.)
% the assertions of Proposition \ref{prop:irreducible}. 
%We now discuss the structure of 
%$(\gg, \omega)$ in more detail. 
Put $\dim \gh = 2k$ and let $\ga = \oplus_{i = 1}^{m} \ga_{i}$ be a decomposition of $\ga$ into $\gh$-invariant non-degenerate and mutually orthogonal irreducible subspaces. Such a decomposition exists, by condition 3. Since the irreducible subspaces $\ga_{i}$ are all two-dimensional, there exists a complex structure $J$ on the vector space $\ga$ such that $\gh$ acts complex linearly on $\ga$. 
With respect to $J$, 
the irreducible subspaces $\ga_{i}$ are complex subspaces 
of $\ga$ and the restriction of $J$ to $\ga_{i}$ is uniquely defined  up to sign.
Since $\gh$ acts by purely imaginary transformations  on $\ga_{i}$, there exist non-zero characters $\lambda_{i}: \gh \ra \mathbb{R}$, $i = 1, \ldots, m$, such that, for all $v = \sum_{i=1}^{m} v_{i}$, $v_{i} \in \ga_{i}$ and, for all $H \in \gh$, we have 
$$[H, v ] \, =  \, \sum_{i=1}^{m} \lambda_{i}(H) J v_{i} \; . $$
Since $\omega$ is symplectic for $\gg$ and $\ga$ is abelian, $\gh$ necessarily acts by transformations on $\ga$,  which are skew with respect to $\o$.
In particular, we  see that the symplectic form $\omega$ on $\ga$ is preserved by the complex structure $J$. Since $\ga$ is maximal abelian,  the action of $\gh$ on $\ga$ is faithful. We also deduce that $m \geq 2k$. 

%Now we show that every Lie algebra $(\gg, \omega)$, which  satisfies the conclusion of  Proposition \ref{prop:irreducible} and
%has mutually distinct characters $\lambda_{i}$, 
%is symplectically irreducible. 

Now we can show that $(\gg, \omega)$ is irreducible.
Indeed, let
$\gj$ be an ideal of such a $\gg$. Then $[\gg, \gj]$ is an 
$\gh$-invariant subspace of $\ga$, which decomposes as a direct sum of irreducible subspaces for $\gh$. These are also irreducible subspaces for the action of $\gh$ on $\ga$. Since these are mutually non-isomorphic,  they coincide with some of the spaces $\ga_{i}$, which are non-degenerate and mutually
orthogonal. Assuming that $\gj$ is isotropic this shows that $[\gg, \gj] = \{ 0 \}$. That is, $\gj$ must be central. However, since $\gh$ acts faithfully, the center of $\gg$ is trivial. Hence, $(\gg, \o)$ is symplectically irreducible. 

For the second part of the proof consider now $(\gg, \omega)$,  which is irreducible. By Proposition \ref{prop:irreducible}, $(\gg, \omega)$ satisfies 1., 2.\ and, as is easily seen, also the first part of 3., which asserts that $\gh$ acts purely imaginarily on its irreducible subspaces.  Our next claim is that the characters $\lambda_{i}$, which as above belong to the action of $\gh$, are mutually distinct. Suppose to the contrary that $\lambda= \lambda_{i} = \lambda_{j}$, where $i \neq j$. Then 
$(\ga_{i} \oplus \ga_{j}, \omega)$ is a non-degenerate subspace and $H \in \gh$ acts by imaginary scalar-multiplication $v \mapsto \lambda(H) J v$ on $\ga_{i} \oplus \ga_{j}$. 
Evidently (compare the proof of Lemma \ref{lem:Lag_irr}), since $\omega$ is $J$-invariant, $\ga_{i} \oplus \ga_{j}$ contains a two-dimensional isotropic and $J$-invariant subspace. A fortiori, this subspace is also $\gh$-invariant and therefore is a non-trivial isotropic ideal of $(\gg, \omega)$. Since this contradicts the irreducibility of $(\gg, \omega)$, all the characters $\lambda_{i}$ must be mutually distinct. Hence, 3.\ is satisfied. Moreover, it is implied that every $\gh$-invariant subspace of $\ga$ that is irreducible coincides with one of the $\ga_{i}$, and, therefore, is non-degenerate with respect to $\omega$. 

\epf 

Let $\gh = \gh_{k}$ be an abelian Lie algebra of dimension
 $ \dim \gh = 2k \geq 2$ over the reals.
Let $\lambda_{i}: \gh \ra \mathbb{R}$, $i =1, \ldots, m$, be a set
of characters. 
% $m \geq 2k$ be a set of mutually distinct characters, which span $\gh^*$.
Let $V$ be a complex vector space of dimension $m$ and denote with $J : V \ra V$ the complex structure. Let 
$e_{i}$, $i =1, \ldots, m$, be a complex basis for $V$ 
with $\ga_{i} = \spanof{e_{i}}_{\mathbb{C}}$. 
Then,   for all $H \in \gh$ and $v = \sum_{i=1}^{m}  v_{i}$, $v_{i} \in \ga_{i}$, define 
\begin{equation} \label{eq:HJ}
  [H, v ]  \,  =  \,  \sum_{i=1}^{m} \lambda_{i}(H) J v_{i} \; . \end{equation}
Furthermore, we let 
$$ \gg_{k,\lambda_{1}, \ldots, \lambda_{m}}  = \; \, \gh_{k} \oplus V$$  be the real Lie algebra of dimension $2k + 2m$, which is the semi-direct product of the abelian Lie algebras $\gh_{k}$ and $V$ such that  $H \in \gh_{k}$ acts on $V$ by the above bracket \eqref{eq:HJ}. 
\begin{Cor}[Classification of irrreducible symplectic Lie algebras]
\label{cor:irreducible}
Let $(\gg, \omega)$ be an irreducible symplectic Lie algebra.
Then $\gg$ is isomorphic to a Lie algebra $\gg_{k, \lambda_{1}, \ldots , \lambda_{m}}$, where $k \geq 1$, $m \geq 2k$ and $\lambda_{1}, \ldots, \lambda_{m}$ is a set of mutually distinct non-zero characters that  span $\gh^*$. 
\end{Cor}
\pf Observe that since $\gh$ acts faithfully on $\ga$, the
set of characters   $\lambda_{1}, \ldots , \lambda_{m}$ spans $\gh^*$. Therefore, the result is a consequence of Theorem \ref{thm:irreducible} and its proof.
\epf 

Of course, up to isomorphism of Lie algebras, the characters $\lambda_{1}, \ldots, \lambda_{2k}$ may be chosen to form a any given standard basis for the dual space $\gh_{k}^*$. 

\subsubsection{Lagrangian subalgebras} \label{sect:irr_Lag}
In the following results we analyze the structure of Lagrangian subalgebras in irreducible symplectic Lie algebras. We start with an important example, which shows that, in general,  irreducible symplectic Lie algebras may have Lagrangian subalgebras.\footnote{\label{fn:Lag_alg2} Note (compare the first part of the proof of Proposition \ref{prop:Lag_irreducible}) that the Lie algebras $(\gg, \omega)$ in Example \ref{ex:irreducible} admit the structure of a (flat) K\"ahler Lie algebra in the sense of \cite{DM2}. The existence of the  Lagrangian subalgebra $\gb$ thus furnishes a counterexample to the assertion of \cite[Th\'eor\`eme 3.8]{DM2},  where it is mistakenly claimed that $\gg$ has no Lagrangian subalgebra.}

\bl \label{lem:Lag_irr} 
Let $\gg$ be the Lie algebra defined in Example \ref{ex:irreducible}.  Then, for every symplectic form $\omega$ on $\gg$, there exists a subalgebra $\gb$ of $\gg$ which is Lagrangian with respect to $\omega$.\el 
\pf  
 By the above $(\ga = [\gg, \gg], \omega) = (\ga_{1}, \omega) \oplus (\ga_{2}, \omega)$ is a non-degenerate ideal of $(\gg, \omega)$. Therefore, we
 have an orthogonal direct sum $(\gg, \o) = (\gh, \o) \oplus (\ga, \o)$, where $\gh$ is an abelian subalgebra of $\gg$. We find a basis as in Example \ref{ex:irreducible} such that $\gh = \spanof{e_{5}, e_{6}}$, $\ga_1=\spanof{ e_1,e_2}$ and
$\ga_2 = \spanof{ e_3,e_4}$. We can assume that
$\omega_{12}, \omega_{34} \in \{ \pm 1\} $.

Define  $H =  \omega_{34} \, e_{5} - \omega_{12} \, e_{6}$, 
$X = e_{1} + e_{3}$ and $Y = -  \omega_{34}\,  e_{2} +  \omega_{12} \, e_{4}$.  Then $\ga = \spanof{H,X,Y}$
is a subalgebra of $\gg$, which is Lagrangian with respect to $\omega$.
Indeed, $[X, Y ] = 0$, $[H, X ] = -  \omega_{34} \, e_{2} + \omega_{12}  \, e_{4} = Y$, $[H, Y ] = 
[H, [H, X ]] = - X$. Therefore, $\ga$ is a three-dimensional subalgebra. It remains to show that $\ga$ is isotropic with respect to $\omega$. Since the ideals $\ga_{1}$ and $\ga_{2}$ are orthogonal, we obtain 
$\omega(X,Y) = -\omega_{12}  \omega_{34} +  \omega_{34} \omega_{12} = 0$. Since $\gh$ is
orthogonal to $\ga$, $\omega(H,X) = \omega(H,Y)  = 0$. This shows that $\ga$ is isotropic. 
\epf 

The following proposition sheds some light on the general structure of Lagrangian subalgebras in irreducible 
symplectic Lie algebras.

\bp \label{prop:Lag_irreducible} 
Let $\gl$ be a Lagrangian subalgebra of an irreducible symplectic Lie algebra $(\gg, \omega)  =  (\gh, \omega) \oplus (\ga = [\gg, \gg], \omega)$. Let $\gb = \gl \cap \ga$ and 
$\gi \cong \gl/ \gb$  the subalgebra of $\gh$, which is the image of $\gl$ under projection to the first direct summand $\gh$ of $\gg$. Write $2 m = \dim \ga$, $ 2k = \dim \gh$. Then the following relations are satisfied: 
\begin{enumerate}
\item $ \dim \gl = m + k$.
\item $ m \, \geq \; \dim \gb \; \geq \, 2k$.
\item $  2 m - {\dim \gb} \, \geq \; 2 \dim \gi \, = \, 2 \, (m + k - \dim \gb)$.
\item $ k \,  \geq \; \dim \gh \cap \gl  \; \geq \, k + \dim \gb - m$.
\item $ 2k - \dim \gh \cap \gl \geq \dim \gi$.
\end{enumerate}
\ep 

\pf 
We use the notation as in the proof of Theorem \ref{thm:irreducible}. Let us further put
$\g(u,v) = \omega(Ju, v)$, for all $u, v \in \ga$. Then $\g$ defines a \emph{positive definite} 
$J$-hermitian inner product on $\ga$, such that 
$\gh$ acts by $\g$-skew maps. Indeed, it follows 
that the image of $\gh$ is a purely imaginary torus in the unitary Lie algebra $\gu(\g,J)$. This torus is diagonalized with respect to
the decomposition $\ga = \oplus_{i = 1}^{m} \ga_{i}$. 

With this remark in place we can turn to the proof of the proposition.
Since the kernel $\gb = \gl \cap \ga$ of the natural map $\gl \ra \gi$ acts trivially on $\gb$, the subalgebra $\gi$ of $\gh$ acts on the isotropic subspace $\gb$ of $\ga$. Since 
$\gamma$ is positive definite and $\gb$ is isotropic, we 
easily deduce that $\gb \cap J \gb = \{ 0 \}$ and, moreover, $\gb \oplus J \gb$ 
is a non-degenerate subspace with respect to $\o$. 
It follows that there is an $\gi$-invariant isotropic decomposition 
\begin{equation} \label{eq:bWjb}
\ga  = \gb \oplus W \oplus J \gb \, , \, \; \text{where } W =  \gb^{\perp_{\o}} \cap  J \gb^{\perp_{\o}} \; , 
\end{equation} 
 % where $W =  \gb^{\perp_{\o}} \cap  J \gb^{\perp_{\o}}$.
and  $\dim W = \dim \ga - 2 \dim \gb = 2m -  2 \dim \gb $. 
We further have $\dim \gg = 2k + 2m$, and,
by Corollary \ref{cor:irreducible}, 
$m \geq 2k$. Since $\gl$ is Lagrangian, this implies 
$$ \dim \gl = k + m \; \text{ and } \;  \dim \gi =  k + m - \dim \gb \; .  $$
Since $\gi$ acts purely imaginarily on 
$\gb$ and faithfully on $\ga$, we deduce from the 
decomposition \eqref{eq:bWjb} that 
$$  \dim \gi =  k + m - \dim \gb   \; \leq  \; {\dim \gb \over 2} + {\dim W \over 2}  =  m -  {\dim \gb \over 2}  \; .$$
We infer that 
$\dim \gb \geq  2k$, as well as $\dim W \leq 2 m- 4k$. 
Note that up to now we have established relations 1.-3.

Every complementary subspace $\bar \gi$ of $\gb$ in $\gl$ must be orthogonal to  $\gb$, since $\gl$ is isotropic. Therefore, we may choose such  $\bar \gi$ of the form $$ \bar \gi = \{ H + \tau(H) \mid H \in \gi \} \; , $$ for some
$\tau: \gi \ra W$. Moreover,
 $$ \ker \tau = \gh \cap \gl = \gi \cap \gl $$ is an isotropic
subalgebra of $\gh$. 
Now we also have $$ k \geq \dim \ker \tau \geq \dim \gi - \dim W = 
k + m - \dim \gb - \dim W = k -m + \dim \gb \; . $$ 
This proves the fourth relation. 

Finally, observe that $\gi$ is contained in the orthogonal of $\gh \cap \gl$ in $\gh$, since $\gl$ is isotropic. This implies $\dim \gi \leq 2k - \dim \gh \cap \gl$.  Therefore, the last part of the proposition follows. 
\epf

In particular, if $m$ is minimal with respect to $k = \dim \gh$, that is,  if $m = 2k$ then Lagrangian subalgebras of $(\gg, \omega)$ decompose into Lagrangian subalgebras of $\gh$ and $\ga$ respectively.

\bc \label{cor:Lag_irreducible} 
Let $\gl$ be a Lagrangian subalgebra of an irreducible symplectic Lie algebra $(\gg, \omega)  =  (\gh, \omega) \oplus (\ga = [\gg, \gg], \omega)$, which satisfies $\dim \ga = 2 \dim \gh$. Then there exist Lagrangian subspaces $\gi \subset (\gh, \o)$ and $\gb \subset (\ga, \o)$ such that 
$[\gi , \gb] \subseteq \gb$ and, moreover,
$ \gl  = \gi \oplus \gb $ is a 
semi-direct sum of subalgebras.
\ec 
\pf  The result follows directly from Proposition \ref{prop:Lag_irreducible}, where we have $m= 2k$.
Indeed, from 2.\ we infer that $\dim \gb = m = 2k$, and therefore
$\gb$ is a Lagrangian subspace of $(\ga, \omega)$. 
The third relation then implies $\dim \gi = k$, and by 4.\ 
also $\dim  \gi \cap \gl = k$. Therefore, $\gi = \gi \cap \gl$ is
contained in $\gl$ and $\gi$ is a Lagrangian subalgebra of $\gh$.
The corollary now follows. 
\epf 

Finally, we use Proposition \ref{prop:Lag_irreducible} 
to show that there exist certain irreducible symplectic Lie algebras that can not have Lagrangian subalgebras. 

\begin{Prop}[Irreducible with no Lagrangian subalgebra] \label{prop:irr_noLag}
Let $\gh_{1}$ denote a two-dimensional abelian Lie algebra and let $$ \Lambda = (\lambda_{1}, \lambda_{2}, \lambda_{3}) : \;  \gh_{1}  \, \longrightarrow  \,  \mathbb{R}^3 $$
be a character map, which satisfies the conditions: \begin{enumerate}
\item[i)] the two-dimensional subspace $\im \Lambda$
does not meet the six points whose coordinates are 
permutations of $( 0, 1, -1)$. 
\item[ii)] the characters satisfy $\lambda_{i} \neq \pm \lambda_{j}$, for $i \neq j$. 
\end{enumerate}
Let $(\gg, \o)$ be an eight-dimensional irreducible symplectic Lie algebra with $\gg$ isomorphic to $\gg_{1, \lambda_{1}, \lambda_{2}, \lambda_{3}}$, where $\lambda_{1}, \lambda_{2}, \lambda_{3}$ are as above. Then $(\gg, \o)$ has no Lagrangian subalgebra. \end{Prop}
\pf 
It is easy to see that there exist characters $\lambda_{1}, \lambda_{2}, \lambda_{3}$ that satisfy 
the assumptions i) and ii). Also by Theorem \ref{thm:irreducible} there do exist irreducible symplectic Lie algebras $(\gg,\o)$ with $\gg = \gg_{1, \lambda_{1}, \lambda_{2}, \lambda_{3}}$. 
In the notation of Proposition \ref{prop:irreducible} these are of type $k= 1$ and $m = 3$.  

Assume now that there exists a Lagrangian subalgebra $\gl$ of 
$(\gg, \omega)$. Then $\dim \gl =4$. Consider the subalgebras 
$\gb = \gl \cap \ga$ and $\gi \cong \gl / \gb$ as in Proposition \ref{prop:Lag_irreducible}. We infer from Proposition \ref{prop:Lag_irreducible} that $3 \geq \dim \gb \geq 2$.

In the case $\dim \gb = 3$, we have $\dim \gi = 1$ and we deduce from 4.\ of Proposition \ref{prop:Lag_irreducible} that also $\dim \gi \cap \gl = 1$. Therefore, $\gi =  \gh \cap \gl$ is a Lagrangian 
subspace of $(\gh,\o)$ and $\gl= \gi \oplus \gb$. Since $\gb$ is Lagrangian in $(\ga, \o)$, there is an $\gi$-invariant 
isotropic decomposition 
% \begin{equation} \label{eq:bWjb}
$$
\ga  = \gb \oplus J \gb \, . 
$$ 
% \end{equation} 
Let $\gi = \spanof{H}$. Since $H$ preserves $\gb$, its action on
$\ga$ is the complexification of its restriction to $\gb$. Also, since 
$H$ has only imaginary eigenvalues except $0$, its complex eigenvalues on $\ga$ are $0$ and $\pm \lambda J$, for some $\lambda \in \mathbb{R}$. After scaling $H$ and rearranging the characters if necessary, we may thus 
assume that $\lambda_{1}(H) = 1 = -\lambda_{2}(H)$ and
$\lambda_{3}(H) = 0$. However, by assumption i) on
the $\lambda_{i}$, this is not possible. Therefore, $\dim \gb \neq 3$.

In the remaining case $\dim \gb = 2$, we necessarily have $\dim \gi = 2 = \dim \gh$  and $\gh \cap \gl = \{ 0 \}$. Since $\dim \gb = 2$, there is an $\gh$-invariant isotropic decomposition 
% \begin{equation} \label{eq:bWjb}
$$
\ga  = \gb \oplus W \oplus J \gb \, , \, \; \text{where } W =  \gb^{\perp_{\o}} \cap  J \gb^{\perp_{\o}} \; , 
$$ 
% \end{equation} 
 % where $W =  \gb^{\perp_{\o}} \cap  J \gb^{\perp_{\o}}$.
and $\dim W = 2$. Since $W$ is $J$-invariant, it is one of the complex irreducible subspaces $\ga_{i}$ of $\ga$, say $W = \ga_{1}$. Moreover, since $\gb$ defines a real structure for
the complex vector space $\gb \oplus J \gb$, it follows as above that $\lambda_{2} = - \lambda_{3}$. This contradicts 
assumption ii) on the characters $\lambda_{i}$. 
Therefore, $(\gg, \o)$ has no Lagrangian subalgebra.
 \epf  

%%%%%%%%%%%%%%%%%%%%%%%%%%%%%%%%%

\part{Symplectic Lie groups of cotangent type}

%%%%%%%%%%%%%%%%%%%%%%%%%%%%%%%%%%
\section{Symplectic geometry of cotangent Lie groups} 
\label{sect:cotangent_groups}
Before developing the Lagrangian symplectic extension theory  in a systematic way on the Lie algebra level in the following Section 
\ref{sect:Lagrange_ext}, we clarify the relation of Lagrangian extensions of Lie groups with symplectic Lie group structures on the cotangent bundle. We show that all extensions of Lie groups to symplectic Lie group structures on the cotangent bundle (which are assumed to be symplectic with respect to the standard symplectic geometry on the cotangent bundle) arise from Lagrangian extensions of flat Lie groups. Special cases of this problem have already been considered in \cite{Chu} and \cite{Lic}. 
As an application of the extension construction we also derive that 
every flat Lie group may be realized as a Lagrangian subgroup 
in a symplectic Lie group. 

\subsection{Cotangent Lie groups} \label{sect:cotangentLGs}
Let $G$ be a Lie group with Lie algebra $\gg$ and its
dual vector space $\gg^*$. The elements of $\gg^*$ are 
henceforth considered to be left-invariant one-forms on $G$. Let
$T G$ be the tangent bundle of $G$ and  let $T^*G$ denote the
cotangent bundle. If $\alpha$ is a one-form on $G$,  we denote with $$ \alpha_{g} \in \, T^{*}_{g} G = (T_{g} G)^*$$  its value 
% at the cotangent space of $G$ 
at $g \in G$. 
Since $G$ is a Lie group, we have the following natural trivialization of the cotangent bundle.

\bl \label{lem:gs1}
The map $G \times \gg^* \ra T^* G$ such that
 $(g, \lambda) \mapsto \lambda_{g}$ is a diffeomorphism
which commutes with the natural projections onto $G$. 
\el 

Let $\varrho: G \ra \GL(\gg^*)$ be a representation. We form the
semi-direct product Lie group 
$$ G \ltimes_{\varrho} \gg^*$$  by defining
the group product on the manifold $G \times \gg^*$ as 
\begin{equation}  \label{eq:ct1}
(g, \alpha) \cdot (h,\beta)  =  \; (g h \, , \, \varrho(h^{-1}) (\alpha)  + \beta )  \;. 
\end{equation}
There is thus a natural exact sequence of Lie groups 
$$  0 \ra \gg^* \ra G \ltimes_{\varrho} \gg^* \ra G \ra 0 $$
which splits by the homomorphism $g \mapsto (g, 0)$. 
Then the following holds (compare \cite{Lic}): 

\bl \label{lem:gs2} 
 Let $\alpha, \beta \in \gg^*$ be left-invariant one-forms. 
Under the identification map $G \times \gg^* \ra T^* G$ the
group product \eqref{eq:ct1} transported to $T^* G$ satisfies 
\begin{equation}   \label{eq:ct2}
\alpha_{g} \cdot \beta_{h} \; =  \;  (\varrho(h^{-1})  (\alpha)  + \beta)_{gh} \; . 
\end{equation}
The splitting $G \ra G \times \gg^* \ra T^* G$ corresponds to the zero-section $G \ra T^*G$ of the cotangent bundle. Moreover, for all $g \in G$, the left-multiplication $$ L_{g}: \, G \times \gg^* \, \ra  \, G \times \gg^*$$  
corresponds to the natural extension of the left-multiplication diffeomorphism $L_{g}: G \ra G$ on the cotangent bundle. 
Finally, the right-multiplication of elements in $G$ on $T^*G$  arises as the natural extension from $G$ if and only if $\varrho$ is the coadjoint representation. 
\el 
\pf 
Let $\Phi: G \ra G$ be a diffeomorphism.  For any one-form $\alpha$, the natural cotangent action $\Phi^*: T^* G \ra T^* G$ of $\Phi$ is expressed by the relation $$ \Phi^*  (\alpha_{h}) \,  {(d\Phi^{-1}})_{|_{h}}   \, = \,  \alpha_{h}  \, . $$ 
Equivalently, $\Phi^* \, (\alpha_{h} )  =  (\Phi^* \alpha)_{\Phi^{-1} h}$.
If $\alpha$ is left-invariant, and $L_{g}: G \ra G$ is the left-multiplication diffeomorphism, we thus have 
$$L_{g}^* \, (\alpha_{h} )  =  \alpha_{g^{-1}h} . $$ 
Correspondingly, by \eqref{eq:ct2}, left-multiplication with $0_{g} \in T^* G$ 
satisfies $$ 0_{g} \cdot \alpha_{h} = \alpha_{gh} . $$ 
Let $R_{g}: G \ra G$ be right-multiplication then similarly we deduce 
$$R_{g}^* \, (\alpha_{h} )  =   ((\mathrm{Ad}(g^{-1}))^* \alpha)_{h g^{-1}} . $$ 
Which compares to 
$$  \alpha_{h} \cdot   0_{g}  =  (\varrho(g)^{-1}  \alpha)_{hg} \; . $$
\epf
 
\paragraph{Natural symplectic structure on $T^*G$.}
As for any cotangent bundle there is a canonical one-form on $T^*G$, which, for all vector fields $X^*$ on $T^*G$  
satisfies $$ \theta_{\alpha_{g}} (X^*_{\alpha_{g}}) \, = \; 
 \alpha_{g}\left(\pi_{*}\left(X^*_{\alpha_{g}}\right)\right) . $$
Its derivative $$ \Omega = d Ê\theta $$
defines the natural symplectic structure on $T^*G$, which
is thus computed by the formula 
\be   \label{eq:omega1}
\Omega(X^*, Y^*) \; = \;   X^* \theta(Y^*) - Y^* \theta(X^*) - \theta([X^*, Y^*])  \; . 
\ee 

\paragraph{Evaluation on left-invariant vector fields.}
Note that left-invariant vector fields on the 
Lie group $G \ltimes_{\varrho} \gg^*$ project to left-invariant
vector fields on $G$, since the natural bundle projection $\pi: T^*G \ra G$ is a homomorphism of Lie groups.  
If $X^*$ is left-invariant, we denote 
with $X = \pi_{*}(X^*)$ its projection to a vector field 
on $G$. Let $\rho: \gg \ra \End(\gg^*)$ be the derivative of
the representation $\varrho: G \ra \GL(\gg^*)$ at the identity.
Since $\theta$ and $\Omega$ are preserved by the
cotangent action of diffeomorphisms of $G$, it follows from
Lemma \ref{lem:gs2} that both forms
are invariant by the left-multiplications with elements $g = (g, 0) 
\in G \ltimes_{\varrho} \gg^*$. A precise formula is contained in the following proposition: 

\bp \label{prop:omega}
For all  left-invariant vector 
fields $X^*,Y^*$ on the Lie
group $G \ltimes_{\varrho} \gg^*$,
there exist linear maps 
$\lambda, \mu: \gg \ra \bR$ such that, we have 
\begin{equation}   \label{eq:omega2}
\begin{split}   
\Omega(X^*,Y^*) (\alpha_{g}) \, = \, & \left(\rho(-X) \alpha\right)\, (Y) \;  -   \, (\rho(-Y) \alpha) \, (X) \,  - \, \alpha([X,Y]) \, \\  &  +      
 \lambda(Y) - \mu(X) \; .  
\end{split}
\end{equation}
\ep 
Let $\tau: \gg \ra \End(\gg)$ be the dual representation for $\rho$, so that,  for all $X, Y \in \gg$, we have 
$$  \alpha \left( \, \tau(X) (Y) \, \right)  \, =  \, (\rho(-X) \alpha) \, (Y) \, .$$   
Then $\tau$ defines a left-invariant flat connection $\nabla = \nabla^{\rho}$ on the Lie group $G$ by declaring 
\begin{equation} \label{eq:cotangentcon}
 \nabla^{\rho}_{X} Y =  \tau(X) (Y) \; ,
\end{equation} 
for all left-invariant vector fields $X, Y \in \gg$.\\ 

The following result can be seen as a generalization of a theorem of Chu % 's theorem
 \cite[Theorem 5]{Chu}, which asserts that for the coadjoint representation $\rho = {\mathrm{Ad^*}}$, $\Omega$
defines a left-invariant symplectic form on $G \ltimes_{\mathrm{Ad^*}} \gg^*$
if and only if $G$ is abelian. 
\bp \label{prop:omegali}
The canonical symplectic form $\Omega$ on $T^*G$ is left-invariant for the Lie group structure defined by the group  law \eqref{eq:ct2} if and only if the flat connection $\nabla^{\rho}$ is torsion-free.
\ep  
\pf Indeed, $\Omega$ is left-invariant, if and only if, for all $\alpha_{g} \in T^*G$, 
$$ \Omega(X^*,Y^*) (\alpha_{g})  \, = \, \Omega(X^*,Y^*) (0_{1}) , 
$$
for all
left-invariant vector fields $X^*, Y^*$. By formula \eqref{eq:omega2} this is the case if and only if, for all $\alpha \in \gg^*$, we have
$$   \left(\rho(-X) \alpha\right)\, (Y) \;  -   \, (\rho(-Y) \alpha) \, (X) \,  - \, \alpha([X,Y]) \, =  \, 0 .$$
Equivalently,  $\alpha\left( \, \tau(X) (Y) \,- \, \tau(Y) (X) \, -  [X,Y] \, \right) \, =  \, 0 $. 
\epf 

\begin{proof}[Proof of Proposition \mbox{\ref{prop:omega}}]
Let $\varphi$ be the one-parameter group on $T^*G$ tangent to the left-invariant vector field $X^*$. Using the coordinates introduced in Lemma \ref{lem:gs1}, we can write $\varphi(s) = a(s)_{\exp s X}$, where $a$ is a curve in $\gg^*$, with $a(0) = 0$. 
Similarly,  we can write $\psi(s) = b(s)_{\exp s Y}$, for the 
one-parameter group tangent to $Y^*$. Observe next that 
$$\theta(\alpha_{g})(X^*_{\alpha_{g}}) = \alpha_{g}(X_{g}) = \alpha(X) \; , $$
 for all $ \alpha_{g} \in T^* G $, where $\alpha \in \gg^*$. 
The curve $$ \alpha_g \cdot \varphi(s) = (\varrho( (\exp sX)^{-1})\, \alpha + a(s))_{g \exp sX}$$  is tangent to $X^*$ at $\alpha_{g}$.
Therefore, 
\begin{equation*} 
\begin{split}
 \left(X^* \theta(Y^*) \right)_{\alpha_{g}}  & = 
 \left.{\partial \over \partial s}\right|_{s=0}  \theta(Y^*)_{\alpha_g \cdot \varphi(s)}  \\
 &  =  \left.{\partial \over \partial s}\right|_{s=0} \left( \varrho( (\exp sX)^{-1} \alpha) (Y) +  a(s) (Y) \right) \\
 & =  \left(\rho(-X) \alpha\right)\, (Y) +  \lambda(Y) \; .
\end{split}
\end{equation*}
Analogously, we compute  $ \left(Y^* \theta(X^*)\right)_{\alpha_{g}} = 
  (\rho(-Y) \alpha) (X) + \mu(X)$. Thus \eqref{eq:omega2} follows by inserting into \eqref{eq:omega1}. 
\end{proof}

\bp\label{prop:Omega} The Lie algebra of the Lie group 
$T^*_{\varrho} G$ 
with the group structure \re{eq:ct2} identifies with 
$\gg \ltimes_\rho \gg^*$, where $\gg$ is the Lie algebra of $\gg$ and $\rho$ is the derivative of $\varrho$. 
The evaluation of the canonical
symplectic form $\Omega$ of $T^*G$ at the tangent space of the 
neutral element of the Lie group $T^*_{\varrho}G$ % with the group structure \re{eq:ct2}
coincides with the (skew-symmetric extension of) the canonical pairing between $\gg$ and $\gg^*$. In particular, if $\Omega$ is
left-invariant for the Lie group $T^*G$ then $G$ is a Lagrangian subgroup in the symplectic Lie group $(T^*G, \Omega)$. 
\ep

\begin{proof}   Evaluating \re{eq:omega2} at the neutral element we have $\Omega (X^*,Y^*)(0_1)= \l (Y) -\mu (X)$. 
This is zero if $X^*, Y^*\in \gg^* \subset  \mathrm{Lie}\, T^*G = \gg \ltimes_\rho \gg^*$, since then $X=Y=0$.
The description of $\l$ and $\mu$ in the proof of Proposition \re{prop:omega} shows that
$\l=0$ if $X^*\in \gg$ and $\mu=0$ if $Y^*\in \gg$. In fact, the one-parameter group $\varphi(s) = a(s)_{\exp s X}$
generated by $X^*=X\in \gg$ is given by the constant curve $a=0$ in $\gg^*$.  (Notice that the zero section 
$G \subset T^*G$ is a subgroup.) 
Thus $\l = a'(0)=0$, in this case. 
So we are left with the case $X^*=\a \in \gg^*$ and $Y^*=Y\in \gg$. Here we have  $\Omega (X^*,Y^*)(0_1)= \l (Y)=\a (Y)$, since 
$\l = a'(0)=\a$ is now computed from $a(s) = s \a$.  (Notice that $\gg^*= T_{0_1}^*G \subset T^*G$
is a subgroup.) 
\end{proof}

Recall that a Lie group $G$ with a flat and torsion-free left-invariant connection $\nabla$ is called a \emph{flat Lie group}. The \emph{affine holonomy} of $(G, \nabla)$ is the representation of the fundamental group of $G$, which comes from affine parallel transport along closed curves. The linear part of this representation is called the \emph{linear holonomy}. Every flat Lie group arises from an \'etale affine representation $\tilde G \ra \mathrm{Aff}(\mathbb{R}^{n})$, $n= \dim G$, of the universal covering group $\tilde G$ of $G$, see for example \cite[Section 5]{Ba}. The affine holonomy of $(G, \nabla)$ is trivial if and only if the associated \'etale affine representation of $\tilde G$ factors over $G$. The linear holonomy is trivial if and only if the linear part of the associated \'etale affine representation factors over $G$.

\bc \label{cor:cotangent_symplectic} 
There exists a Lie group structure on $T^*G$ with multiplication law of the form \eqref{eq:ct2} with respect to a representation $\varrho: G \ra \GL(\gg^*)$, which is symplectic for the canonical symplectic form $\Omega$ if and only if $G$ admits the structure of a flat Lie group with trivial linear holonomy. 
\ec 
\pf Indeed, if $(T^*_{\varrho }G, \Omega)$ is symplectic then Proposition \ref{prop:omegali} shows that $(G, \nabla^{\varrho})$ is a flat Lie group. Moreover, since $\nabla^{\varrho}$ is defined by \eqref{eq:cotangentcon}, we see that the linear part of the derivative of the \'etale affine representation belonging to $\nabla^{\varrho}$ is $\tau$. Since $\tau$ is the dual of the representation $\rho: \gg \ra \End(\gg^{*})$ which integrates to $\varrho$, the representation $\tau$ integrates to the dual representation for $\varrho$. This shows that 
$(G, \nabla^{\varrho})$ has trivial linear holonomy. We postpone the proof of the converse statement to the proof of Theorem \ref{thm:WeinsteinLG}. There we show that a flat Lie group with trivial holonomy has a cotangent extension $(T^*_{\varrho }G, \Omega)$.
\epf 

We call symplectic Lie groups of the form $(T^*G, \Omega)$ as above \emph{cotangent symplectic Lie groups}. Symplectic Lie groups of this type % Lagrangian extensions of flat Lie algebras 
provide  a rich
and easily accessible source of examples for symplectic Lie groups. This is illustrated by the following example, in which we construct nilpotent symplectic Lie groups of arbitrary large solvability degree:\footnote{As explained in the introduction, Example \ref{ex:uppertriang} disproves a conjecture (cf.\ \cite{Guan,Ovando}) that the solvability degree of a compact symplectic solvmanifold is bounded by three (or, more generally,
we disprove that there exists a  uniform bound at all).}

\begin{Ex}[Symplectic solvmanifolds with unbounded derived length] \label{ex:uppertriangquot}

Let $T_{n}$ be the nilpotent group of upper triangular 
$n\times n$ matrices with $1$ on the diagonal and 
$\gt_{n}$ its Lie algebra ($\gt_{n}$ is the Lie algebra 
of upper triangular matrices with zero diagonal). The Lie group 
$T_{n}$ acts simply transitively by left-multiplication on the affine 
subspace $T_{n }\subset \mathrm{Mat}(n,\bR)$, and therefore inherits a left-invariant torsion-free connection $\n$, since
the action is by affine transformations. Let us 
denote by $\varrho : T_{n} \ra \mathrm{GL}(\gt_{n}^*)$ 
the representation which is dual to the left-multiplication 
action of $T_{n}$ on $\gt_{n}$. 
Then we have $\nabla = \nabla^{\varrho}$, 
as defined in \eqref{eq:cotangentcon}, that is, $\varrho$
is the associated representation for the flat connection $\nabla$. (Compare also Example \ref{ex:uppertriang}.)
Let  $(G_{n}= T_{n }\ltimes_\varrho \gt_{n}^* \, , \Omega)$ be the associated  nilpotent cotangent symplectic Lie group. 
The nilpotent Lie group $G_{n}$ has a cocompact lattice $\Gamma$, since the structure constants of it Lie algebra are rational.  (For example, the semi-direct product of the subgroup $\Gamma_{n}$
of matrices in $T_{n}$ with integral entries with a suitable lattice in $ \gt_{n}^*$ will do.) The quotient manifold $$(\G \, \backslash \, G_{n} \, , \,  \Omega) $$ endowed with the induced symplectic structure is a compact symplectic nilmanifold (in particular, it is also a symplectic solvmanifold). Observe that the length $s_{n}$ of the derived series for $G_{n}$ is unbounded as a function of $n$ (in fact, $s_{2^m+1} = m+1$, cf.\ Example \ref{ex:uppertriang}.).  
%this shows that  there does not exist a uniform bound for the derived length of a symplectic solvmanifold. 
%This gives 
%a negative answer to Guan's question as soon as $n\ge 9$, that is,  if $\dim G \ge 18$. 
\end{Ex}

\subsection{Lagrangian subgroups and induced flat connection}
%  of symplectic Lie groups}
\label{sect:Lag_subgps}
Let $\mathcal L$ denote a Lagrangian foliation in a symplectic manifold $(M, \Omega)$. As for any foliation, the normal bundle for $\mathcal L$ carries a natural connection (Bott connection, see for example  \cite[Lemma 6.1.7]{CanCon}), 
which is flat in the directions tangential to $\mathcal L$. In case of a Lagrangian foliation in a symplectic manifold, this induces a torsion-free flat connection on any leaf $L$ of the foliation $\mathcal L$, cf.\  \cite[Theorem 7.7]{Weinstein1} and \cite{Dazord}. This  connection is called \emph{Weinstein connection} and it is defined by the formula 
\begin{equation} \label{eq:W_connection}
   \Omega(\nabla^{\mathcal L}_{X} Y , Z ) \; =  \; -  \Omega(Y, [X, Z]) + L_{X} (\Omega(Y,Z))  \; , 
\end{equation}
for all tangent vector fields $X,Y$ on $L$, and all sections 
$Z$ of $TM$. Weinstein \cite[Theorem 7.7]{Weinstein1} proved  
that every flat manifold, that is,  every manifold with a torsion-free flat connection $\nabla$,  appears as a closed leaf of a Lagrangian foliation of some symplectic manifold such that $\nabla$ is equal to the induced Weinstein connection.  
In the context of symplectic Lie groups we look at Lagrangian subgroups $L$ of symplectic Lie groups $(G, \Omega)$. Clearly,  the left-cosets of $L$ in $G$ define a left-invariant Lagrangian foliation ${\mathcal L}_{L}$ on $G$. We remark:

\bl Let $L$ be a Lagrangian subgroup of a symplectic Lie group $(G, \Omega)$. Then the Weinstein connection $\nabla^{{\mathcal L}_{L}}$ on $L$, which is induced by the foliation ${\mathcal L}_{L}$, is left-invariant and gives $L$ the structure of a flat Lie group.
\el 
\pf We may evaluate above \eqref{eq:W_connection} for left-invariant vector fields $X,Y$ on $G$, which are tangent to $L$. That is, $X,Y$ are contained in the Lie subalgebra $\gl$ of $\gg$ which is the Lie algebra of $L$. Then, for all $Z \in \gg$,  \eqref{eq:W_connection} reduces to 
\begin{equation}  \label{eq:W_connectionli}
   \Omega(\nabla^{\mathcal L}_{X} Y , Z ) \; =  \; -  \Omega(Y, [X, Z])  \; .
\end{equation}
Since $\Omega$ is left-invariant, this shows that $\nabla^{\mathcal L}_{X} Y$ is a left-invariant vector field and tangent to $L$. In particular, this implies that the connection $\nabla^{\mathcal L}$ is left-invariant. 
\epf 

Note that by  \eqref{eq:W_connectionli}, the restriction of the Weinstein connection $\nabla^{\mathcal L}$ to the Lie algebra $\gl$ defines a connection on $\gl$, which coincides with the canonical flat induced connection $\nabla^{\o}$ on the Lagrangian subalgebra $\gl$ of $(\gg, \o)$, as introduced in Proposition \ref{prop:tg}. In particular, the Weinstein connection on $L$ coincides with the left-invariant connection on $L$,  which is defined by $\nabla^{\o}$. \\

The following result is the analogue of Weinstein's theorem in the category of symplectic Lie groups.

\bt  \label{thm:WeinsteinLG}
Let $(H, \nabla)$ be a flat Lie group which has trivial {linear}  holonomy representation. Then there exists a symplectic Lie group $(G, \Omega)$  such that $(H, \nabla)$ is a Lagrangian (closed) subgroup of $(G, \Omega)$ and $\nabla$ coincides with the
induced flat Weinstein connection. In particular,
a simply connected Lie group $H$ admits the structure of a flat Lie group if and only if it is a simply connected Lagrangian subgroup of a symplectic Lie group.  
\et 
\pf  
Let $\rho: \gh \ra \End(\gh^{*})$ be the dual representation of $\tau: X \mapsto \nabla_{X}$. Since the linear holonomy representation for $\nabla$ is trivial, $\tau$ integrates to a representation of $H$. Therefore, $\rho$ integrates to a representation $\varrho: H \ra  \GL(\gh^{*})$. Let $(G= T_{\varrho}^{*} H, \Omega)$ be the cotangent symplectic Lie group with respect to $\varrho$  (that is, $G = H \ltimes_{\varrho} \gh^*$ is satisfying the group law \eqref{eq:ct2}, and it  is equipped with the canonical cotangent symplectic form $\Omega$). By Lemma \ref{lem:gs2}, 
$H$ is the zero-section in the cotangent space $T^{*} H$ 
and it is therefore a closed Lagrangian subgroup of $T^{*}_{\varrho} H$. Let $\nabla^{\rho}$ be the canonical left-invariant flat connection on $H$, which is attached to $\varrho$ by \eqref{eq:cotangentcon}. By duality, we observe that $\nabla^{\rho} = \nabla$. Therefore, $\nabla^{\rho}$ is torsion-free and, by Proposition \ref{prop:omegali}, $\Omega$ is a left-invariant symplectic form on $T^{*}_{\varrho} H$. 
Hence, $H$ is a closed Lagrangian subgroup of the cotangent symplectic  Lie  group $(T^{*}_{\varrho} H, \Omega)$. According to \eqref{eq:W_connectionli}, 
we have
$$   \Omega(\nabla^{\mathcal L}_{X} Y , \alpha ) \; =  \; -  \Omega(Y, [X, \alpha])  \; = \;  -  \Omega(Y, \rho(X) \alpha)   \; $$
for left-invariant vector fields $X, Y, \alpha$ on $T^{*}_{\varrho} H$, where  $X, Y \in \gh$ and $\alpha \in \gh^{*}$. By Proposition \ref{prop:Omega}, $\Omega$ induces the duality pairing between $\gh$ and $\gh^{*}$. Hence, we can deduce that $\nabla = \nabla^{\mathcal L}$ is the Weinstein connection.
\epf 

%%%%%%%%%%%%%%%%%%%%%%%%%%%%%%%%%
\section{Lagrangian extensions of flat Lie algebras} 
\label{sect:Lagrange_ext} 
In this section we describe in detail the theory of Lagrangian extensions of flat Lie algebras and its relation to Lagrangian reduction. Lagrangian extensions of flat Lie algebras generalize symplectic cotangent Lie groups as introduced in Section \ref{sect:cotangent_groups} on the infinitesimal level. 
% appear as a special case of these constructions. 
Our main observation is that the isomorphism classes of Lagrangian symplectic extensions of a flat Lie algebra 
$(\gh, \nabla)$ are parametrized 
by  a suitable restricted cohomology group $H^2_{L, \nabla}(\gh, \gh^*)$. This
describes the Lagrangian symplectic extension theory in a manner analogous to the Lie algebra case. 

\subsection{Lagrangian extensions and strongly polarized symplectic Lie algebras}
Let $(\gh ,\n )$ be a flat Lie algebra, that is, a 
Lie algebra endowed with a flat torsion-free connection $\n$. 
We explain how to construct symplectic Lie algebras $(\gg, \omega)$,
which have $\gh$ as quotient algebra arising in a Lagrangian reduction and $\nabla$ as the induced quotient flat connection, as described in Section~\ref{NormalSubsect}. These Lie algebras are called \emph{Lagrangian extensions}  of $(\gh ,\n) $. \\

Since $\n$ is a flat connection, the association $u \mapsto \nabla_{u}$ defines a representation $\gh \ra \End(\gh)$.  
We denote by $\rho: \gh \ra \End(\gh^*)$ the dual representation, which satisfies 
\begin{equation} \label{rhoEqu} \rho (u)\xi := -\n_u^*\xi = -\xi \circ \n_u , \quad  u \in \gh,\,  
\xi \in \gh^*.
\end{equation}
Define $Z^2_{\n} (\gh ,\gh^*) = Z^2_{\rho} (\gh ,\gh^*)$. Every cocycle $\a \in Z^2_{\n} (\gh ,\gh^*)$ thus gives rise 
to a Lie algebra extension 
$$ 0  \ra \gh^*   \ra \gg_{\n,\alpha} \ra \gh \ra 0 \, \; , $$
where the non-zero Lie brackets $ [ \, , \, ]_{\gg}$ for $\gg = \gg_{\n,\alpha}$ are defined 
on the vector space direct sum $\gg = \gh \oplus \gh^*$ by the formulas 
\begin{eqnarray} \label{eq:f1}
 [u,v]_\gg  & = & [u,v]_\gh + \a (u,v), \hfill  \quad\text{for all $u,v \in \gh$},  \\
 \label{eq:f2}
 \lbrack u,\xi  \rbrack_\gg & = &  \rho (u)\xi,  \hfill  \quad\text{for all $u \in \gh, \, \xi  \in \gh^*$}.
\end{eqnarray}
We let $\omega$ be the non-degenerate alternating two-form on 
$\gg$, which is defined by the dual pairing of $\gh$ and $\gh^*$. (Namely 
$\gh$ and $\gh^*$ are $\omega$-isotropic subspaces of $\gg$, and 
$ \o (\xi ,u) = - \o(u, \xi) = \xi (u)$, for all 
$ u  \in \gh$, $\xi \in\gh^*$.) 

\bp \label{prop:ext_condition}
The form $\o$ is symplectic for the Lie-algebra  $\gg_{\n,\alpha}$
if and only if 
\begin{equation} \label{aEqu} 
% \o (\a (u,v),w) + \o (\a (u,v),w) + \o (\a (u,v),w)=0 \; ,
\a (u,v) (w) +  \a (w,u) (v) +  \a (v,w) (u)=0 \; ,
\end{equation}
for all $ u,v,w \in \gh$. 
\ep
\pf It suffices to check that $\o$ is closed. Clearly, 
the equation \re{aEqu}
is equivalent to the vanishing of  $ \partial \o$ on $\gh \times \gh \times \gh$. Moreover,  $ \partial \o$ 
vanishes on $\gh^* \times \gh^* \times \gg$, since the subspace 
$ \gh^*$ is an abelian ideal in $\gg$, which is isotropic with respect to $\o$.  Finally, we obtain for
$u,v\in \gh$ and $\xi \in \gh^*$: 
\begin{eqnarray*} \o ([u,v],\xi ) +  \o ([v,\xi ],u) + \o ([\xi,u],v)
&=& -\xi ([u,v])- \xi (\n_v u) + \xi (\n_u v)\\ 
&=& \xi (T^\n (u,v)) \; \, .
\end{eqnarray*} 
Now the latter term vanishes, since $\nabla$ is torsion-free.
\epf 

It is customary (cf.\ \cite{Weinstein2}) to call a choice of a Lagrangian foliation in 
a symplectic manifold a \emph{polarization} of the manifold. This motivates the following terminology:   
\bd A polarization for a  symplectic Lie algebra $(\gg ,\o )$ 
is a choice of a Lagrangian subalgebra $\gl$ of $(\gg ,\o )$. 
A \emph{strong polarization} of a symplectic Lie algebra $(\gg ,\o )$ is a pair $(\ga , N)$ consisting 
of a   Lagrangian ideal $\ga \subset \gg$  
and a complementary Lagrangian subspace $N\subset \gg$.
The quadruple $(\gg ,\o , \ga , N)$ is then called a
\emph{strongly polarized symplectic Lie algebra}. An {isomorphism} 
of strongly polarized symplectic Lie algebras $(\gg ,\o , \ga , N)\ra 
( \gg' ,\o' , \ga' , N')$ is an isomorphism  
of symplectic Lie algebras $(\gg ,\o )\ra (\gg' , \o')$ which maps the
strong polarization $(\ga , N)$ to the strong polarization  $(\ga' , N')$. 
\ed  

With this language in place, we can summarize the above 
as follows:

\bt \label{thm:framed_symplectic}
Let $(\gh ,\n )$ be a \emph{flat Lie algebra}. 
To every  two-cocycle $\a \in Z^2_{\nabla} (\gh ,\gh^*)$ 
which satisfies  \eqref{aEqu} 
one can canonically
associate a strongly polarized symplectic Lie algebra 
$ F(\gh , \n ,\a ) := (\gg_{\nabla,\alpha} ,\o , \gh^* , \gh)$, whose
Lagrangian reduction has associated quotient flat Lie algebra  $(\gh ,\n )$.
%The Lie algebra  $\gg = \gg_{\rho ,\a }$ is defined as direct sum 
%of vector spaces $\gg = \gh +\gh^*$, where $\ga = \gh^*$ is an abelian ideal
%with the action of $\gg/\ga \cong \gh$ defined by $\rho=\rho_\n$, see \re{rhoEqu}, 
%and the Lie bracket on the subspace $U=\gh$ is defined by 
%\[ [X,Y]_\gg = [X,Y]_\gh + \a (X,Y),\quad\mbox{for all}\quad X,Y\in \gh.\]  
%The symplectic form $\o$ is defined by the natural pairing
%of $\gh^*$ with $\gh$:
%\[ \o (\xi ,X) = \xi (X),\quad \o (X,Y) = \o (\xi ,\eta )=0,\quad
%\mbox{for all}\quad X,Y\in \gh,\quad \xi, \eta\in\gh^*.\]   
%
\et  

\vspace*{1ex}
\noindent 
We call the symplectic Lie algebra $(\gg_{\nabla,\alpha} ,\o)$ 
the \emph{Lagrangian extension} of the flat Lie algebra $(\gh , \n)$ with respect to $\alpha$. 

\begin{Ex}
Note that every flat Lie algebra has at least one Lagrangian extension, using the zero-cocycle $\alpha \equiv 0$. This is the 
\emph{semi-direct product} Lagrangian extension 
$(\gh \oplus_{\nabla} \gh^*, \omega)$ or \emph{cotangent} Lagrangian extension. 
In particular, this 
shows that \emph{every flat Lie algebra arises as a quotient from a Lagrangian reduction}. Theorem \ref{thm:WeinsteinLG} proves that the simply connected symplectic Lie group $G$ with symplectic Lie algebra $(\gh \oplus_{\nabla} \gh^*, \omega)$ 
is indeed a cotangent symplectic Lie group $(T^*H, \Omega)$ as introduced in Section \ref{sect:cotangentLGs}.
\end{Ex} 

%Lagrangian extensions of flat Lie algebras provide  a rich
%and easily accessible source of examples for symplectic Lie algebras. This is illustrated by the following example, in which we construct nilpotent symplectic Lie algebras of arbitrary large solvability degree:\footnote{As explained in the introduction, Example \ref{ex:uppertriang} disproves a conjecture (cf.\ \cite{Guan,Ovando}) that the solvability degree of a compact symplectic 
%nilmanifold is uniformly bounded.}
%
We discuss now the series of nilpotent symplectic Lie groups $(G_{n}, \Omega_{n})$ constructed in Example \ref{ex:uppertriangquot} on the infinitesimal level. %  and give some more details: 

\begin{Ex}[Nilpotent with unbounded derived length] \label{ex:uppertriang}
Let $\gt_{n}$ denote the Lie algebra of upper triangular matrices 
with zero diagonal. This Lie algebra is nilpotent of class $n-1$. 
Its derived length (or solvability length) $s_{n}$ satisfies $s_{n} = \min\{ k \mid 2^k \geq n \}$. 
In particular, $s_{2^m} = 
s_{2^m-1} = m$, for $m \geq 2$,  and $s_{2^m+1} = m+1$. 
Hence, the family of Lie algebras $\gt_{n}$, $n \in \mathbb{N}$, has unbounded solvability degree.
Recall that the Lie algebra $\gt_{n}$ carries a natural torsion-free flat connection $\nabla$,  which arises from matrix multiplication. Indeed, for $A,B \in \gt_{n}$, this connection is 
defined by $$  \nabla_{A} B  \, = \, A \cdot B    \; .  $$
% Clearly, $\nabla$ is a torsion-free flat connection. 
Now let  $$ (\gg_n, \omega) =  (\gt_{n} \oplus_{\nabla} \gt_{n}^*, \omega)$$ 
be  the semidirect-product Lagrangian extension of the flat Lie algebra $(\gt_{n}, \nabla)$. This extension therefore is a nilpotent symplectic Lie algebra which has nilpotency class $n-1$. 
The surjective homomorphism $\gg_{n} \ra \gt_{n}$
shows that the solvability degree $s(\gg_{n})$ of $\gg_{n}$
is bounded from below by $s_{n}$. Also, it is easy to see that 
$s(\gg_{2^m+1}) = m+1$. 
\end{Ex}

\subsection{Extension triples associated to Lagrangian reduction}
\label{sect:Lfunctoriality}
We shall prove now that 
every symplectic Lie algebra $(\gg ,\o )$, which has a Lagrangian ideal $\ga$ arises as a Lagrangian extension of a flat Lie algebra. \\ 
%Since $\ga$ is an abelian ideal, the quotient Lie algebra  $\gh=\gg /\ga$ acts
%on $\ga$ by the restriction $\rho$ of the adjoint representation of $\gg$. 
%Let  $(\gg ,\o )$ be a symplectic Lie algebra which 
%admits a Lagrangian ideal $\ga$. 

\noindent 
Recall from Proposition 
\ref{prop:induced_fc} and \eqref{eq:nablao} that the associated flat torsion-free connection $\n=\bar \n^\o$ on the quotient Lie algebra $\gh=\gg /\ga$ satisfies the relation 
\begin{equation} \label{nEqu} \o_{\gh} (\n_u v, a) = -\o ( \tilde v, [ \tilde u, a]),
\; \text{ for all $u, v\in \gh, \, a \in \ga$.}\end{equation}
(Here, $\tilde u, \tilde v \in \gg$  denote respective lifts of $u,v$.)
The pair $(\gh, \nabla)$ is called the \emph{quotient flat Lie algebra} associated to the Lagrangian ideal $\ga$ of $(\gg, \omega)$.

\bt  \label{thm:ext_cocycle}
Let $(\gg ,\o, \ga , N)$ be a strongly polarized symplectic Lie algebra
and  $(\gh, \nabla)$ its associated quotient flat 
Lie algebra. Then there exists 
$$\alpha  = \alpha_{(\gg ,\o, \ga , N)} \in  Z^2_{\nabla}(\gh, \gh^*)
$$ satisfying \eqref{aEqu}, 
such that $(\gg ,\o ,\ga , N)$ is isomorphic to $F(\gh , \n ,\a )$. 
\et

\pf 
Observe first that, since $\ga$ is abelian, there is a well defined adjoint representation $\ad_{\gh,\ga}$ of $\gh$ 
on $\ga$, which satisfies   
\begin{equation}  \ad_{\gh,\ga} (u)   a   =   [u, a],  \; 
\text{ for all $u \in \gh,\, 
 a \in \ga$.}\end{equation}

Let $\pi_{\ga}: \gg \ra \ga$ be the 
projection map which is
induced by the strong polarization $\gg = \ga \oplus N$.
For $u,v \in \gh$,  let $\tilde u, \tilde v \in N$ denote their
lifts in $N$ with respect to the quotient homomorphism $\gg \ra \gh$. 
The expression  \begin{equation} \label{eq:talpha}
\tilde \alpha(u,v) = \pi_{\ga}( [ \tilde u, \tilde v ] ) 
\end{equation} 
then defines a two-cocycle
$\tilde \alpha \in Z^2_{\ad}(\gh ,\ga )$ 
for the representation $\ad_{\gh,\ga}$.

Let $\iota_{\o}: \ga \ra \gh^*$, $a \mapsto \omega(a, \cdot)$ be 
the identification of $\ga$ with $\gh^*$, which is induced by
$\omega$. With these definitions,
the  equation \re{nEqu} is equivalent to the relation 
\begin{equation} \label{eq:rho_omeg}
 \rho(u) \circ \iota_\omega =  \iota_\omega \circ  \ad_{\gh,\ga} (u) \; . 
\end{equation}
In particular, this shows that the representation $\rho$ of $\gh$ on $\gh^*$, which belongs to the flat connection $\nabla$ by \eqref{rhoEqu} 
is equivalent to $\ad_{\gh,\ga}$. 
%And $\phi_\o$ intertwines the two representations. 

Let $\pi_{\gh}: N \ra \gh$ be the isomorphism of vector spaces
induced by the quotient map $\gg \ra \gh$. The isomorphisms  
$\pi_\gh$ and $\iota_\o$ 
assemble to an isomorphism $$ 
\pi_\gh \oplus \iota_\o: \;
\gg = N \oplus \ga \,  \longrightarrow  \, \gh \oplus  \gh^* \; . $$
Define \begin{equation} \label{eq:alpha}
\alpha= \iota_{\o} \circ \tilde \alpha \; \, \in Z^2_{\rho}(\gh, \gh^*)
\end{equation} to be the push-forward of $\tilde \alpha$. 
It is now easily verified that the map  
$$ \pi_\gh \oplus \iota_\o: \;
(\gg, \omega) \ra (\gg_{\nabla,\alpha} , \o) \;  $$
defines an isomorphism of symplectic Lie algebras. 
As a consequence, $\alpha$ satisfies \eqref{aEqu},  by Proposition \ref{prop:ext_condition}.
\epf 

%whose cohomology class $[\alpha] \in H^2_{\rho}(\gh, \ga)$
%determines the Lie algebra extension $\gg \ra \gh$. 
We call the triple $(\gh, \nabla, \alpha)$ constructed in Theorem \ref{thm:ext_cocycle}
the extension triple associated to $(\gg,\omega, \ga , N)$. In general, triples 
$(\gh, \nabla, \alpha)$, 
where $\alpha \in  Z^2_{\nabla}(\gh, \gh^*)$ satisfies the \emph{symplectic extension condition} \eqref{aEqu}, will be called  \emph{a flat Lie algebra 
with symplectic extension cocycle}.  \\

\subsection{Functoriality of the correspondence}
We consider briefly the functorial properties of our constructions. The following
Lemma is an easy consequence of \eqref{nEqu}:
\bl  \label{lem:nabla_funct}
Let $(\gg, \o)$ be a symplectic Lie algebra with Lagrangian ideal $\ga$ and $(\gh =  \gg/\ga, \nabla)$ the associated quotient flat 
Lie algebra.  Let $\Phi: (\gg, \o) \ra (\gg',\omega')$ be an isomorphism of symplectic 
Lie algebras, and $\Phi_{\gh}: \gh = \gg/\ga \ra \gh' =  \gg'/\ga'$ the 
induced map on quotients, where $\ga' = \Phi(\ga)$. 
Then $\nabla' = (\Phi_{\gh})_{*} \nabla$ (push-forward of $\nabla$) is the 
associated quotient flat connection on $\gh'$.
\el 

Similarly, we can state:
\bl  \label{lem:alpha_funct}
Let $\Phi: (\gg, \o, \ga, N) \ra (\gg',\omega', \ga', N')$ be an isomorphism of strongly polarized symplectic Lie algebras. Then $\alpha_{(\gg',\omega', \ga', N')} = (\Phi_{\gh})_{*} \alpha_{(\gg, \o, \ga, N) }$. 
\el 
\pf Put $\alpha = \alpha_{(\gg, \o, \ga, N) }$ and $\alpha' =  \alpha_{(\gg',\omega', \ga', N')}$. By equations \eqref{eq:talpha} and 
\eqref{eq:alpha}, we have (using repeatedly that $\Phi$ is an isomorphism of strongly polarized symplectic Lie algebras) that \begin{eqnarray*}
\alpha' (\Phi_{\gh} u, \Phi_{\gh} v) ( \Phi_{\gh} w) & =& 
\omega'( \pi_{\ga'} ([\widetilde {\Phi_{\gh} u}, \widetilde{\Phi_{\gh} v]}'  ),  \Phi_{\gh} w) \\
 & = & \omega'( \pi_{\ga'} ([ {\Phi  \tilde u}, \Phi {\tilde v} ]' ),  \Phi_{\gh} w) \\
% & = & \omega'( \pi_{\ga'} ([ {\Phi \tilde u}, {\Phi \tilde v}]'  ),  \Phi_{\gh}  w) \\
& = & \omega'( \pi_{\ga'}( \Phi [ \tilde u, \tilde v]  ),  \Phi_{\gh}  w)  \\
& = &  \omega'( \Phi \,  \pi_{\ga} ([ \tilde u, \tilde v]  ),  \Phi_{\gh}  w) \\
& = & \omega( \pi_{\ga} ([ \tilde u, \tilde v]  ),  w) \\ & = & \alpha(u,v) (w)  \; 
\text{, for all $u,v, w \in \gh$.}
\end{eqnarray*} 
\vspace*{-1ex} \epf

Let  $(\gh, \nabla, \alpha)$ and $(\gh', \nabla,' \alpha')$ be flat Lie algebras with extension cocycles. An isomorphism of Lie algebras $\varphi: 
\gh \ra \gh«$ is an \emph{isomorphism of flat Lie algebras with extension cocycle}
if $\varphi^* \nabla' = \nabla$ and $\varphi^* \alpha ' = \alpha$. \\

In the view of Lemma \ref{lem:nabla_funct} and Lemma \ref{lem:alpha_funct},  isomorphic strongly polarized symplectic Lie algebras give rise to isomorphic flat Lie algebras with extension cocycle. 
Together with Theorem \ref{thm:framed_symplectic}, we therefore have:
 
\bc \label{cor:flangrangian_corresp}
The correspondence which associates to a strongly polarized symplectic Lie algebra $(\gg,\omega, \ga , N)$ its extension triple $(\gh,\nabla, \alpha)$ induces a bijection
between isomorphism classes of  strongly polarized symplectic Lie algebras
and isomorphism classes of flat Lie algebras with symplectic extension cocycle. 
\ec

\paragraph{Change of strong polarization}
Let $(\gg, \omega, \ga, N)$ and $(\gg, \omega, \ga, N')$ be two
strong polarizations belonging to the same Lagrangian extension. Then the
corresponding two-cocycles $\alpha =  \alpha_{(\gg, \omega, \ga, N)}$ and
$\alpha' =  \alpha_{(\gg, \omega, \ga, N')}$ differ by a coboundary: 

\bl \label{lem:changeofframe}
There exists $\sigma \in \Hom(\gh, \gh^*)$, satisfying
\begin{equation} \label{eq:sigmaL}
\sigma(u) (v) -  \sigma(v) (u) = 0 \; \text{, for all $u,v \in \gh$, } 
\end{equation}
such that $\alpha' = \alpha + \partial_{\rho} \sigma$.
\el 
\pf 
Let $\pi_{\ga}$, 
$\pi_{N}$, as well as $\pi'_{\ga}$, $\pi'_{N'}$ be the corresponding projection 
operators on the factors of $\gg$. Then there exists $\tau \in \Hom(\gg,\ga)$ with
$\ga \subseteq \ker \tau$ such that $\pi'_{N'} =  \pi_{N} + \tau$ and $\pi'_{\ga} =  \pi_{\ga} -  \tau$. Since, both $N$ and
$N'$, $\ga$ are Lagrangian, the homomorphism $\tau$ satisfies the condition  
\begin{equation} \label{eq:tauL}
\omega( \tau(n),m) + \omega( n, \tau(m)) = 0 \; 
\text{, for all $n,m \in N$.}
\end{equation}
Let $u,v \in \gh$ and $\tilde u , \tilde v \in N$ their respective lifts to $N$.
We compute $  [\pi'_{N'} \tilde u, \pi'_{N'} \tilde v] =[\tilde u +\tau \tilde u, \tilde v + \tau \tilde v] =  [\tilde u, \tilde v] + [\tilde u, \tau \tilde v] + [\tau \tilde u, \tilde v] $, since $\ga$ is abelian. 
By \eqref{eq:alpha}, we thus have 
\begin{eqnarray*} \alpha_{(\gg, \omega, \ga, N')}(u,v)  & = & \omega(\pi'_{\ga} ( [\pi'_{N'} \tilde u, \pi'_{N'} \tilde v]) , \, \cdot \, ) \\ 
& = & \omega(\pi_{\ga} ([\tilde u +\tau \tilde u, \tilde v + \tau \tilde v])
- \tau( [\tilde u +\tau \tilde u, \tilde v + \tau \tilde v]),  \, \cdot \, ) \\
& = &  \alpha_{(\gg, \omega, \ga, N)}(u,v) +  \omega([\tau \tilde u, \tilde v] + 
[ \tilde u, \tau \tilde v] - \tau( [ \tilde u,  \tilde v]),  \, \cdot \,) 
\end{eqnarray*}
Since $\ga$ is contained in $\ker \tau$, $\tau$ defines an element 
$\bar \tau \in \Hom(\gh, \ga)$ and therefore $\sigma = \iota_{\omega} \circ \bar \tau \in \Hom(\gh, \gh^*)$. Using \eqref{eq:rho_omeg}, we deduce from the above that $\alpha' = \alpha + \partial_{\rho} \sigma$. Moreover, \eqref{eq:sigmaL} is implied
by \eqref{eq:tauL}
\epf

\subsection{Equivalence classes of Lagrangian extensions} 
Let $\gh$ be a Lie algebra. %  and $\ga$ an abelian Lie algebra. 
A \emph{Lagrangian symplectic} extension $(\gg, \omega, \ga)$ 
over  $\gh$ % by $\ga$  
is a symplectic Lie algebra $(\gg, \omega)$ 
together with an extension of Lie algebras  
$$ 0 \ra \ga \ra \gg \ra \gh \ra 0 \; , $$ such that 
the image of $\ga$ in $\gg$ is a Lagrangian ideal of $(\gg, \omega)$.
 
\bd An \emph{isomorphism of Lagrangian extensions} over $\gh$ is an isomorphism 
of symplectic Lie algebras $\Phi: (\gg,\o) \ra (\gg', \o')$ such that the diagram 
$$
\begin{CD} \minCDarrowwidth45pt
 0 @>>>  \ga @>>>  \gg  @>>> \gh @>>> 0 \\
@.  @VV\Phi|_{\ga}V  @VV\Phi V @| \\
0 @>>>\ga' @>>>\gg'  @>>> \gh @>>> 0
\end{CD} 
$$ 
is commutative. 
\ed 
It is important to note that \emph{isomorphic Lagrangian extensions over $\gh$ give rise to the same associated quotient  flat Lie algebra $(\gh, \nabla)$}. This is an easy consequence of Lemma \ref{lem:nabla_funct}.\\
%So strictly speaking we are
%considering here isomorphisms of extensions over flat Lie algebras $(\gh,\nabla)$. 
%By slight abuse of notation we shall denote a Lagrangian extension with 
%associated flat Lie algebra $(\gh,\nabla)$ as an exact sequence 
% $$  0 \ra \ga \ra (\gg, \omega) \ra (\gh, \nabla) \ra 0 \; . $$

We now construct for any flat Lie algebra $(\gh,\nabla)$ 
a cohomology group, %  $H^2_{L,\nabla}(\gh, \gh^*)$, 
which describes all Lagrangian extensions of $\gh$ with
associated flat Lie algebra $(\gh,\nabla)$: \\

First, we  define Lagrangian one- and two-cochains on $\gh$ as 
\begin{eqnarray*}
\cC^1_{L}(\gh, \gh^*)  &= & \{ \varphi \in \cC^1 (\gh,\gh^*) \mid \varphi(u) (v) - \varphi(v) (u) = 0 \text{, for all $u,v \in \gh$} \} 
% \; \; \quad \quad 
\\
\cC^2_{L}(\gh, \gh^*)  &= & \{ \alpha \in \cC^2 (\gh,\gh^*) \mid \alpha \text{ satisfies \eqref{aEqu} } \}  
\end{eqnarray*}
Furthermore,  let $\rho = \rho^\nabla$ be the representation of $\gh$ on $\gh^*$
associated to $\nabla$, as defined in \eqref{rhoEqu}. 
Denote by $\partial_{\nabla} = \partial^{i}_{\rho}$
the corresponding coboundary operators for cohomology with 
$\rho$-coefficients. 
\bl  \label{lem:d1nabla}
The coboundary operator $\partial_{\nabla}: \cC^1(\gh, \gh^*) \ra  \cC^2 (\gh,\gh^*)$
maps the subspace $\cC^1_{L}(\gh, \gh^*)$ into $\cC^2_{L}(\gh, \gh^*) \cap Z^2_{\rho}(\gh, \gh^*)$. 
\el
\pf
Let $\f\in \cC^1(\gh, \gh^*)$, and $u,v,w\in \gh$. Then 
\begin{eqnarray*} (\partial_{\nabla}  \f ) (u,v) &=& \r (u)\f(v) -\r (v)\f (u) -\f ([u,v])\\
&=&-\f (v)\circ \n_u+\f (u)\circ \n_v-\f ([u,v])
\end{eqnarray*}
and, hence,
\begin{eqnarray} (\partial_{\nabla} \f ) (u,v)(w) &=& 
-\f (v)( \n_uw)+\f (u)( \n_vw)-\f ([u,v])(w). \label{eq:coboundrho}
\end{eqnarray}
Taking the sum over all cyclic permutations of $(u,v,w)$, we obtain:
 \begin{eqnarray*} \sum _{cycl}(\partial_{\nabla}  \f ) (u,v)(w)  &=&  \sum_{cycl}(
-\f (v)( \n_uw)+\f (u)( \n_vw)-\f ([u,v])(w)) \\ 
&=&  \sum _{cycl}(
-\f (w)( \n_v u)+\f (w)( \n_uv)-\f ([u,v])(w))\\
&=&  \sum _{cycl}(
\f (w)([u,v])-\f ([u,v])(w)).
\end{eqnarray*}
For  $\f\in \cC^1_{L}(\gh, \gh^*)$, all summands in the latter 
sum are zero. Hence, in this case $\partial_{\nabla}  \f  \in \cC^2_{L}(\gh, \gh^*)$.
\epf

Let $Z^2_{L,\nabla}(\gh,\gh^*) = \cC^2_{L}(\gh, \gh^*) \cap Z^2_{\rho}(\gh, \gh^*)$ denote the space of Lagrangian cocycles. 
We now define 
the \emph{Lagrangian extension cohomology group} for the flat Lie algebra $(\gh,\nabla)$
as $$ H^2_{L,\nabla}(\gh, \gh^*) \, = \; {{Z^2_{L,\nabla}(\gh,\gh^*)} \over {\partial_{\nabla} \, \cC^1_{L}(\gh, \gh^*)}} \; . $$

\begin{remark} By construction there is a natural map from $H^2_{L,\nabla}(\gh, \gh^*)$
to the ordinary Lie algebra cohomology group $H^2_{\rho}(\gh, \gh^*)$. Note that this map need not be injective, in general, see Example \ref{ex:Lag_cohom} below. 
\end{remark}

Together with Corollary \ref{cor:flangrangian_corresp}, the following shows that the isomorphism classes of Lagrangian extensions over 
a flat Lie algebra are in one-to-one correspondence with
the  group $H^2_{L,\nabla}(\gh, \gh^*)$: 
\bt  \label{thm:Lagrangian_corresp} Every symplectic Lagrangian extension $(\gg, \o, \ga)$ over the flat Lie algebra $(\gh, \nabla)$ gives rise to a \emph{characteristic extension class}
$$ [\alpha_{\gg,\omega, \ga}] \in \, H^2_{L,\nabla}(\gh, \gh^*) \; .$$
Two extensions  $(\gg, \o, \ga)$ and $(\gg', \o', \ga')$ over $(\gh, \nabla)$ are isomorphic 
if and only if they have the same extension class in $H^2_{L,\nabla}(\gh, \gh^*)$. 
\et
\pf Let $(\gg, \omega, \ga)$ be a Lagrangian extension over $(\gh, \nabla)$. 
Choose a strong polarization $(\gg, \omega, \ga, N)$, and put  $[\alpha_{\gg,\omega, \ga}] 
= [\alpha_{\gg,\omega, \ga,N}]$, where $\alpha_{\gg,\omega, \ga,N} \in Z^2_{L,\nabla}(\gh,\gh^*)$ is defined as in Theorem \ref{thm:ext_cocycle}. By Lemma \ref{lem:changeofframe},  this cohomology class is independent of the choice of strong polarization 
$N$.

Now suppose $\Phi: (\gg, \o, \ga) \ra (\gg', \o', \ga')$ is an isomorphism 
over $\gh$. Choose a complementary Lagrangian subspace 
 $N$ for $(\gg, \o, \ga)$. Then 
$\Phi: (\gg, \o, \ga, N) \ra (\gg', \o', \ga', N' = \Phi(N))$ is an isomorphism of
strongly polarized symplectic Lie algebras. 
By Lemma  \ref{lem:alpha_funct},  $\alpha_{(\gg',\omega', \ga', N')} = 
(\Phi_{\gh})_{*} \alpha_{(\gg, \o, \ga, N) } = \alpha_{(\gg, \o, \ga, N) }$, since,
by assumption, $\Phi_{\gh} = \mathrm{id}_{\gh}$. This shows that isomorphic
extensions over $\gh$ have the same cohomology class.

It remains to show that two extensions over $\gh$ with the same class are isomorphic. By Theorem \ref{thm:ext_cocycle}, it is enough to show that 
any two strongly polarized symplectic Lie algebras $F(\gh, \nabla, \alpha)$ and
$F(\gh,\nabla, \alpha')$ give rise to isomorphic extensions over $\gh$
if $ [\alpha] = [\alpha'] \in H^2_{L,\nabla}(\gh, \gh^*)$; that is, if $\alpha' 
= \alpha - \partial_{\nabla} \sigma$, for some $\sigma \in \cC^1_{L}(\gh, \gh^*)$.
Using \eqref{eq:f1} and  \eqref{eq:f2} it is easily verified that then the  map
$$ (\gg_{\nabla,\alpha}, \omega) \ra (\gg_{\nabla,\alpha'}, \omega) \; , \; (u, \xi ) \mapsto (u, \xi + \sigma(u) )$$  is the required
isomorphism of Lagrangian extensions over $\gh$.
\epf 

Similarly, we obtain the following refinement of Corollary \ref{cor:flangrangian_corresp}:

\bc \label{cor:Lagrangian_corresp}
The correspondence which associates to a  symplectic Lie algebra  with Lagrangian 
ideal $(\gg,\omega, \ga)$ the  extension triple $(\gh,\nabla, [\alpha])$ induces a bijection
between isomorphism classes of symplectic Lie algebras with Lagrangian 
ideal and isomorphism classes of flat Lie algebras with symplectic extension cohomology class. 
\ec

In certain situations the preceding corollary can be used  to classify also 
isomorphism classes of symplectic Lie algebras. See, for instance, Corollary
 \ref{cor:filiform_corresp} on filiform symplectic Lie algebras. 
 
\subsection{Comparison of Lagrangian extension cohomology and ordinary cohomology} \label{sect:Lcohomo_comp}
In this subsection, we shall briefly describe the kernel $\kappa_{L}$  of the natural map  
\begin{equation} \label{eq:Lextmap}
H^2_{L,\nabla}(\gg, \gg^*) \, \longrightarrow \, H^2_{\rho}(\gg, \gg^*) \; . 
\end{equation}
It is clear that the kernel of this map % \eqref{eq:Lextmap} 
is 
$$  \kappa_{L} =  \; { B^{2}_{\rho}(\gg,\gg^*) \cap Z^{2}_{L,\nabla}(\gg,\gg^*) \over  B^{2}_{L,\nabla}(\gg,\gg^*) }  \; ,  $$
where  $B^{2}_{\rho}(\gg,\gg^*) = \{ \partial^{1}_{\rho} \, \lambda  \mid  \lambda \in \Hom(\gg, \gg^*) \}$ is the set of ordinary two-coboun\-da\-ries with $\rho$-coefficients and
$B^{2}_{L,\nabla}(\gg,\gg^*) = \{ \partial^{1}_{\rho} \, \lambda  \mid  \lambda \in \mathcal{C}^1_{L}(\gg, \gg^*) \}$  is the set of two-coboundaries for Lagrangian extension cohomology. 

\begin{remark}
In the following computations we will use the natural identification of 
the module $\mathcal{C}^1(\gg, \gg^*)= \Hom(\gg, \gg^*) $  with bilinear forms, as well as the inclusion of $ \mathcal{C}^2(\gg, \gg^*) = \Hom( \bigwedge^2 \gg, \gg^*)$ 
into the vector space of trilinear forms. We thus have a  decomposition 
$$ \mathcal{C}^1(\gg, \gg^*) = S^2 \gg^* \oplus  \bigwedge^2 \gg^* $$ of  
$\gg^*$-valued one-cochains into symmetric and alternating forms, where  $$ \mathcal{C}^1_{L}(\gg, \gg^*)  =  S^2 \gg^* \; . $$ In particular, we may consider the module of two-cocycles  $Z^2(\gg) \subset  \bigwedge^2 \gg^*$ %$$ Z^2(\gg) \subset  \bigwedge^2 \gg^* 
%$$ 
as a subspace of  $ \mathcal{C}^1(\gg, \gg^*)$, and the coboundary map (for Lie algebra cohomology with trivial coefficients) 
$ \partial^2:   \bigwedge^2 \gg^*  \;   \ra  \;  \bigwedge^3 \gg^*$ % \subset \mathcal{C}^2_{\rho}(\gg, \gg^*)$ 
maps to a subspace of  $\mathcal{C}^2_{\rho}(\gg, \gg^*)$. 
\end{remark}
\par

\bp \label{prop:Lkernel} 
We have \begin{enumerate}
\item  For all $\lambda \in \bigwedge^2 \gg^* $,  $\; \sum _{cycl}(\partial^1_{\rho} \lambda)  =\,  2 \, \partial^2 \lambda \; . $
\item 
$   B^{2}_{\rho}(\gg,\gg^*) \cap Z^{2}_{L,\nabla}(\gg,\gg^*) =  \, \partial^1_{\rho} \left( S^2 \gg^* \oplus  Z^2(\gg) \right) = \, 
B^{2}_{L,\nabla}(\gg,\gg^*) + \partial^1_{\rho} \left( Z^2(\gg) \right) \; . 
$
\end{enumerate}
\ep 
\pf Recall from the proof of Lemma \ref{lem:d1nabla} that 
$$ \sum _{cycl}(\partial^1_{\rho}  \lambda) (u,v,w) = 
 \sum _{cycl} (
\lambda (w, [u,v])- \lambda ([u,v], w)) \; . $$
 For alternating $\lambda$,
this implies $\sum _{cycl} (\partial^1_{\rho}  \lambda) (u,v,w) = 
2  \sum _{cycl} \lambda (w, [u,v])$. Therefore, $\sum _{cycl} (\partial^1_{\rho}  \lambda)  = 2 \, \partial^2 \lambda$, 
thus proving the first formula.

By definition of $Z^{2}_{L,\nabla}(\gg,\gg^*)$, we have 
$$ B^{2}_{\rho}(\gg,\gg^*) \cap Z^{2}_{L,\nabla}(\gg,\gg^*) \, = \,
 \{ \partial^1_{\rho}  \lambda \mid  \sum _{cycl} (\partial^1_{\rho}  \lambda) =0 \; , \;  \lambda \in \mathcal{C}^1_{\rho}(\gg, \gg^*) \} \; . 
$$ 
Decomposing, $\mu \in  \mathcal{C}^1_{\rho}(\gg, \gg^*)$ as $\mu = \mu_{L} + \lambda$, where $\mu_{L} \in S^2 \gg^*$ and $\lambda \in \bigwedge^2 \gg^*$,  we deduce (using Lemma 
\ref{lem:d1nabla}) 
that $\sum _{cycl} (\partial^1_{\rho} \mu) =  \sum _{cycl} (\partial^1_{\rho} \lambda)$. By 1., this implies  $\sum _{cycl} (\partial^1_{\rho} \mu)  = 0$ if and only if $\lambda \in Z^2(\gg)$. 
Hence, the second part of the proposition holds. 
\epf

For an illustration, we compute two examples. 
\bp Let $(\gh,\nabla)$ be a two-step nilpotent flat Lie algebra,  where $\nabla$ is the canonical flat connection on $\gh$, satisfying  $\nabla_{u} v = {1 \over 2} [u, v]$, for all $u,v \in \gh$. Then 
the natural map $H^2_{L,\nabla}(\gh, \gh^*) \, \longrightarrow \, H^2_{\rho}(\gh, \gh^*)$ is injective. 
\ep 
\pf By Proposition \ref{prop:Lkernel},  it is enough to show that  $\partial^1_{\rho} \left( Z^2(\gh) \right)$ is contained in
$B^{2}_{L,\nabla}(\gh,\gh^*)$. Let $\varphi \in \cC^1(\gh, \gh^*)$ then by 
\eqref{eq:coboundrho} we have 
$$ (\partial_{\nabla} \varphi) \, (u,v,w) =  -  {1 \over 2} \f (v,  [u, w] )+ {1 \over 2} \f (u, [v,w]) -\f ([u,v], w) \; .
$$
If $\lambda \in Z^2(\gh)$ is a two-cocycle for $\gh$ we infer that 
\begin{equation} \label{eq:plambda}
(\partial_{\nabla} \lambda) \, (u,v,w) =    {1 \over 2} \lambda (w, [u,v]) \; . 
\end{equation}   
Note that $\lambda$ defines a map $\mathcal{E}(\lambda): \gh \ra [\gh, \gh]^*$, 
$w \mapsto \lambda(w, \cdot)$, such that 
$$  (\partial_{\nabla} \lambda) \, (u,v,w) =    {1 \over 2} \mathcal{E}(\lambda)(w) ([u,v])   \; . $$
Note further that $\mathcal{E}(\lambda)$ vanishes on the center $Z(\gh)$ of $\gh$,
which contains the commutator $[\gh, \gh]$. Therefore, there exists a symmetric form $\mu \in S^2(\gh^*)$  
which satisfies, for all $u,v,w \in \gh$,   
$$   \mu(w, [u,v ] ) =  \mathcal{E}(\lambda)(w) ([u,v]) =  \lambda (w, [u,v]) \; . $$
Then we compute 
\begin{eqnarray*}
 (\partial_{\nabla}  \mu) \, (u,v,w)  & = &   -  {1 \over 2} \mu (v,  [u, w] )+ {1 \over 2} \mu (u, [v,w]) -\mu  ([u,v], w) \\
& = &  -  {1 \over 2} \mu (v,  [u, w] )+ {1 \over 2} \mu (u, [v,w]) -\mu (w, [u,v])  \\
& = &  -  {1 \over 2} \lambda (v,  [u, w] )+ {1 \over 2} \lambda (u, [v,w]) -\lambda (w, [u,v]) \\
& = &  -  {1 \over 2}   \lambda (w, [u,v]) -  \lambda (w, [u,v]) \\
& = &  - { 3 \over 2}    \lambda (w, [u,v])
\end{eqnarray*}
Comparing with \eqref{eq:plambda}, we deduce that $\partial_{\nabla} ( -  {1 \over 3}   \mu) = \partial_{\nabla} \lambda$. 
This shows that $\partial^1_{\rho} \left( Z^2(\gh) \right)$ is contained in
$B^{2}_{L,\nabla}(\gh,\gh^*)$. 
\epf 

In view of the previous proposition it seems important to remark
that there do exist simple examples of flat Lie algebras where the
homomorphism  \eqref{eq:Lextmap} is not injective: 

\begin{Ex}[Natural cohomology map is not injective]   
\label{ex:Lag_cohom} 
Let $V = \spanof{e_{1}, e_{2}}$ be a two-dimensional vector space
viewed as an abelian Lie algebra $\gg$. Let $\nabla$ be the
% torsion-free flat 
connection on $\gg$ whose non-zero products in the basis are 
$$\nabla_{e_{1}} e_{1} = e_{1} \,  \text{ and } \,  
\nabla_{e_{2}} e_{1} = \nabla_{e_{1}} e_{2} = e_{2}\; .$$  It is straightforward to verify that $\nabla$ is a torsion-free flat connection on the abelian Lie algebra $\gg$. Let $\lambda \in \cC^1_{\rho}(\gg,\gg^*)$. Using \eqref{eq:coboundrho}  we compute 
\begin{eqnarray*} \partial^1_{\rho} \lambda (e_{1}, e_{2}, e_{1}) & = & 
- \lambda( e_{2}, e_{1} ) + \lambda(e_{1}, e_{2}) \\
 \partial^1_{\rho} \lambda (e_{1}, e_{2}, e_{2}) & = & 
- \lambda( e_{2}, e_{2} )  
\end{eqnarray*}
This shows that $B^{2}_{L,\nabla}(\gg,\gg^*) = \{ \mu \in \Hom(\bigwedge^2 \gg, \gg^*) \mid \mu( e_{1}, e_{2}, e_{1}) = 0 \}$ and 
\begin{equation*} \begin{split}
B^{2}_{\rho}(\gg,\gg^*) \cap Z^{2}_{L,\nabla}(\gg,\gg^*) & = B^{2}_{L,\nabla}(\gg,\gg^*) + 
\partial^1_{\rho} \left( Z^2(\gg) \right) \\ & =  B^{2}_{L,\nabla}(\gg,\gg^*)
\oplus \{ \mu \in \Hom(\bigwedge^2 \gg, \gg^*) \mid \mu( e_{1}, e_{2}, e_{2}) = 0 \}. 
\end{split}
\end{equation*}
In particular, $\kappa_{L} \cong   \{ \mu \in \Hom(\bigwedge^2 \gg, \gg^*) \mid \mu( e_{1}, e_{2}, e_{2}) = 0 \}$ is non-zero.  
\end{Ex}

%%%%%%%%%%%%%%%%%%%%%%%%%%%%%%%%

\part{Existence of Lagrangian normal subgroups}

%%%%%%%%%%%%%%%%%%%%%%%%%%%%%%%%

\section{Existence of Lagrangian ideals: basic counterexamples}  
\label{sect:c_examples}

In this section we present several fundamental examples of real symplectic 
Lie algebras which show that, even for nilpotent symplectic Lie algebras, Lagrangian ideals do not necessarily exist. It seems worthwhile to mention that all Lie algebras occurring in these examples do have a \emph{unique} maximal abelian ideal,  which is non-degenerate for the symplectic structure, and, in particular, the examples decompose as proper symplectic  semi-direct products. 
%  in a symplectic Lie algebra. 

\subsection{Solvable symplectic Lie algebras without Lag\-ran\-gian ideal}
Every two-dimensional symplectic Lie algebra $(\gg, \o)$ has a one-dimensional ideal, which is Lagrangian for dimension reasons. The following four-dimen\-sio\-nal solvable symplectic Lie algebra
$(\gg ,\o )$ does not admit a  Lagrangian ideal.

\begin{Ex}[four-dimensional, metabelian, symplectic rank one] 
\label{ex:fdim_metab}

\hspace{1ex} \\
\noindent The Lie algebra 
$\gg$ is defined with respect to the basis $\{ X,Y,Z,H \}$ by the following
non-trivial Lie brackets
\[ [H,X]=-Y,\; [H,Y]=X,\; [H,Z] = -Z.\] 
Therefore $\gg$ is metabelian with three-dimensional 
commutator ideal $[ \gg, \gg ] = \spanof{X,Y,Z}$.  
The only two-dimensional 
subalgebras of $\gg$ which are invariant
under the derivation $\ad(H)$ are
$\spanof{ X,Y}$ and $\spanof{ H,Z}$.
Since every ideal is invariant by $\ad(H)$, 
$\spanof{ X,Y}$ is the only two-dimensional
ideal of $\gg$.

Let $\{ X^*,Y^*,Z^*,H^* \}$ be the dual basis. Using the differentials
\[ \partial X^* = -H^*\wedge Y^*,\; \partial Y^* = H^*\wedge X^*,\; 
\partial H^* = 0,\; \partial Z^* = H^*\wedge Z^* \] 
it is easily checked that the two-form \[ \o = X^*\wedge Y^* + H^*\wedge Z^* \] 
is closed. 
Neither of the above two subalgebras is isotropic for 
$\omega$. 
Indeed, $\spanof{ X,Y}$ is an abelian \emph{non-degenerate} ideal of $(\gg, \o)$ and there is a semi-direct 
orthogonal decomposition 
 $$   (\gg, \o) =  (\spanof{ H,Z} \, , \,  H^*\wedge Z^*) \oplus ( \spanof{ X,Y} \, , \,  X^*\wedge Y^*) \; . $$  
 In particular,  \emph{the symplectic Lie algebra $(\gg,\omega)$ has no Lagrangian ideal}.
Note, however, that $\gj = \spanof{Z}$ is a one-dimensional
normal isotropic ideal, with $\gj^{\perp_{\omega}} = [\gg, \gg ]$.
In particular, \emph{$(\gg,\omega)$ is symplectically reducible and has a two-dimensional abelian reduction}. 
\end{Ex} 
The above example is neither completely solvable nor unimodular. 
The next example in dimension six 
shows that complete solvability and unimodularity
are not sufficient to ensure the existence of a Lagrangian ideal.

\begin{Ex}[six-dimensional, completely solvable, no 
La\-gran\-gian ideal] \label{ex:noLag_cs6}

% \hspace{1ex} \\  
Let $\gg = \spanof{ d_1, d_2} \oplus_{D_{1}, D_{2}} V_{4}$
be the semi-direct sum of the abelian ideal $V_{4}$  and the plane spanned by $d_{1} , d_{2}$.  With respect to 
the basis $\{ e_1,\ldots ,e_4 \}$ of $V_{4}$, we let $d_{1}, d_{2}$ 
act on $V_{4}$ via the
two derivations $D_1=\mathrm{diag}(\mu_1,-\mu_1,0,0)$ and $D_2
=\mathrm{diag}(0,0,\mu_2,-\mu_2)$, where $\mu_1, \mu_2\in \bR$.  
Putting $\{d^1, d^2, e^1,\ldots ,e^4 \}$ for the dual basis, one can check that
\[ \o = e^1\wedge e^2 + e^3\wedge e^4 + d^1\wedge d^2\]
is a sum of closed forms. Moreover, the two-form $\o$ is
non-degenerate. This shows that $\o$ is a symplectic form. 

{}From now on we assume
that $\mu_1\mu_2\neq 0$. We claim that, in this case,  
$(\gg ,\o )$ does not admit any Lagrangian ideal. To see this
we consider the root decomposition of $\gg$ with respect to the 
Cartan subalgebra $\gh= \spanof{ d_1, d_2}\subset \gg$, which 
is of the form 
\[ \gg = \gh \oplus \bigoplus_{i=1}^{4} \mathbb{R} e_{i} \; .\]
Using this decomposition, 
it follows that every abelian ideal $\ga \subset \gg$ is 
contained in the \emph{maximal abelian ideal} 
$V_{4}= [\gg, \gg] \subset \gg$. 
Since the subspace $V_{4}$ is non-degenerate with respect to 
$\o$, $V_{4}$ is also a \emph{non-degenerate ideal}  and 
$$ (\gg, \o) =   (\spanof{ d_1, d_2} \, , \, d^{1} \wedge d^{2})  \oplus_{D_{1}, D_{2}} (V_{4}, \o)$$  is a symplectic semi-direct sum.
Since, by 1.\ of Lemma \ref{lem:normal_red}, every isotropic ideal is abelian, it is contained in $V_{4}$.
This implies that every isotropic abelian ideal 
has dimension at most two. In particular, 
this proves that there exists no Lagrangian ideal 
$\ga \subset \gg$. 
\end{Ex}

\subsection{An eight-dimensional nilpotent symplectic Lie algebra of symplectic rank three}
We shall establish in Theorem \ref{thm:dimlt8} that 
every nilpotent symplectic Lie algebra 
of dimension less than or equal to six admits a Lagrangian
ideal. The next example shows that this result can not be extended to higher dimensions. 
\begin{Ex}[eight-dimensional,  nilpotent, no  Lagrangian ideal]
\label{ex:noLag_n8}

\hspace{1ex} \\  
Let $\bar \gg = \gh_3 \oplus \gh_3$ be the direct sum of two
Heisenberg Lie algebras. Choose a basis $ \{ X,Y,Z,X',Y',Z' \}$, such that non-zero 
commutators are $[ X, Y] = Z$ and  $[X', Y'] = Z'$. 
Let  $\{ X^*,Y^*,Z^*,X'^*,Y'^*,Z'^* \}$ be the corresponding dual basis.
We endow the nilpotent Lie algebra $\bar \gg$ with the symplectic form
\[ \bar \o = X^*\wedge Z^* + X'^*\wedge Z'^* + Y^*\wedge Y'^*.\]
Let $\varphi$ be the derivation of $\bar \gg$ which maps $Y$ to
$X$, $Y'$ to $X'$ and the other basis vectors to zero. 
The derived  two-cocycle 
$\a  = \bar \o (\varphi \, \cdot \,  , \cdot \, ) +
\bar \o(\cdot \,  ,\, \varphi \, \cdot  )$ % on $\bar \gg$ 
then is 
\be \label{a1Equ} \a = Y^*\wedge Z^* + Y'^*\wedge Z'^*.\ee
This is easily computed using the fact that the dual map 
$\varphi^*: \bar \gg^*\ra \bar \gg^*$ maps 
$X^*$ to $Y^*$, $X'^*$ to $Y'^*$
and the remaining vectors of the dual basis to zero. The same fact
immediately implies that $\a_{\varphi} = \a (\varphi \,\cdot \,  , \, \cdot \,  ) + \a (\cdot \,  , \, \varphi  \,  \cdot \, )=0$. 
Let 
\be \label{gg8Equ} ( \, \gg = \spanof{\xi} + \bar \gg + \spanof{ H} \, , \, \o = \xi^*\wedge H^* + \bar \o \, )
\ee
be the symplectic oxidation of $(\bar \gg, \bar \o)$ associated with 
the derivation $\varphi$ and the one form $\l=0$, as constructed in  Proposition \ref{prop:symplecticox}.
\end{Ex}

\vspace{1ex}
\bp \label{Ex3Prop} The eight-dimensional symplectic Lie algebra $(\gg , \o )$ defined in \eqref{gg8Equ} of
Example \ref{ex:noLag_n8} is four-step nilpotent and is of symplectic rank three. In particular, $(\gg , \o )$ has no Lagrangian ideal.
\ep 
\pf As is easily deduced from the above definitions and the defining equations \eqref{eq:lie1}, \eqref{eq:lie2},  
the descending central series of $\gg$ is given by 
\begin{eqnarray*} 
C^1(\gg ) &=& [\gg ,\gg ] =  \langle X,Z,X',Z', H\rangle \\
C^2(\gg)  &=& \langle Z,Z',H\rangle\\
C^3(\gg)  &=& \langle H\rangle \\
C^4(\gg)  &=& 0 .
\end{eqnarray*}
Therefore, $\gg$ is of nilpotency class four. 
Observe that \be  W:=  \spanof{ \xi, X,Z,X',Z', H } =  \spanof{ \xi, 
 [\gg ,\gg ] } 
= \spanof{ Y, Y' }^{\perp_\o} \ee
is a \emph{non-degenerate abelian ideal}  in $(\gg, \o)$ and 
$$  (\gg, \o) = (  \spanof{ Y, Y' } \, , \, Y^{*} \wedge Y^{'*}) \oplus (W, \omega) \; .  $$ 
Suppose that $\ga \subset \gg$ is an abelian ideal. 
We claim that 
\be \label{claim1Equ} \ga \subset  W.
\ee
Suppose otherwise. Then there exists $U=aY+a'Y'+ \bar{U}\in \ga$, $\bar{U}\in  W$, 
such that $a\neq 0$ or $a'\neq 0$. Without 
restriction we can then assume that $a=1$.   
Using \re{a1Equ}, as well as \eqref{eq:lie1}, \eqref{eq:lie2}, we derive that 
\[  [\xi , U]= X + a' X' \, \in \ga\]
and 
\[ [[\xi , U],U] = Z+(a')^2Z'\neq 0 \, ,\]
which contradicts the fact that $\ga$ is abelian.  
This proves \re{claim1Equ}. 
Now let $\ga\subset \gg$ be an isotropic ideal. It is abelian by 
Lemma \ref{lem:normal_red} and, hence, contained in $W$.  
Since  $\o$ is non-degenerate on $W$,  
the dimension of the isotropic subspace $\ga\subset W$ is at most three. Therefore, $\ga$ is not Lagrangian in $(\gg, \omega)$.

On the other hand $\gj = \spanof{H,Z,Z'}$ is clearly an isotropic ideal, showing that the symplectic rank of $(\gg, \omega)$ is three. 
\epf

As a side-remark we state the following corollary which shows that Theorem \ref{thm:Lagrange1} in Section \ref{InvsubspacesSection} cannot be generalized from endomorphisms of symplectic vector spaces to derivations of 
symplectic Lie algebras: 

\bc The six-dimensional symplectic Lie algebra $(\bar 
\gg =\gh_3\oplus \gh_3, 
\bar \o )$ occuring in Example \ref{ex:noLag_n8}  has no $\varphi$-invariant Lagrangian ideal. 
\ec 
One can ask whether the non-existence of Lagrangian ideals
in Example \ref{ex:noLag_n8} is due to the particular choice of
symplectic form. The following result answers this
question. 
\bt  \label{thm:noLg84}
For every choice of symplectic form $\o$ on the  
eight-dimen\-sio\-nal Lie algebra $\gg$ of 
Example \ref{ex:noLag_n8},  the symplectic Lie algebra 
$(\gg ,\o )$ has no Lagrangian ideal.
\et 

\pf Let us first determine the space $Z^2(\gg )$ of closed
$2$-forms. Its dimension can be read off from \[ \dim  Z^2(\gg ) = \dim \bigwedge^2\gg^* - \dim \partial  \left( \bigwedge^2\gg^*  \right) = 28-17=11,\]
where 
$$ \partial = \partial^2: \; \bigwedge^2\gg^* \ra \bigwedge^3\gg^*$$ is the boundary operator on forms. 
In fact, 
\begin{eqnarray*}   \partial \left( \bigwedge^2\gg^* \right)  &=& 
\langle \,  e^{123}, e^{134}, e^{136}, e^{156}, e^{167},   e^{236}, e^{246}, 
e^{356} ,e^{367},   -e^{137}+e^{256},
\\
&& -e^{237}+e^{456}, e^{235}+e^{146}, 
-e^{238} + e^{467}, -e^{568}+e^{347}, e^1\wedge ( e^{35}+e^{26}),\\ 
&&
 -e^{138} + e^2\wedge (e^{34}+ e^{67}),
-e^{168}+e^5\wedge (e^{34} + e^{67}) \, \rangle,
\end{eqnarray*}
where we are using  the basis $\{ e^1,\ldots ,e^8 \} $ of $\gg^*$, which is dual
to 
\[ \{ e_1,\ldots ,e_8 \} = \{ \xi , X, Y, Z, X' , Y', Z', H \} \] 
and the notation 
$e^{ij} = e^i\wedge e^j$, $e^{ijk}= e^i\wedge e^j\wedge e^k$.  
One can easily check that the following eleven $2$-forms are 
closed and linearly independent:
\be \label{basisEqu}  \left\{ \,  e^{12}, e^{13},  e^{15}, e^{16}, e^{23},  e^{34}, e^{36},  e^{56}, e^{67}, 
   e^{26}-e^{35},e^{24}+e^{57}+e^{18} \, \right\} .\ee
Therefore, they form a basis of $Z^2(\gg )$.\\

\begin{claim} \label{claim1}
The subspace  
\[ W : = \spanof{ e_1, e_2, e_4, e_5, e_7 , e_8 } = 
\spanof{ \xi, X,Z,X',Z', H}\subset \gg \, ,\]
cf.\  \re{claim1Equ},  
is an abelian ideal,  which is  \emph{non-degenerate} for every symplectic form $\o$ on $\gg$. 
\end{claim}
\noindent 
\begin{proof}[Proof of Claim \ref{claim1}]   Since the last element  
of the basis \re{basisEqu} is the only which does not vanish
on $e_8=H$, we see that every symplectic form $\o$ on $\gg$ has 
non-zero coefficient over $e^{24}+e^{57}+e^{18}$.
Up to scaling, we can assume that the coefficient is $1$, so that 
\be \label{oEqu} \o = e^{24}+e^{57}+e^{18} + ae^{12} +  be^{15} + \o_{1},\ee 
where $a,b\in \bR$,  and $\o_{1}$ is a linear combination of the remaining eight 
basis vectors of $Z^2(\gg)$. Note  that the restriction 
of $\o_{1}$ to $\bigwedge^2 W$ vanishes. Now one can easily check that $\o$
is non-degenerate on $W$, for any choice of $a,b\in \bR$.  
\end{proof} 

\noindent 
Claim \ref{claim1}  above implies that $W$ does not contain a Lagrangian
ideal of $\gg$. On the other hand, we saw in the proof of
Proposition \ref{Ex3Prop} that every abelian ideal
of $\gg$ is contained in $W$. Since a Lagrangian ideal
is necessarily abelian, this proves the theorem. 
\epf

\begin{remark} \label{RemarkCP} The symplectic form $\o$ considered in Example \ref{ex:noLag_n8}
is obtained by putting $a=b=0$ and $\bar \o = e^{36}$ in \re{oEqu}. As in Example \ref{ex:noLag_n8},
the above proof shows that 
the ideal $\gj = \spanof{H,Z,Z'}$ is isotropic for all 
symplectic forms $\omega$ on the Lie algebra $\gg$. 

Since  $\gj^{\perp_\o} = \spanof{H,Z,Z',Y,Y'}$, $\ga = \spanof{H,Z,Z',Y}$ is an abelian
Lagrangian subalgebra. This shows that the existence of Lagrangian abelian
subalgebra of a symplectic Lie algebra $(\gg ,\o )$ does not imply the existence of an abelian ideal 
of $(\gg ,\o )$ 
and not even the 
existence of a symplectic form 
$\o'$ on $\gg$ such that $(\gg, \o')$ admits an abelian ideal, contrary to the case of commutative
polarizations considered in \cite[Thm.\ 4.1]{EO}. 
\end{remark}

%%%%%%%%%%%%%%%%%%%%%%%%%%%%%%%%

\section{Endomorphism algebras on symplectic vector spaces}
\label{sect:qalgebras}

Let $\o$ be an alternating form on a vector space $V$. (At the moment 
we do not assume that $\o$ is non-degenerate.)
For any endomorphism $\varphi$ of $V$ we define alternating forms $\a = \omega_\varphi$ and $\b  =
\o_{\varphi,\varphi}$ on $V$ by
\be \a (u,v) = \omega_\varphi(u,v)\,  := \; \, \o (\varphi u, v) + \o(u, \varphi v) \ee
and 
\be \b (u,v) := \, \o (\varphi^2 u,v) +
2\o (\varphi u,\varphi u) + \o (u,\varphi^2 v).\ee
We are interested to study solutions $\varphi$ of the 
quadratic equation
\be \label{betaaEqu} \beta = \o_{\varphi,\varphi} =0. \ee
More specifically, we will be interested to study Lie algebras of endomorphisms % on a symplectic vector space, 
which satisfy this condition.

\bd  \label{def:symplecticalgs} 
Let $(V, \omega)$ be a symplectic vector space. 
An abelian Lie subalgebra $\gn$ of $\End(V)$ will be called 
\emph{symplectic} for $\o$ if, for all $\varphi, \psi \in \gn$, 
$u,v \in V$,  the 
relation 
\begin{equation} 0 = 
 \o (\varphi \psi u,v) +
\o (\varphi u,\psi v)  + \o (\psi u,\varphi v)+ \o (u,\varphi \psi v )
\end{equation} 
is satisfied. 
\ed 

The analysis of abelian symplectic subalgebras of endomorphisms is 
motivated by Proposition \ref{prop:abelian_red}, which shows that 
such algebras arise from central symplectic reduction to abelian Lie algebras. 
They can also be viewed as
generalizing abelian subalgebras of the symplectic Lie algebra
$\gs\gp(\omega) \subset \End(V)$. Our results on such algebras 
will be a principal tool for our investigations in Section \ref{sect:threestep}.
A first application is discussed in Theorem \ref{thm:onedimabelian_reduction}
below. 

%\subsection{Symplectic subalgebras generated by one endomorphism}
\subsection{Invariant Lagrangian subspaces for nilpotent endomorphisms}
\label{InvsubspacesSection}

Let 
$\varphi  \in \End V$ an endomorphism which is a solution of \eqref{betaaEqu}.  A basic and straightforward observation is:

\bl \label{lem:basics} Let $\alpha = \o_{\varphi}$. Then:
\begin{enumerate} 
\item 
$\varphi$ is skew with respect to $\alpha$.
\item $\ker \alpha$ is a 
$\varphi$-invariant subspace of $V$.
\item $\im \varphi  \perp_{\omega} (\ker  \alpha \cap \ker 
\varphi)$.
\item For all $Z \in \ker \alpha \cap \ker \varphi$, $Z^{\perp_{\omega}}$  (the ortohogonal of $Z$ with respect to $\omega$) 
is a $\varphi$-invariant subspace of $V$.
\item  $(\im \varphi)^{\perp_{\omega}} \cap \ker \varphi$ is contained in $\ker \alpha$. 
\end{enumerate}
\el  

\bl \label{lem:imagesperp} Let $\varphi  \in \End V$ be an  endomorphism which satisfies $\varphi^k = 0$, 
for some $k \geq 1$, and 
which is a solution of the equation \eqref{betaaEqu}. 
Then %,  with  respect to $\o$,  
\[ \im  \varphi^{j} \perp_{\omega} \im   \varphi^{k-j} ,\]
for all $j \in \{ 0, 1, \ldots  , k \}$. 
\el  

\pf We may assume $k \geq 2$.
Clearly, the statement is true for  $j \in \{0,k\}$.
Let $u,v \in V$ and define $$ \tau_{j} = \; \o(\varphi^j(u),\varphi^{k-j}(v)) \; , \;  \; 1 \leq j \leq k-1 \; . $$
Since $\varphi$ satisfies  \eqref{betaaEqu}, we have equations 
\begin{eqnarray} \tau_{2} + 2 \tau_{1}  & = & 0 \\
\tau_{j+1} + 2 \tau_{j} + \tau_{j-1}  & = & 0  \;  ,  \;  j \in \{2, \ldots, k-2\} \\
2 \tau_{k-1} + \tau_{k-2} & =  & 0
\end{eqnarray}
These define a linear system 
$$   A  \cdot \; \left(\begin{matrix} \tau_{1} \\ \vdots \\ \tau_{k-1} \end{matrix}\right)   \;  = \;  \;
\left(\begin{matrix} 0 \\ \vdots \\ 0 \end{matrix}\right) \; , $$
where $A \in {\rm Mat}(k-1,k-1)$ is a matrix of the following form
(non-zero entries shown): 
$$ A \;  = \;   \left(\begin{matrix} 2 & 1 & & \\ 1 & 2 &1 & \\ 
  &     \ddots    & \ddots & \ddots & \\
 &   &  1 & 2 & 1  \\
 &    &     &  1 & 2 \\
  \end{matrix}\right)    \; \; \; . $$ 
Since $A$ is nonsingular, we conclude that $\tau_{j} = 0$, for all $j$.
This shows that $\im \varphi^j$ is perpendicular to $\im  \varphi^{k-j}$ with respect to $\o$.
\epf 

We deduce: 
\bp \label{prop:zperp}
 Let $\varphi  \in \End V$ be an  endomorphism which satisfies $\varphi^k = 0$, 
for some $k \geq 1$, and 
which is a solution of the equation \eqref{betaaEqu}. 
Then 
\begin{enumerate}
\item $\im \varphi^{k-1} \subset \ker \a$. In particular, 
the alternating form $\a$ is degenerate.
\item For all $Z \in \im \varphi^{k-1}$, the orthogonal
complement $Z^{\perp_{\omega}}$ is invariant by $\varphi$. 
\end{enumerate} 
\ep 
\pf By Lemma \ref{lem:imagesperp}, $\im  \varphi^{k-1} \perp_{\omega} \im \varphi$.
Since $\im  \varphi^{k-1}$ is contained in $\ker \varphi$, 5.\ and 4.\ of Lemma \ref{lem:basics}
imply the statements of the proposition. 
\epf 

We call a subspace of $(V,\omega)$ \emph{Lagrangian} if it is isotropic for $\o$ and of maximal dimension with this property. Our main observation is:

\bt  \label{thm:Lagrange1}
Let $(V,\o )$ be a vector space with alternating bilinear form.  
Let $\varphi  \in \End V$ an  endomorphism which satisfies $\varphi^k = 0$, 
for some $k \geq 1$, and 
which is a solution of the quadratic equation 
\eqref{betaaEqu}. Then
there exists a $\varphi$-invariant Lagrangian 
subspace of $(V,\o)$.
\et

\pf We may assume that $\varphi^{k-1} \neq 0$.
The proof is now by induction on the dimension of $V$.
Let $\dim V = n$, and assume that the result is shown
for all vector spaces with alternating form, which have dimension less than $n$.
Let $0 \neq Z \in \im \varphi^{k-1}$. By Proposition \ref{prop:zperp}, the subspace $Z^{\perp_{\omega}}$ is invariant by $\varphi$.  Consider the  vector space $W = Z^{\perp_{\omega}}/ \, \langle Z \rangle$ with alternating form $\bar \omega$ 
induced by $\omega$. Since $\varphi(Z) = 0$, $W$ has a nilpotent endomorphism $\bar \varphi$, which is induced by  $\varphi$, and also satisfies $\bar \beta = \bar \omega_{\bar \varphi, \bar \varphi} = 0$. The  induction hypothesis implies that  $\bar \varphi$ has a Lagrangian $\bar \varphi$-invariant subspace. 
Its preimage in $Z^{\perp_{\omega}}$ is a Lagrangian 
subspace of $(V,\o)$, which by construction 
is invariant by $\varphi$.
\epf

%%%%%%%%%%%%%%%%%%%%%%%%%%%%%%%%%%%%

The theorem is certainly well known under the \emph{much stronger} assumption that $\varphi$
is \emph{skew} with respect to a non-degenerate form 
$\omega$, that is, $\varphi$ is contained in
in the Lie algebra $\gsp(\omega) \subset \End V$ of the symplectic group $\Sp(\omega)$.
Indeed, we recall:

\bp Let $(V,\o )$ be a symplectic vector space,  
and ${\mathfrak D}  \subset \gsp(\omega)$ a 
Lie subalgebra which consists of  
nilpotent endomorphisms. Then
there exists a $\mathfrak D$-invariant Lagrangian 
subspace of $V$.
\ep
\pf  Since ${\mathfrak D}$ is nilpotent,  it is contained in a maximal triangular  subalgebra of $\gsp(\omega)$. By the conjugacy theorem, every such subalgebra stabilizes a Lagrangian subspace of $V$.
\epf 
 
\subsection{Symplectic quadratic algebras of endomorphisms}
\label{sect:abqsa}
A Lie subalgebra $\gn$ of $\End(V)$ will be called \emph{quadratic}
if for all $\varphi \in \gn$, $\varphi^2 = 0$. In such $\gn$, all elements $\varphi, \psi  \in \gn$ anti-commute,
that is  $\varphi \, \psi = - \psi \, \varphi$. By Engel's 
theorem, the Lie algebra $\gn$ is nilpotent. Notice 
that a quadratic subalgebra $\gn$ is  \emph{abelian}  
if and only if 
$\varphi \, \psi = 0$, for all $\varphi, \psi  \in \gn$. 
A quadratic abelian algebra $\gn$
is \emph{symplectic} for $\omega$ in the sense of Definition \ref{def:symplecticalgs} if and only if  
\begin{equation} \label{eq:im_isotropic}
\omega( \varphi u, \varphi v ) = 0 \; , \end{equation} 
for all $\varphi \in \gn$ and all $u,v \in V$. 
That is,  a quadratic abelian subalgebra $\gn$ is symplectic if 
the image of every $\varphi \in  \gn$ is 
isotropic. 
The main motivation to study symplectic abelian quadratic algebras is the fact that they arise from central reductions of  three-step nilpotent symplectic Lie algebras to abelian Lie algebras (compare Section \ref{sect:reduction}).\\
% and the proof of Proposition \ref{prop:red3to2}). \\

In this section we prove that in low dimensions every abelian symplectic quadratic algebra admits invariant Lagrangian subspaces.  However, quite to the contrary, we also construct an abelian quadratic symplectic algebra (acting on a six-dimensional vector space), which does not have an invariant Lagrangian subspace, see Section \ref{sect:q6}. 

\subsubsection{Invariant Lagrangian subspaces in dimensions less than six}
We first remark that in low dimensions Theorem \ref{thm:Lagrange1} holds for quadratic endomorphisms even without the additional assumption that the condition \eqref{eq:im_isotropic} is satisfied:

\bl \label{lem:philow}
Let $(V,\omega)$ be a symplectic vector space of dimension 
less than or equal to four, and let $\varphi \in \End(V)$ satisfy $\varphi^2=0$.
Then there exists a $\varphi$-invariant Lagrangian subspace $\ga$ in $V$.
\el 
\pf If $\im \varphi$ is isotropic for $\omega$, every Lagrangian subspace
$\ga$ which contains  $\im \varphi$ is invariant. Otherwise, we may assume
that $\dim V = 4$ and that $\im \varphi$ is a non-degenerate subspace. 
Let $U =( \im \varphi)^{\perp_{\omega}}$ be the orthogonal complement,
and let $0 \neq u \in U$. Then $\ga = \langle u, \varphi u \rangle$ is a 
$\varphi$-invariant Lagrangian subspace. 
\epf 

The following shows that symplectic abelian 
quadratic subalgebras in low dimensions
always admit invariant Lagrangian subspaces: 

\bp \label{prop:quadraticlow}
Let $(V,\omega)$ be a symplectic vector space of dimension
less than or equal to four. Then every abelian quadratic symplectic subalgebra $\gq \subset \End(V)$ has an invariant Lagrangian subspace in $(V,\omega)$.
\ep
\pf  Let us consider $\im \gq = \gq V$. If $\im \gq$ is one-dimensional then every Lagrangian subspace $\ga$ of $(V,\omega)$ which contains $ \im \gq$ is invariant. In particular, if $\dim V =2$ then there exists an invariant Lagrangian subspace. 

Assume that $V$ is four-dimensional and has no invariant Lagrangian subspace for $\gq$. Because $\gq$ is abelian, $\gq^2=0$. Therefore, every subspace of $\im \gq$ is annihilated by $\gq$, and, in particular, it must be invariant. It follows that $\im \gq$ must be two-dimensional and $\omega$ is  non-degenerate on $\im \gq$. Since the image of every $\varphi \in \gq$ is isotropic, this also implies that $\im \varphi$ is one-dimensional, for all $\varphi \neq 0$.  

Now choose $u_{1} \in V$, $\varphi_{1} \in \gq$ with $\varphi_{1} u_{1} \neq 0$, $0 \neq u_{2} \in \im \gq^{\perp_{\omega}} \cap \ker \varphi_{1}$. 
Then $\ker \varphi_{1} = \langle u_{2}, \im \gq \rangle$. Moreover,
there exists $\varphi_{2} \in 
\gq$ such that $ \im \gq = \langle \varphi_{1} u_{1}, \varphi_{2} u_{2} \rangle$. 
(Otherwise, $\ga = \langle \varphi_{1} u_{1}, u_{2} \rangle$ would be a Lagrangian subspace, which is invariant for all $\varphi_{2} \in \gq$.)
% Therefore, 
%$ \im \gq = \langle \varphi_{1} u_{1}, \varphi_{2} u_{2} \rangle$, 
%for some $\varphi_{2} \in \gq$. 

Clearly, we may also assume that $u_{1} \in \ker \varphi_{2}$, unless
$\ker \varphi_{1} = \ker \varphi_{2}$. The latter case is not possible, since
$\varphi_{2} u_{2} \neq 0$.  
Since $(\varphi_{1} + \varphi_{2}) \, u_{1} = \varphi_{1} u_{1}$ and 
$(\varphi_{1} + \varphi_{2}) \, u_{2} = \varphi_{2} u_{2}$ are linearly independent,
the element $\varphi_{1} + \varphi_{2} \in \gq$ has rank two, which is a contradiction. %  since $\omega$ is non-degenerate on $\im \gq$. 
\epf

\subsubsection{A quadratic symplectic algebra without invariant Lagrangian subspace}  \label{sect:q6}
\noindent
Consider a six-dimensional 
vector space $V_{6}$  with basis  
$$\{ u_{1},  u_2, v_{1}, v_{2}, w_{1}, w_{2} \}$$  and 
define subspaces 
%\begin{eqnarray} N =  \langle  x, y \rangle \; ,  \;  U =  \langle  u_{1},  u_{2} \rangle \; , \;
%V =  \langle  v_{1},  v_{2} \rangle \\
% W =  \langle  w_{1},  w_{2} \rangle \; , \; 
%  Z =   \langle  z_{1},  z_{2} \rangle \end{eqnarray} 
%so that 
%$ V_{10 }   = N \oplus U \oplus V \oplus W \oplus Z $.
$U =  \langle  u_{1},  u_{2} \rangle$, 
$V =  \langle  v_{1},  v_{2} \rangle$, 
$W =  \langle  w_{1},  w_{2} \rangle$,  
so that 
\begin{equation} V_{6}   = V \oplus U \oplus W \; . \label{eq:V6}
\end{equation}

We now choose a symplectic form $\omega = \omega_S$
on $V_6$ such that the decomposition \eqref{eq:V6} is
isotropic (that is, $V$, $W$ are isotropic and $U =  (V \oplus W)^{\perp_{\omega}}$).
%\begin{enumerate}
%\item 
%$\omega$ is non-degenerate on $U$, 
%\item the decomposition $V \oplus W$ is isotropic, and 
%\item $U \perp_\omega (V \oplus W)$ is orthogonal. 
%\end{enumerate} 

To $\omega$ we associate the non-singular $2 \times2$-matrix   
\begin{equation} \label{eq:Sofomega}
S = \left(\begin{matrix} \omega(v_1, w_1) & \omega(v_1, w_2) \\   \omega(v_2, w_1) & \omega(v_2, w_2)
\end{matrix}\right) \; . 
\end{equation}
Note that, conversely, to every such $S$ there exists $\omega = \omega_S$ 
such that  \eqref{eq:V6} is an
isotropic decomposition 
and \eqref{eq:Sofomega} is satisfied.\\

Now define linear operators $X, Y \in \End(V_6)$ by declaring
\begin{equation} \label{eq:defq6}
X u_i = v_i ,\,  Y u_i = w_i, \;   X v_i = X w_i = Y v_i = Y w_i = 0\, , \; \, i \in 1,2\, . 
\end{equation}
% \hspace{1cm}

\bl The linear span $\gq_6 \subset \End(V_6)$ of $X$ and $Y$ is a two-\-dimen\-sio\-nal abelian quadratic algebra 
of endomorphism of $V_6$. The algebra $\gq_6$ 
is symplectic with respect to $\omega_S$ if and
only if $S$ is symmetric. 
\el
\pf Since $X Y = Y X = X^2 = Y^2 = 0$, $\gq  = \gq_6=  \langle X, Y \rangle$ is indeed an abelian quadratic subalgebra of endomorphisms. 
Now let $\varphi = \alpha X + \beta Y \in \gq$. Then
\begin{equation*}  \begin{split} 
\beta_\varphi( u_i , u_j)   = & 
 \; \,  \alpha^2 \omega( X u_i, X u_j)   + \alpha \beta  (\omega( X u_i, Y u_j) + \omega( Y u_i,  X v_j)) \\ & + \beta^2  \omega( Y u_i, Y u_j) \\
 = &  \; \,  \alpha \beta \,  (\omega( v_i, w_j) + \omega( w_i,  v_j)) \; .
\end{split} 
\end{equation*}
Hence, $\beta_\varphi = 0$, for all $\varphi \in \gq$, if and only 
if $\omega(v_1, w_2) = \omega(v_2, w_1)$. That is,
$\gq$ is symplectic if and only if $S$ is symmetric.
\epf

The following proposition shows that there do exist quadratic 
symplectic algebras which do not admit a Lagrangian invariant
subspace: 

\bp \label{prop:q6Lagrangian}
Let $S$ be a symmetric non-singular $2 \times 2$-matrix. Then the symplectic quadratic algebra $\gq_6 = \langle X, Y \rangle \subset \End(V_6)$ admits an invariant Lagrangian subspace 
$\ga$ (that is,  $\gq_6 \, \ga \subseteq \ga$) in the symplectic vector space $(V_6, \omega_S)$ if and only if $\det S < 0$. 
\ep 
\pf Assume that $\ga$ is a Lagrangian (that is, $\ga$ is three-dimensional and isotropic) subspace of $(V_{6}, \omega_{S})$, which satisfies $\gq \ga \subset \ga$. Observe that $\im \gq = 
\gq  V_6 = V \oplus W$ is a non-degenerate subspace with respect to $\omega = \omega_S$. Therefore, $ \ga \cap \im \gq$ is at most two-dimensional. We infer that there exists $a = u + w \in \ga$, such that $0 \neq u = \alpha u_{1} + \beta u_{2} \in U$ and $w \in V \oplus W$.
Then $ \gq a \subset \ga$ contains $X u =  \alpha X u_{1} + \beta X u_{2} = \alpha v_{1} + \beta v_{2}$ and $Y u = \alpha w_{1} + \beta w_{2}$. We compute $$ \omega(X u, Y u) =  \alpha^2 \omega(v_1, w_1) + 2 \alpha \beta \omega(v_2, w_1) + \beta^2 \omega(v_2, w_2) \; . $$
We may view the expression $\omega(X u, Y u)$ as a quadratic form in the variables $\alpha, \beta$. Hence, if the 
matrix $S$, as defined in \eqref{eq:Sofomega} is positive (or negative) definite, we conclude that $\omega(X u, Y u) \neq 0$.
This is a contradiction to the assumption that $\ga$ is isotropic.
Therefore, $S$ is not definite.

Conversely, assume that $S$ is not definite. Then there exists
$0 \neq u =  \alpha  u_{1} + \beta  u_{2} \in U$ with
 $\omega(X u, Y u) = 0$. Therefore, $\ga = \langle u, X u, Yu \rangle$ is a $\gq$- invariant Lagrangian ideal.
\epf

\subsection{Application to symplectic Lie algebras}

\bt  \label{thm:onedimabelian_reduction} 
Every nilpotent symplectic Lie algebra $(\gg ,\o )$ which admits a 
central element $H$ such that the reduction 
$\bar \gg = H^\perp/\spanof{H}$ is 
abelian has a Lagrangian ideal. 
\et 

\pf  Let $\varphi$ be the derivation 
of $\bar \gg$ induced by an element $\xi \in \gg \setminus H^\perp$,
$\omega(\xi,H) =1$, and 
$\l = \omega(\xi, ad_\xi|_{\bar \gg} \cdot )\in \bar \gg^*$. 
Using Proposition \ref{prop:red_ox}, we have that $(\gg, \omega)$
is the symplectic oxidation of $(\bar \gg, \bar \omega)$ for the
data $\varphi$, $\lambda$. In particular, the Lie product of
$\gg$ is described with respect to an isotropic decomposition 
of the form \eqref{eq:lie0} by equations  \eqref{eq:lie1},  \eqref{eq:lie2}.
Here,   the extension two-cocycle
$\alpha$ appearing in  \eqref{eq:lie1}  satisfies $\alpha = \omega_{\varphi}$, according to Proposition \ref{prop:symplecticox}.
Moreover, by Proposition \ref{prop:Lieox}, also  \eqref{eq:cobound_alphaphi}, that 
is, $\beta = \alpha_\varphi= \omega_{\varphi,\varphi}= -\partial \l$ is satisfied. Since
$\bar \gg$ is abelian, $\partial \lambda = 0$. In other words, \eqref{betaaEqu} 
is satisfied for $\varphi$. Since \eqref{betaaEqu} holds, according to Theorem \ref{thm:Lagrange1}, the vector space
$\bar \gg$ has a 
$\varphi$-invariant subspace, which is Lagrangian
for the induced symplectic form $\bar \omega$. 
This subspace is,  of course, 
also an ideal $\bar \ga$ of $\bar \gg$, since $\bar \gg$ is abelian.
The preimage $\ga$ of $\bar \ga$ under the projection
$H^\perp \ra \bar \gg$ is then a Lagrangian ideal of $\gg$.  
\epf 

\begin{remark} As follows from Proposition \ref{prop:symplecticox}, for every 
endomorphism $\varphi$ of a symplectic vector space $(V, \bar  \omega)$, which satisfies \eqref{betaaEqu} for $\bar \omega$, there exist symplectic oxidations $(\gg, \omega) = ( \bar \gg_{\varphi, \lambda}, \omega)$, for every $\lambda \in V^*$. This construction thus provides a rich class of examples of symplectic Lie algebras $(\gg,\o)$, which have a Lagrangian ideal. 
\end{remark}
Further applications can be found in the proof of Theorem \ref{thm:red38}.

\section{Isotropic ideals in nilpotent symplectic Lie algebras}

In the first part of this section we derive basic orthogonality relations of the ideals in the central series of a 
nilpotent symplectic Lie algebra $(\gg, \o)$.
These are used  to detect lower bounds for the 
symplectic rank of $(\gg, \o)$ in terms of the nilpotency
class of $\gg$. In particular, filiform symplectic Lie 
algebras are shown to be of maximal symplectic rank.
More precisely, we show that filiform symplectic Lie algebras 
admit a unique and characteristic Lagrangian ideal. 
This allows to classify filiform symplectic Lie algebras 
via Lagrangian extension classes of certain flat Lie algebras. 

\subsection{Existence of isotropic ideals}
We have the following orthogonality relations 
in every symplectic Lie algebra $(\gg, \omega)$: 
\bl \label{lem:basic_orthogonality} 
Let $(\gg, \omega)$ be a symplectic Lie algebra. 
Then, for all $i \geq 0$,  $$ C^{i} \gg \subseteq (C_{i} \gg)^{\perp_{\o}} \; . $$
In particular, $[\gg, \gg] \subseteq \Z(\gg)^{\perp_{\o}}$. 
\el 
\pf The proof is by induction. Assume that the statement is 
true for all $i \leq \ell-1$. For the induction step, consider $z \in C_{\ell} \, \gg$ and 
$w = [u, v] \in C^\ell \gg$, where $v \in  C^{\ell-1} \gg$.
Thus $\omega(w,z) = \omega( [u, v],z) = \omega([u,z], v) + 
\omega([z,v],u)$. Note that $ [u, z] \in C_{\ell-1} \, \gg$. Thus
by induction the first summand is zero. On the other hand,
by \eqref{eq:cs2}, $[z,v] \in C_{\ell -\ell + 1 -1}\, \gg 
 = C_{0} \gg = \{ 0 \}$. Therefore, $\omega(w,z) = 0$.
 \epf 

As a consequence, nilpotent symplectic Lie algebras
admit non-trivial isotropic ideals: 
\bp \label{prop:isotropic_ideals} Let $(\gg, \omega)$ be 
a symplectic Lie algebra which is of nilpotency class $k$. 
Then the following hold:
\begin{enumerate}
\item For all $i,j$,  with $i+j \geq k$, $C^{i} \gg \perp_{\omega} C^{j} \gg$.
\item For all  $i$, with $2i \geq k$, $C^{i} \gg$ is an
isotropic ideal. 
\item If $k= 2 \ell$ is even then $(\gg, \omega)$ has an isotropic
ideal of dimension at least $\ell$.
\item If $k =  2 \ell -1$ is odd then $(\gg, \omega)$ has an isotropic
ideal of dimension at least $\ell$.
\end{enumerate}
\ep 
\pf 
Recall from $\eqref{eq:cs3}$ that $C^{i} \gg \subseteq C_{k-i} \gg$ if $\gg$ has 
nilpotency class $k$.
Assuming  $i +j \geq k$, we have $k-i \leq j$. Hence, $C_{k-i} \gg \subseteq C_{j} \gg$.
On the other hand, by Lemma \ref{lem:basic_orthogonality}, $C_{j} \gg 
\subseteq (C^{j} \gg)^{\perp_{\o}}$. Therefore, $C^{i} \gg \subseteq (C^{j} \gg)^{\perp_{\o}}$.
%Recall that $C^{i} \gg \subseteq C_{k-i} \gg$. 
%Since $2i \geq k$, $k-i \leq i$. Hence, $C_{k-i} \gg \subseteq C_{i} \gg$.
%On the other hand, by Lemma \ref{lem:basic_orthogonality}, $C_{i} \gg 
%\subseteq (C^{i} \gg)^{\perp_{\o}}$. Therefore, $C^{i} \gg \subseteq (C^{i} \gg)^{\perp_{\o}}$.
This shows 1) and 2). 
Observe that for any $i \leq k$, $\dim C^{i} \gg \geq k-i$. 
If $k= 2\ell$ is even then $C^\ell \gg$ is isotropic and $\dim C^{\ell} \gg \geq \ell$,
which shows 3).

Now assume that $k =  2 \ell -1$.
 By 2), $C^{\ell} \gg$
 is an isotropic ideal of dimension at least $k - \ell  =  \ell -1 $.
 Next observe that, by 1), $C^{\ell-1} \gg \perp_{\omega} C^{\ell} \gg$. 
 Now choose an ideal $\gj$ of $\gg$, $C^{\ell} \gg \subset \gj \subseteq C^{\ell-1} \gg$ such that $C^{\ell} \gg$ has codimension one in $\gj$. Then $\gj$ is isotropic 
of dimension at least $\ell$. 
\epf 

In the two-step nilpotent case we thus have:

\bt \label{thm:2step}
Let $(\gg ,\o )$ be a two-step nilpotent 
symplectic Lie algebra. Then the following hold:
\begin{enumerate}
\item The ideal $C^1\gg =  [\gg ,\gg ]$ is isotropic.
\item $(\gg,\o)$ has a Lagrangian ideal.
\end{enumerate}
\et 

\pf For $u, v\in \gg$ and $z \in Z(\gg)$, we have 
$\o ([u,v],z)= - \o ([v,z],u) -  \o ([z,u],v)=0$.  
This implies that $[\gg ,\gg ]$ is isotropic, since 
$[\gg ,\gg ]\subset Z(\gg)$. Any maximally isotropic subspace 
$\ga$ which
contains  $[\gg ,\gg ]$ is a Lagrangian ideal.
\epf 

In the following subsection we deduce the existence of Lagrangian ideals
in filiform symplectic Lie algebras. Here we conclude with 
another application of Proposition \ref{prop:isotropic_ideals}, which 
concerns the existence of Lagrangian ideals in low dimensions: 

\bt \label{thm:dimlt8}
 Let $(\gg, \omega)$ be a nilpotent symplectic Lie algebra
of dimension less than or equal to six. Then $(\gg, \omega)$ 
has a Lagrangian ideal.
\et 
\pf Assume first that $\gg$ is six-dimensional. If $\gg$ is filiform (that is, of nilpotency class five) then $C^2 \gg$ is
a Lagrangian ideal by Corollary \ref{thm:ff_Lagrangian}, which will
be proved below.  
% Otherwise $\gg$ is (at most) four-step nilpotent. 

We now analyze the case that $\gg$ is of class four: 
%Then $C^2 \gg$ is an isotropic ideal by Proposition \ref{prop:isotropic_ideals}. Note if the center $C_{1} \gg$ is two-dimensional then $\gj = C^2 \gg + C_{1} \gg$ is a three-dimensional isotropic, hence Lagrangian,  ideal. Therefore, we may assume that the center of $\gg$ is one-dimensional. 
We may  centrally reduce $(\gg, \omega)$ to a
symplectic algebra  $(\bar \gg, \bar \omega)$ of dimension four, which is at 
most three-step nilpotent. If $\bar \gg$ is of class three then it is filiform and
therefore $(\bar \gg, \bar \omega)$ has a characteristic Lagrangian ideal.
It follows from Proposition \ref{prop:isotropic_lifting} that this ideal lifts to
a Lagrangian ideal in $(\gg, \omega)$. So we may assume that $\bar \gg$ 
is at most two step. If $\bar \gg$ is of class two then it has a two-dimensional center which is a characteristic Lagrangian ideal in $(\gg, \omega)$. (Indeed, $\bar \gg$ is 
the product of a Heisenberg algebra with a one-dimensional algebra.)
So this case can be excluded as well. Finally, if the reduction $\bar \gg$ is abelian 
then the existence of a Lagrangian ideal is ensured by Theorem \ref{thm:onedimabelian_reduction}. Hence, $(\gg, \omega)$ has a Lagrangian ideal. 

It remains to consider the case that $\gg$ is three-step nilpotent. 
This will be postponed to the proof of Theorem \ref{thm:red38}.
\epf 

\subsection{Filiform symplectic Lie algebras}
An important class of nilpotent Lie algebras are {filiform} nilpotent
Lie algebras, cf.\ \cite[Chapter 2, \S 6.2]{LGAIII}. 
Recall that a nilpotent Lie algebra $\gg$ is called  \emph{filiform} 
if its nilpotency class is maximal relative to its dimension $n = \dim \gg$. 
This is the case if and only if $\gg$ is of nilpotency 
class $k= n-1$. In a filiform Lie algebra,  
$\dim C^{i} \gg = n -i -1$, $i \geq 1$, and $C_{i}\,  \gg = C^{n-i-1} \gg$. Filiform symplectic Lie algebras
have been studied intensely, see e.g.\ \cite{GJK,Million}. 
We show now that filiform symplectic Lie algebras admit a {\em unique} Lagrangian
ideal:  

\bt \label{thm:ff_Lagrangian} Let $(\gg, \omega)$ be a filiform nilpotent symplectic Lie algebra
of dimension $n = 2 \ell$. Then $C^{\ell-1} \gg$ is a Lagrangian ideal. 
Moreover,  $C^{\ell-1} \gg$  is the only Lagrangian ideal in $(\gg, \omega)$. 
\et
\pf Since $\gg$ is filiform, it is of nilpotency class $k= 2 \ell -1$. By Proposition 
\ref{prop:isotropic_ideals}, there exists an isotropic ideal $\gj \subseteq C^{\ell-1} \gg$ containing $C^{\ell} \gg$ of codimension one.  Since $\gg$ is filiform, 
$\gj = C^{\ell-1} \gg$ and $\dim C^{\ell-1} \gg = \ell$. Therefore, $C^{\ell-1} \gg$ is
a Lagrangian ideal. 

To show uniqueness, we argue as follows. Let $\ga \subseteq \gg$ be a Lagrangian ideal
in $(\gg, \omega)$ and consider $\gh = \gg/ \ga$. Since $\dim \gh = \ell$, we have 
$C^{\ell-1} \gh = \{0 \}$. This implies that $C^{\ell-1} \gg \subseteq \ga$. Since
$\dim  C^{\ell-1} \gg = \ell = \dim \ga$, we conclude that $\ga = C^{\ell-1} \gg$. 
\epf 

Observe that every quotient $\gh$ of a filiform Lie algebra $\gg$ must be filiform
as well. In view of the correspondence theorems in Section \ref{sect:Lagrange_ext}, 
the previous theorem shows that filiform symplectic Lie algebras arise as 
Lagrangian extensions of certain flat Lie algebras $(\gh, \nabla)$, where
$\gh$ is filiform. Ignoring the precise conditions on $\nabla$, $\alpha$, which are 
required in addition to the assumption that $\gh$ is filiform, we 
call a Lagrangian extension triple $(\gh,\nabla, \alpha)$ (as defined in Theorem 
\ref{thm:ext_cocycle}) a filiform triple if the extension  $F(\gh,\nabla, \alpha)$ is
filiform. The uniqueness of the Lagrangian ideal in $\gg$ has the following
remarkable consequence: 

% \bc  The correspondence which associates to 
\bc \label{cor:filiform_corresp}
The correspondence which associates to a filiform symplectic Lie algebra  the  extension triple $(\gh,\nabla, [\alpha])$, which is defined with respect to the Lagrangian ideal 
 $C^{\ell-1} \gg$ of $(\gg, \omega)$, 
induces a bijection between isomorphism classes of filiform symplectic Lie algebras and isomorphism classes of flat Lie algebras with filiform symplectic extension cohomology class. 
\ec
\pf  Since every symplectic automorphism $\Phi: (\gg, \omega)  \ra  (\gg', \omega')$
maps $\ga = C^{\ell-1} \gg$ to $\ga' = C^{\ell-1} \gg'$, $\Phi$ is an isomorphisms
of pairs $(\gg, \ga) \ra (\gg', \ga')$, as well. Therefore, the corollary follows from Corollary
\ref{cor:Lagrangian_corresp}. 
\epf

\section{Three-step nilpotent symplectic Lie algebras} 
\label{sect:threestep}
In this section,  we present a ten-dimensional example of a three-step nilpotent symplectic Lie algebra which has less than maximal rank. In addition, we show that every three-step nilpotent symplectic Lie algebra of smaller dimension than ten has 
a Lagrangian ideal. The proofs of both results depend heavily on the methods developed in Section \ref{sect:abqsa}.

\subsection{A three-step nilpotent symplectic Lie algebra without
Lagrangian ideal} \label{3-stepcounterEx}
%
%We define a nilpotent Lie algebra $\gg_{10,3}$ of nilpotency 
%class three as follows: 
%

\paragraph{The Lie algebra $\gg_{10}$.} Consider a ten-dimensional vector space $ V_{10 }$  with basis  
$$\{ x, y, u_{1}, u_{2}, v_{1}, v_{2}, w_{1}, w_{2}, z_{1}, z_{2} \}$$  and 
define subspaces 
%\begin{eqnarray} N =  \langle  x, y \rangle \; ,  \;  U =  \langle  u_{1},  u_{2} \rangle \; , \;
%V =  \langle  v_{1},  v_{2} \rangle \\
% W =  \langle  w_{1},  w_{2} \rangle \; , \; 
%  Z =   \langle  z_{1},  z_{2} \rangle \end{eqnarray} 
%so that 
%$ V_{10 }   = N \oplus U \oplus V \oplus W \oplus Z $.
$N =  \langle  x, y \rangle$,  $U =  \langle  u_{1},  u_{2} \rangle$, 
$V =  \langle  v_{1},  v_{2} \rangle$, 
$W =  \langle  w_{1},  w_{2} \rangle$, 
$Z =   \langle  z_{1},  z_{2} \rangle$,  
so that 
$$ V_{10 }   = N \oplus U \oplus V \oplus W \oplus Z  \; . $$
We declare the non-trivial brackets 
of the basis vectors as 
\begin{eqnarray}  \label{eq:gt1}
 [ x, u_{i} ] = v_{i} \; , \; [y, u_{i}] =  w_{i} \,  ,   \\  \label{eq:gt2}
   { [ u_{i}, v_{i} ]} =  - z_{2} \; , \;   [ u_{i}, w_{i} ] =  z_{1} \; \;   ,  \;  \;   i \in {1,2}  . 
 \end{eqnarray}  
Since the Jacobi identity holds for the above brackets, 
linear extension defines 
a Lie Algebra $\gg_{10} = (V_{10}, [ \, , \,  ])$. 
The descending central series of $\gg_{10}$ is 
\begin{eqnarray}  C^1 \gg_{10} & =  & V  \oplus W \oplus Z \, ,\\
  C^2 \gg_{10} & = &  Z  \; ,\\
  C^3 \gg_{10} & = &  \{ 0 \} \; . 
\end{eqnarray}  
Moreover,  $C_1\,  \gg_{10} = Z( \gg_{10}) = Z$ and 
$C_2 \, \gg_{10} = \{ v \in \gg_{10} \mid  [ \gg_{10}, v ] \subseteq Z( \gg_{10}) \} = C^1 \gg_{10}$.  In particular, $\gg_{10}$ is a three-step nilpotent Lie algebra. \\

Note that $ \ga_{\mathrm{m}} = N \oplus C^1 \gg_{10}$
is a maximal abelian ideal of $ \gg_{10}$, which contains 
the commutator subalgebra $C^1 \gg_{10}$.
%is a six-dimensional
%abelian ideal which is contained in the 

\bl  \label{lem:abideal}
Every abelian ideal of $\gg_{10}$ is contained
in $\ga_{\mathrm{m}}$. In particular,
the maximal dimension of an abelian ideal in $\gg_{10}$ is eight.
Every (not necessarily abelian) ideal of dimension at least five contains $C^2 \gg_{10}$.
\el 
\pf 
Let $\ga$ be an abelian ideal.  
Let $a \in \ga$. We write $a = \alpha\,  u_{1} + \beta\,  u_{2} + w$, 
with $w \in \ga_{\mathrm{m}}$. Then $[x, a] = \alpha v_{1} + \beta v_{2}  \in \ga$. Therefore, $[a, [x, a]] =  - (\alpha^2 + \beta^2)  z_{2} = 0$, since $\ga$ is abelian.  This implies $a= w$. Hence, $\ga \subseteq \ga_{\mathrm{m}}$.

Suppose now that $\gj$ is an ideal, which satisfies 
 $\gj \cap  C^2 \gg_{10} = \{ 0 \}$.
Then $\gj \cap C_2 \, \gg_{10} $
is contained in the center $Z( \gg_{10})$ of $\gg_{10}$. Therefore,
$$  \gj \cap C^1 \, \gg_{10} = \gj \cap C_2 \, \gg_{10} \, \subseteq \; \gj \cap Z( \gg_{10}) = \{ 0 \} \, . $$  So $[\gg_{10}, \gj] = \{ 0 \}$, which means that $\gj$ is central. Hence, $\gj$ is contained in $C^2 \gg_{10}$ and, therefore, $\gj = \{0 \}$.  

Similarly, suppose that $\gj$ is an ideal such that  
$\gj \cap  C^2 \gg_{10}$ is one-dimensional. 
Without loss of generality (cf.\ Lemma \ref{lem:gl2action} below), 
we can assume that 
$\gj \cap  C^2 \gg_{10}$ contains $z_1$. Therefore, 
$ \gj \cap C^1 \gg_{10}$ is contained in $$\{ v \mid [\gg_{10}, v ] \subseteq \langle z_{1} \rangle \} = W \oplus  \langle z_1 \rangle . $$
This implies $[\gg_{10}, \gj] \subseteq W  \oplus  \langle z_1 \rangle$. Hence, $\gj$ is contained in the four-dimensional (abelian) 
ideal $  \langle y \rangle \oplus W \oplus  \langle z_1 \rangle$.
In particular, $\gj$ is at most four-dimensional.
\epf 

We briefly observe that there is a natural action of $\GL(2,\bR)$
by automorphisms on the Lie algebra $\gg_{10}$:

\bl  \label{lem:gl2action} 
Let $(a_{ij}) \in \GL(2,\bR)$. Then the definitions
\begin{equation*}  \begin{split}
 A x  & =  a_{11}  x + a_{21} y \, , \;    A y =  a_{12}  x + a_{22} y \, , \;  A  u_{i} = u_{i},  \\
A v_{i} & = a_{11}  v_{i} + a_{21} w_{i} \, , \;  A w_{i} = a_{12}  v_{i} + a_{22} w_{i} \\
A z_{1} & =  a_{22} z_{1} - a_{12} z_{2}   \, , \; A z_{2}  =  -a_{21} z_{1} + a_{11} z_{2} 
\end{split}
\end{equation*}
define an automorphism $A$ of  the Lie algebra $\gg_{10}$.
\el 

\paragraph{A symplectic form for $\gg_{10}$.}
%Let $S = (s_{ij})$ be a non-singular  $2 \times 2$ matrix. 
%We define a  family $\omega_{S} \in \bigwedge^2 \gg_{10}^* $ of non-degenerate two-forms on $\gg_{10}$ parametrized by $S$:
% with respect to the  above basis: 
We define a non-degenerate two-form $\omega_{o}$ on the vector space $V_{10}$: 
\begin{equation*}  \begin{split}
\omega_{o}   \, = \,  & x^* \wedge z_{1}^* \, + \, y^*  \wedge z_2^* \, + \, 
u_{1}^* \wedge u_{2}^* \,  +  \, 
 v_{1}^*  \wedge w_{1}^* \, + \, % s_{12}\,  v_{1}^* \wedge w_{2}^* \,  + \, \\   &  s_{21}\,   v_{2}^*  \wedge w_{2}^* \, + \, 
v_{2}^*  \wedge w_{2}^*
\end{split}
\end{equation*}
With respect to the form $\omega_{o}$,  the decomposition
$$  V_{10} =  N \oplus (U \oplus V \oplus W) \oplus Z $$
is isotropic,  the  subspace $(V \oplus W, \omega_{o})$ is non-degenerate with isotropic decomposition, and the sum 
$$    U \oplus (V \oplus W)  = U \perp_{\omega_{o}} (V \oplus W) $$
is orthogonal.

\bl The form $\omega_{o}$ is symplectic
for the Lie algebra $ \gg_{10}$.  In particular,  $(\gg_{10}, \omega_{o})$ is a  symplectic three-step nilpotent Lie algebra, and the maximal abelian ideal $ \ga_{\mathrm{m}} = N \oplus C^1 \gg_{10}$ is non-degenerate in $(\gg_{10}, \o_{o})$. 

\el
\pf Clearly, from \eqref{eq:gt1} and  \eqref{eq:gt2},  $\partial x^*= \partial y^*=  \partial u_{i}^* = 0$ and 
 \begin{equation*} 
\begin{split}
 & \partial v_{i}^{*} =  - x^* \wedge u_{i}^*  \; , \;  \partial w_{i}^{*} =  - y^* \wedge u_{i}^* \;  ,  \; \partial z_{1}^* =  - \, ( u_{1}^* \wedge w_{1}^{*} +   u_{2}^* \wedge w_{2}^{*} ) \\
 &  \partial z_{2}^* =  u_{1}^* \wedge v_{1}^{*} +   u_{2}^* \wedge v_{2}^{*} \; .   \end{split}
\end{equation*} Therefore,
\begin{equation*} \begin{split} 
\partial \omega_{o} & =  -x^* \wedge \partial z_{1}^* + y^* \wedge \partial z_{2}^* 
+ \sum_{i = 1,2}  ( \partial{v_{i}^*} \wedge w_{i}^* - v_{i}^* \wedge \partial  w_{i}^*)\\ 
&  = 
  x^* \wedge u_{1}^* \wedge w_{1}^{*} - x^*  \wedge  u_{2}^* \wedge w_{2}^{*} -   y^* \wedge u_{1}^* \wedge v_{1}^{*} -   y^* \wedge   u_{2}^* \wedge v_{2}^{*} \\
 &  \quad -  \sum_{i = 1,2}  (x^* \wedge u_{i}^* \wedge w_{i}^* -  y^* \wedge u_{i}^*   \wedge v_{i}^*)  \; \\
 &   = \;  0 \; . 
 \end{split}
\end{equation*}
Hence, $(\gg_{10}, \omega_{o})$ is symplectic. 
\epf

\emph{Remark that the action $\SL(2,\bR)$ on $\gg_{10}$ which 
is defined in Lemma \ref{lem:gl2action} is by \emph{symplectic 
automorphisms} with respect to $\omega_{o}$.} \\

We recall (cf.\ Definition \ref{def:symp_rank}) that the rank of a 
symplectic Lie algebra is defined to be 
the maximal dimension of an isotropic ideal.  

\bp The symplectic Lie algebra $(\gg_{10}, \omega_{o})$ has symplectic rank four. In particular, $(\gg_{10}, \omega_{o})$ has no Lagrangian ideal. 
\ep 
\pf  Observe first that  the ideal $\gj =  \langle y \rangle \oplus W \oplus  \langle z_1 \rangle$ is isotropic and of dimension four. Therefore, the rank of 
$(\gg_{10}, \omega_{o})$ is at least four. 

%To prove that $(\gg_{10}, \omega_{o})$ has no Lagrangian (that is, five-dimensional and isotropic) ideal,  we remark that every such ideal contains $C^2 \gg_{10}$, 
%
Let us consider next the symplectic reduction of  $(\gg_{10}, \omega_{o})$ with respect to the isotropic ideal $C^2 \gg_{10}$.
The reduction consists of an abelian symplectic
Lie algebra $(\ga_{6}, \omega_{o})$, which is defined
on the vector space 
$V_{6} = (C^2 \gg)^\perp/ C^2 \gg \cong U \oplus V \oplus W$; the
symplectic structure on $V_{6}$ is given by the restriction of 
$\omega_{o}$. 

The adjoint representation  $\ad$ of  $\gg_{10}$ induces a 
two-dimensional Lie subalgebra $\gn= \ad(N)|_{V_{6}}$ % of $\End(V_{6})$
of derivations of the symplectic
Lie algebra $(\ga_{6}, \omega_{o})$. 
It is obtained by restricting  $\ad$ to
$ (C^2 \gg)^\perp/ C^2 \gg$. 
Note that $\gn$ has a basis $X= \ad(x)$, $Y = \ad(y)$, which, by \eqref{eq:gt1}, satisfies 
$$ X u_i = v_i ,\,  Y u_i = w_i, \,  X v_i = X w_i = Y v_i = Y w_i = 0 \; , \, \, i \in 1,2 \, .$$
In particular, for every element $\varphi \in \gn$, we have 
$\varphi^2 = 0$. Indeed, the subalgebra $\gn$ coincides
with the quadratic algebra of endomorphisms $\gq_6$ of 
$V_6$ as defined in Section \ref{sect:q6}, equation \eqref{eq:defq6}.

Assume now that  $\ga$ is a Lagrangian ideal in $(\gg_{10}, \omega_{o})$. Since,
by Lemma \ref{lem:abideal}, every Lagrangian ideal $\ga$
of $\gg_{10}$ contains $C^2 \gg_{10}$, 
$\ga \subset C^2 \gg_{10}^{\perp_\omega}$ 
projects to a Lagrangian subspace of 
$(V_{6}, \omega_{o})$, which is furthermore 
invariant by the Lie algebra of  operators $\gn = \gq_6$.

We observe that the matrix $S$ associated to $\omega_{o}$ in
\eqref{eq:Sofomega} is the identity matrix, and 
$\omega_o|_{V_{6}} = \o_{S}$.  
Therefore, $\gn$ is a quadratic symplectic algebra with
respect to $\omega_o|_{V_{6}}$. According to Proposition \ref{prop:q6Lagrangian}, 
$(V_{6}, \omega_{o})$ does not admit a Lagrangian
$\gn$-invariant subspace, since $\det S >0$. 
This is a contradiction to our
assumption that $\ga$ is a Lagrangian ideal in 
$(\ga_{6}, \omega_{o})$.
\epf
%%%%%%%%%%%%%%%%%%%%%%%%%%%%%%%%%%%%

Theorem \ref{thm:red38} below shows that $(\gg_{10}, \omega_{o})$ is an example for a three-step nilpotent 
symplectic Lie algebra without Lagrangian ideal, which
is of minimal dimension. 
\subsection{Three step nilpotent Lie algebras of dimension less than ten}

Here we show that every three step nilpotent symplectic Lie algebra of dimension less than ten admits a Lagrangian ideal. 
We start with the following proposition which ensures that a large class of three-step nilpotent symplectic Lie algebras admits Lagrangian ideals:

\bp \label{prop:red3to2}
Let $(\gg, \omega)$ be a three step nilpotent symplectic Lie algebra
which has a central reduction to a two-step nilpotent symplectic Lie algebra of
dimension at most four. Then $(\gg, \omega)$ admits a 
Lagrangian ideal $\ga$ which contains $C^2 \gg$. 
\ep 
\pf By assumption,  there exists a central and isotropic 
ideal $\gj$ such the symplectic reduction of $(\gg, \omega)$
with respect to $\gj$  is a two-step nilpotent symplectic Lie algebra 
$(\bar \gg, \omega)$\footnote{In this proof, we deviate from our previous convention and denote the symplectic form on the reduction also with $\omega$.}, and $\bar \gg$  is of dimension at most four. 
We may assume without loss of generality that $\gj$ contains
$C^2 \gg$, since $C^2 \gg$ is isotropic and also orthogonal to $\gj$. (Indeed, $C^2 \gg$ contained in $\Z(\gg)$ and it is orthogonal to $\Z(\gg)$, by Lemma  \ref{lem:basic_orthogonality}.)

We may choose an  isotropic decomposition 
$$\gg = N \oplus W \oplus \gj \; ,  $$ 
where $ \gj^{\perp_{\omega}} = W \oplus  \gj$ is an ideal in $\gg$. (Recall that 
$ \gj^{\perp_{\omega}}$   contains $C^1 \gg = [\gg, \gg ]$.) 
Since $ \gj^{\perp_{\omega}}$ is an ideal, the adjoint representation of $N$ 
induces a subspace $\gq= \ad(N)|_{\bar \gg}$ of derivations of 
$$ \bar \gg =  \gj^{\perp_{\omega}}/  \, \gj \; . $$
Clearly (see Proposition \ref{prop:isotropic_lifting}), if $(\bar \gg, \omega)$ admits a Lagrangian ideal $\bar \ga$ which satisfies $\gq \bar \ga \subset \bar \ga$, then the pull back 
$\ga$ of $\bar \ga$ under the natural homomorphism 
$$ \gj^{\perp_{\omega}} \rightarrow
\bar \gg$$ is a Lagrangian ideal in $(\gg, \omega)$,  
which contains $\gj$. Since $\gj$ contains $C^2 \gg$, this will prove the proposition.
We claim now that such an ideal $\bar \ga$ always exists.

Since $[N, [N, \gg]] \subset C^2 \gg$, every $\varphi \in \gq$ 
satisfies $\varphi^2 = 0$. 
%Moreover, by Proposition \ref{prop:reduction},
%every $\varphi \in \gq$ satisfies the coboundary condition \eqref{eq:cobound}. 
Note further that $\gq$ is also an
abelian Lie subalgebra of $\Der(\bar \gg)$, since 
$[ \, [N, N] , \gg] \subseteq  C^2 \gg$. Similarly, since
$ [ [N, W], W] \subseteq C^2 \gg$, we deduce that
$\im \gq$ is contained in the center of $\bar \gg$. This implies
that $\im \gq \perp_{\omega} C^1 \bar \gg$.
In particular, $\gq$ preserves the ortohogonal 
$(C^1 \bar \gg)^{\perp_{\omega}}$.

If the Lie algebra $\bar \gg$ is abelian then  $\gq$ is an abelian 
quadratic symplectic subalgebra  (see Section \ref{sect:abqsa}), 
since its elements satisfy
\eqref{eq:cobound_zero}, by Proposition \ref{prop:abelian_red}.
Therefore, in the latter case the existence of $\bar \ga$ is ensured by Proposition \ref{prop:quadraticlow}, which may be applied since the reduction $\bar \gg$ is at most four-dimensional.

Therefore, it remains the case  that  $\bar \gg$ is 
of nilpotency class two and of dimension four. Recall that then $C^1 \bar \gg$ 
is an isotropic ideal in $(\bar \gg, \omega)$. Therefore,  if $\dim C^1 \bar \gg = 2$ we can put $\bar \ga =  C^1 \bar \gg$. We thus assume now that  $\dim C^1 \bar \gg = 1$.
%Otherwise Lemma \ref{lem:qLagrange_red} implies that 
As remarked above, $\gq$ preserves the ortohogonal $(C^1 \bar \gg)^{\perp_{\omega}}$. Therefore, 
$(\bar \gg, \omega)$ reduces 
to a two-dimensional symplectic vector space $(V_{2}, \omega)$
and $\gq$, as well as $\bar \gg$, act on $V_{2}$. The pull back $\bar \ga$ 
of any $\gq$- and $\bar \gg$-invariant Lagrangian subspace in $V_{2}$ will be a
$\gq$-invariant ideal in $\bar \gg$. Since $V_{2}$ is two-dimensional, and the algebra of endomorphisms of $V_{2}$, which is generated by $\gq$ and $\bar \gg$, is nilpotent, such a subspace clearly exists. This shows that $\bar \ga$ exists if $\bar \gg$ is two-step nilpotent and of dimension four.
\epf 

We now prove:
\bt \label{thm:red38}
Let $(\gg, \omega)$ be a three step nilpotent symplectic Lie algebra of 
dimension less than or equal to eight. Then $(\gg, \omega)$ admits a 
Lagrangian ideal $\ga$ which contains $C^2 \gg$. 
\et 
\pf By Theorem \ref{thm:2step}, Lagrangian ideals always exist 
if $\gg$ is two-step nilpotent. We may thus assume 
from the beginning that $\gg$ is of class three. 

For the construction of $\ga$,  
we reduce now $(\gg, \omega)$ by the central
ideal $C^2 \gg$ to a symplectic Lie algebra 
$(\bar \gg, \omega)$.  (In a class three nilpotent symplectic Lie algebra 
$C^2 \gg$ is  isotropic by 2.\ of Proposition \ref{prop:isotropic_ideals}.)

Assume first that $\dim  C^2 \gg$ is
at least two.  Then the reduction $\bar \gg$ is at most four-dimensional, since $\dim \gg \leq 8$, and $\bar \gg$ 
is two-step nilpotent.  
In this case, the existence of $\ga$ is ensured by
Proposition \ref{prop:red3to2}.

Therefore, we can assume now that $C^2 \gg$ is one-dimensional.
This means that we may choose an isotropic
decomposition $$ \gg \, = \;  \langle \xi \rangle \oplus W \oplus C^2 \gg \; .  $$
Moreover, in the view of Theorem \ref{thm:onedimabelian_reduction},  it suffices to consider the case that the
reduction $\bar \gg$ is of class two (and not abelian). 
Let $\varphi = \ad \xi_{|_{\bar \gg}}$ be the  derivation
on $\bar \gg$ which is induced by the reduction. 
Clearly, $\varphi^2=0$.  
Similarly, we have $\varphi \, \bar \gg  \subseteq \Z(\bar \gg)$.
We need to construct a $\varphi$-invariant 
Lagrangian ideal $\bar \ga$ in $(\bar \gg, \omega)$. (As usual,
its pullback to $\gg$ is the required Lagrangian ideal $\ga$ in $(\gg,\omega)$.)

If the commutator 
$C^1 \bar \gg$ has dimension three, $\bar \ga = C^1 \bar \gg$ is such 
an invariant Lagrangian ideal, since $\dim \bar \gg = 6$. 
If this is not the case, reduction with respect to $C^1 \bar \gg$ gives 
a (now at most four-dimensional) symplectic abelian Lie algebra
$(\bar \gg', \omega) = (V, \omega)$. Note that $(\bar \gg', \omega)$ has an induced action of $\varphi$, since 
$\varphi\,  \bar \gg \subseteq \Z(\bar \gg) \subseteq (C^1 \gg)^{\perp_{\o}}$. 
Moreover, any Lagrangian subspace $\bar \ga'$ of $(V, \omega)$,
which is invariant by $\varphi$, will pull back to a $\varphi$-invariant Lagrangian 
ideal $\bar \ga$ in $(\bar \gg,  \omega)$.  The
existence of  such $\bar \ga'$ is guaranteed by Lemma \ref{lem:philow}. 

This finishes the proof.
\epf

\end{document}